\documentclass[11pt]{amsart}
\usepackage[utf8]{inputenc}
\usepackage{csquotes}
\usepackage[foot]{amsaddr}

\usepackage[dvips]{graphics,epsfig}
\usepackage{times}
\usepackage{amsmath,amsthm,amssymb}
\usepackage{array}
\usepackage{bbm}
\usepackage{xcolor}
\usepackage{verbatim}
\usepackage{url}
\def\BB{\mathcal{B}}

\def\NN{\mathcal{N}}

\def\bK{{\bf K}}

\def\bU{{\bf U}}
\def\aut{\mathop{aut}}

\def\Uscr{\mathcal U}

\def\abar{{\bf a}}
\def\bbar{\bf b}
\def\cbar{\bf c}

\def\nmdeg{\mathop{\rm nmdeg}}
\def\icl{\mathop{\rm icl}}

\def\nmdeg{\operatorname{nmdeg}}
\newcommand{\Fmb}[1]{F_{Mb}\left( #1  \right)}

\newcommand{\fcb}[1]{F_{cb}\left( #1  \right)}
\newcommand{\ficb}[1]{F_{Icb}\left( #1  \right)}

\DeclareMathOperator{\tp}{tp}
\DeclareMathOperator{\acl}{acl}

\newtheorem{theorem}{Theorem}[section]
\newtheorem{thm}[theorem]{Theorem}
\newtheorem{question}[theorem]{Question}
\newtheorem{lemma}[theorem]{Lemma}
\newtheorem{cor}[theorem]{Corollary}
\newtheorem{prop}[theorem]{Proposition}
\newtheorem{fact}[theorem]{Fact}
\newtheorem{claim}[theorem]{Claim}
\newtheorem{subclaim}[theorem]{Subclaim}
\newtheorem{conj}[theorem]{Conjecture}

\newtheorem{observation}[theorem]{Observation}
\newtheorem{notation}[theorem]{Notation}

\theoremstyle{definition}
\newtheorem{definition}[theorem]{Definition}
\newtheorem{example}[theorem]{Example}

\theoremstyle{remark}
\newtheorem{remark}[theorem]{Remark}
\newcommand{\m}{\mathbb }
\newcommand{\mc}{\mathcal }

\newcommand{\mb}{F_{Mb} }

\newcommand{\U}{\mathrm{U} }

\def\Ind{\setbox0=\hbox{$x$}\kern\wd0\hbox to 0pt{\hss$\mid$\hss} \lower.9\ht0\hbox to 0pt{\hss$\smile$\hss}\kern\wd0}
\def\Notind{\setbox0=\hbox{$x$}\kern\wd0\hbox to 0pt{\mathchardef \nn=12854\hss$\nn$\kern1.4\wd0\hss}\hbox to 0pt{\hss$\mid$\hss}\lower.9\ht0 \hbox to 0pt{\hss$\smile$\hss}\kern\wd0}
\def\ind{\mathop{\mathpalette\Ind{}}}
\def\nind{\mathop{\mathpalette\Notind{}}}
\numberwithin{equation}{section}

\title{Simple Homogeneous Structures and Indiscernible Sequence Invariants}
\author[1]{John T. Baldwin}
\author[2]{James Freitag}
\author[3]{Scott Mutchnik}
\address{Department of Mathematics, Statistics, and Computer Science, University of Illinois at Chicago }
\email{jbaldwin@uic.edu}
\email{jfreitag@uic.edu}

\email{mutchnik@uic.edu}

\thanks{JB was partially supported by Simons travel grant G3535. JF was partially supported by NSF CAREER award number 1945251. SM, who is additionally affiliated with the Institut de mathématiques de Jussieu – Paris Rive Gauche, was supported by the NSF under Grant No. DMS-2303034, and by the European Union’s Horizon 2020 research and innovation programme under the Marie Sklodowska-Curie grant agreement No 101034255.}

\date{\today}
\begin{document}

\begin{abstract}
We introduce new properties describing dependence in indiscernible sequences: $F_{ind}$ and its dual $F_{Mb}$, the definable Morley property, and $n$-resolvability. After developing these new properties, we resolve a number of open questions:

We show that the \textit{degree of nonminimality} introduced by Freitag and Moosa, which is closely related  to $F_{ind}$, may take on any positive integer value, answering a question of Freitag, Jaoui, and Moosa.

Proving a conjecture of Koponen, we show that every simple theory with quantifier elimination in a finite relational language has finite rank and is one-based. Our proofs of the supersimple case use Tomasi\'c and Wagner's results on pseudolinearity, and the forking degrees $F_{\lambda}$ of Freitag. The arguments closely rely on producing a type $q$ with $F_{Mb}(q) = \infty$. Our proof of the strictly simple case uses the notion of $n$-resolvability, which is related to $F_{Mb}$. Both $F_{Mb}$ in supersimple theories, and $n$-resolvability in simple theories, depend on the possible ranks of canonical bases, and in the setting of the Koponen conjecture, they allow us to distinguish between infinitely many indiscernible sequences over a finite set. 

We prove some variants of the \textit{simple Kim-forking conjecture}, a generalization of the stable forking conjecture to $\mathrm{NSOP}_{1}$ theories. We show a global analogue of the simple Kim-forking conjecture with infinitely many variables holds in \textit{every} $\mathrm{NSOP}_{1}$ theory, and show that Kim-forking with a realization of a type $p$ with $\mathrm{F}_{Mb}(p) < \infty$ satisfies a finite-variable version of this result. We then show, in a low $\mathrm{NSOP}_{1}$ theory or when $p$ is isolated, if $p \in S(C)$ has the definable Morley property for Kim-independence, Kim-forking with realizations of $p$ gives a nontrivial instance of the simple Kim-forking conjecture. In particular, when $F_{Mb}(p) < \infty$ and $|S^{F_{Mb}(p) + 1}(C)| < \infty$, Kim-forking with realizations of $p$ gives us a nontrivial instance of the simple Kim-forking conjecture.

We show that the quantity $F_{Mb}$, motivated in simple and $\mathrm{NSOP}_{1}$ theories by the above results, is in fact nontrivial even in stable theories. Using constructions based on the \textit{free projective planes} studied by Hyttinen and Paolini, for every $n < \infty$, we find a stable theory admitting types $p$ such that $F_{Mb}(p) = n$, and so that $F_{Mb}(p) \leq n$ for all types $p$.
\end{abstract}

\maketitle

\tableofcontents

\section{Introduction}
\numberwithin{theorem}{subsection} \setcounter{theorem}{0}
Forking, dividing, and their associated notions of independence are at the center of modern model theory from both a theoretical perspective and in applications. Many of the most important dividing lines in model theory are most easily defined by the presence or absence of certain local combinatorial properties, but many admit equivalent characterizations by properties of independence. These characterizations (which have been developed as far back as Harnik and Harrington (\cite{HarnikHarringtonForking})) often play a powerful role in results and give anyone who studies the subject a feeling that many of the prominent dividing lines are canonical - they arise inevitably from a number of different perspectives that do not immediately appear to be connected. For instance, Shelah \cite{shelah1980simple} introduced \emph{simple theories} as those in which a certain combinatorial configuration (the tree property) was absent. The class is equivalently characterized by the property that each type does not fork over a subset of its domain of size at most $|T|$ (see chapter three of  \cite{grossberg2002primer}). Kim and Pillay \cite{kim1998stability} gave another characterization, proving under a list of general assumptions on an abstract independence relation $\nind^0$ that the relation must be nonforking independence in a simple theory. In a more applied direction, there is for instance,  a large literature  whose central goal is characterizing forking dependence for realizations of a specific type (or class of types). For instance, there is extensive recent work just in the theory of differentially closed fields  \cite{freitag2023equations, freitag2017algebraic, freitag2017strong, nagloo2017algebraic}.

In a general theory, dividing of any partial type can always be detected by some indiscernible sequence in the parameters of the formulas. \emph{The central thrust of this paper is to make quantitative various aspects of these indiscernible sequences which control forking and dividing.} We do so by defining several new invariants of indiscernible sequences relative to a theory, all of which are closely related to how forking is witnessed in the theory. We believe many of these invariants are of intrinsic interest - some are already studied in the literature from a different perspective. For instance, Lemma \ref{cbmorleysequence} is a generalization of the most natural case\footnote{In the case that the outputs of system are the solutions themselves, the parameters to be identified are the coefficients, and the coefficients form a canonical base of the generic solution} of the main result of Ovchinnikov, Pillay, Pogudin, and Scanlon \cite[Theorem 3.1]{ovchinnikov2022multi},  a result about the number of independent observations needed to determine the value of the coefficients of a system given by differential equations with a fixed model. After developing basic results around our invariants and connections between them, we turn towards a number of applications. The paper targets three general areas: establishing bounds on the invariants under various tameness assumptions, the Koponen conjecture on one-basedness in simple theories, and variants of the stable forking conjecture. We describe these areas in detail next. 

\subsection{Nonminimality}

The inspiration for our invariants comes from several sources. Several of our notions are closely related to the \emph{degree of nonminimality} of a type \cite{freitag2023degree}. Given a type $p$ of Lascar rank larger than $1$, $\nmdeg {(p)}$ is defined\footnote{In Section \ref{nmdegfmb} there are careful definitions of the invariants we study.} to be the minimal number of realizations of the type $p$ such that $p$ has a nonalgebraic forking extension over the parameter set given by these realizations.

Powerful upper bounds for the degree of nonminimality were established in theories satisfying various hypotheses \cite{freitag2023degree}, which resulted in several applications \cite{freitag2022any, devilbiss2023generic}. But as the authors of \cite{freitag2023bounding} write: 
\begin{displayquote}
Concerning lower bounds on the degree of nonminimality, the situation is highly unsatisfactory. Examples of degree $>1$ seem very difficult to produce. In fact, we were unable to construct any of degree strictly greater than $2$.
\end{displayquote}
We resolve this problem, first in the setting of simple theories in Section \ref{nm>}, using a construction of Evans \cite{Ev02}. This examples will in fact be supersimple, and $\omega$-categorical. We solve the problem in Section~\ref{nm1cat} for the original setting of $\omega$-stable finite rank types using a more involved construction involving $\aleph_1$-categorical $(r,k)$-Steiner systems. For any $n \in \m N$, we build rank $2$ types whose degree of nonminimality is $n$. 

Of course, many variations on the degree of nonminimality are possible, and we introduce several new notions in Section \ref{nmdegfmb}. For instance, $F_{ind}(p) $ is defined (in a simple theory) to be the least $n$ such that there exists an infinite (nonconstant) indiscernible sequence $a_i$ in $p$, so that $\tp (a_{n+1}/a _{\leq n})$ is a forking extension of $p$.\footnote{See Definition \ref{finddef}.} As we note - it isn't hard to see that $F_{ind}(p) $ bounds the degree of nonminimality, but the two invariants can be different, even in the finite rank context, as we show in Section \ref{nm>}, though they are the same in a large class of theories, including, for instance differentially closed fields and compact complex manifolds. It is, however, interesting to note that $F_{ind} (p) $ can be bounded in terms of the degree of nonminimality and the Lascar rank of $p.$ 

In defining $F_{ind} (p)$ we quantify existentially over all indiscernible sequences in the type $p$. In a definition which evokes Kim's lemma, we consider the same definition but replacing the existential quantifier with a universal quantifier, defining $\Fmb {p}$ in Definition \ref{Mbr}. In special cases of the rank context, the notion is related to \emph{pseudolinearity} as considered for instance by Tomašić and Wagner \cite{TW03}. This invariant and the ideas around it play an essential role in Section \ref{Kop}.

\subsection{Countable categoricity and the Koponen conjecture} 
There is a long and rich tradition of studying the intersection of countable categoricity with tameness notions coming from classification theory. Before explaining the particular setting of Koponen's conjecture, which we resolve in Section~\ref{Kop}, we explain how the conjecture fits into this tradition. Lachlan \cite{lachlan1974two} conjectured that if $T$ is stable and $\omega$-categorical, then $T$ is totally transcendental and proves the weaker statement that if $T$ is superstable and $\omega$-categorical, then $T$ is totally transcendental. 
Cherlin, Harrington, and Lachlan \cite{cherlin1985aleph} later proved that superstable $\omega$-categorical theories have finite Morley rank. Buechler \cite{buechler1985invariants} gave a complete classification of $\omega$-stable $\omega$-categorical theories in terms of a simple collection of invariants. Hrushovski \cite{hrushovski1989totally} proved that any totally categorical theory can be finitely axiomatized relative to the collection of axioms saying the structure is infinite,  completing the resolution of questions about finite axiomatizability of categorical theories  pondered since the early 1960's \cite{M65}.

More recently, Palac\'in \cite{palacin2012omega} showed that an $\omega$-categorical $CM$-trivial simple theory is low and asked \cite[see the Question on Page 10]{palacin2012omega} 
if every simple $\omega$-categorical theory is low.\footnote{This also appears as \cite[Question 2, page 755]{casanovas2002local}. Palac\'in hints that without some additional assumption, there isn't strong evidence that that the question should have an affirmative answer.} Simon \cite{simon2022cat} classified $\omega$-categorical NIP structures of rank one. A version of Simon's classification result was extended to the higher rank setting by Onshous and Simon \cite{onshuus2021dependent}. Shelah \cite{shelah1971stability}  showed that any unstable NIP structure is SOP, and asked if one can show the stronger statement that the structure has an interpretable infinite linear order. Simon \cite{simon2022linear} showed this in the case that the theory is assumed to be $\omega$-categorical. Often under the (equivalent) assumption of having an oligomorphic permutation group, there is an equally extensive literature in more specific algebraic contexts. See page 7 of Cherlin's bibliography of homogeneity and related topcis \cite{cherlin2021homogeneity}.  

Koponen asks the following (\cite{Kop16}, \cite{Kop19}, \cite{Kop16slides}), conjecturing a positive answer to the second question:

\begin{question} \label{kopQ}
   Is every simple theory with quantifier elimination in a finite relational language one-based? Is every simple theory with quantifier elimination in a finite relational language supersimple of finite rank?
\end{question}

Note that a positive answer to the first question implies a positive answer to the second question. For a theory to have quantifier elimination in a finite relational language is the same thing as being $\omega$-categorical and $n$-ary for some $n < \omega$ (that is, having quantifier elimination in a language with only $n$-ary relation symbols). Assuming $T$ as in the question is not just simple, but \textit{supersimple}, we will see in Section \ref{Kop} a positive answer to Question \ref{kopQ} follows with a bit of effort from a result of Tomašić and Wagner \cite{TW03}, appearing in this paper as Fact \ref{pseudolinearity}. Again in the supersimple case, we give a second proof which goes through our development of $F_{Mb}$ via Corollary \ref{omegacategoricalsupersimple}. This uses the \textit{forking degrees} $F_{\lambda}$ (Definition \ref{flambda}), originally introduced in \cite{oberwolfachFdeg} to generalize the degree of nonminimality. Either proof supplies us with invariants that distinguish infinitely many indiscernible sequences over a finite set, thereby contradicting the assumption of quantifier elimination in a finite language.

In any case, the most difficult part of our proof of the full Koponen conjecture involves showing that a theory which is simple with quantifier elimination in a finite relational language must be supersimple. In less syntactic terms, the complicated part is showing that every $\omega$-categorical $n$-ary theory which is simple is in fact supersimple. In the finite rank case, the result of \cite{TW03} produces canonical bases (for types with fixed variables) of arbitrarily large $\mathrm{SU}$-rank--or equivalently, in the sense of Corollary \ref{rank2} and Corollary \ref{Mbpsquared}, a type $q$ with $F_{Mb}(q) = \infty$. In the general supersimple case, \cite{TW03} produce canonical bases of arbitrarily large finite $\mathrm{SU}_{p}$-rank, where $p$ is a regular type, again producing types $q$ with $F_{Mb}(q) = \infty$. So in the supersimple case, we are able to distinguish infinitely many indiscernible sequences over a finite set according to the $\mathrm{SU}$- or $\mathrm{SU}_{(p)}$-ranks of the canonical bases of their limit types. However, in the simple but not supersimple setting, $\mathrm{SU}$-rank will not be well-defined, nor will we necessarily get nontrivial ranks of the form $\mathrm{SU}_{p}$. To remedy this, in Definition \ref{grank} we introduce a new notion of rank for general simple theories, the \textit{$\mathrm{G}$-rank} of a set. Then in Theorem \ref{omegacategoricalsimplebutnotsupersimple}, we show that, in strictly simple $\omega$-categorical theories, there is an invariant distinguishing infinitely many indiscernible sequences (with terms of fixed length), given by producing canonical bases of arbitrarily large finite $\mathrm{G}$-rank.

The stable case of the Koponen conjecture is well-known: any stable theory with quantifier elimination in a finite relational language is superstable of finite rank, and in fact one-based. Cherlin and Lachlan (\cite{CL20}) show that a stable theory with quantifier elimination in a finite relational language is superstable of finite rank, by analyzing automorphism groups using the classification of finite simple groups. That every $\omega$-categorical superstable theory of finite rank is one-based is a classical consequence of the Zilber dichotomy; see Pillay (\cite{P96}), Chapter 1. Koponen (\cite{Kop17}) also observes that a stable theory with quantifier elimination in a finite relational language is superstable by a trivial type-counting argument. However, this argument cannot be easily generalized to the simple case, so new techniques are needed.\footnote{Namely, if $T$ is a stable theory, then if $\Delta(x; y)$ is a finite set of formulas, for any set $A$ there are at most $|A|$ many $\Delta$-types $p(x) \in S^{\Delta}(A)$. But if $T$ is a stable theory with quantifier elimination in a finite relational language, then for any variables $x$ and set $A$, there is a fixed finite set $\Delta(x; y)$ such that every type $p(x) \in S(A)$ is axiomatzed by a $\Delta$-type, so if $T$ is stable then for any set $A$ there are $|A|$ many types $p(x) \in S(A)$, and $T$ is $\omega$-stable. To give another version of this argument, if $T$ is stable and $p\in S(A)$ is axiomatized by a $\Delta$-type for finite $\Delta$, then there is some $A_{0} \subseteq A$ such that $p(x)$ does not fork over $A$ (by the definition of forking in terms of definability of types; see, say, \cite{P96}, Chapter 1), so if $T$ is a stable theory with quantifier elimination in a finite relational language, then $T$ is superstable.

However, this argument cannot be easily generalized to show that every simple theory with quantfier elimination in a finite relational language is supersimple. More precisely, it is \textit{not} known that, in a simple theory, for every partial type $p(x)$ over $A$ axiomatized by a $\Delta$-type for $\Delta$ a finite set, there is some set $A_{0}$ so that $p(x)$ does not contain any formulas that fork over over $A_{0}$. For example, in a nonlow $\omega$-categorical simple theory (if such a theory exists),  Palac\'in \cite{palacin2012omega} observes that there is a formula $\varphi(x, y)$ and a $\varphi(x, y)$-type $p(x) \in S^{\varphi(x, y)}(A)$, such that for every finite set $A_{0} \subseteq A$, $p(x)$ contains a formula $\varphi(x,b)$ that divides over $A_{0}$.

Of course, in this case the degree of dividing depends on $A_{0}$, but it is also not obvious that a simple theory cannot admit an example that does not even depend on the degree of dividing: a partial type $p(x)$ axiomatized by a $\Delta$-type for $\Delta$ a finite set, such that, for some $k < \omega$, for every finite $A_{0} \subseteq A $, there is a formula in $p(x)$ that $k$-divides over $A_{0}$. All that is obvious is that in a simple theory, for a $\Delta$-type $p(x)$ over $A$ for $\Delta$ a finite set and $k < \omega$, there is some $A_{0} \subseteq A$ so that no instance of a formula in $\Delta$ \textit{itself} belonging to $p(x)$ $k$-divides over $A_{0}$. But here, we do not assert that no \textit{conjunction} of instances of formulas in $\Delta$ belonging to $p(x)$ $k$-divides over $A_{0}$.}

In the unstable case, partial progress on the Koponen conjecture can be found in numerous sources \cite{aranda2013omega, AKop15, Kop16, Kop17, Pal17, Kop19}. Aranda \cite{aranda2013omega} shows that binary homogeneous relational structures cannot have infinite monomial $SU$-rank. Ahlman and Koponen \cite{AKop15} develop restrictions (towards the structure being binary) on the minimal (that is, $SU$-rank $1$) definable sets in $\omega$-categorical simple structures in a finite relational language. Koponen \cite{Kop16} resolved the conjecture when the theory is additionally assumed to be binary. Palac\'in \cite{Pal17} proves Koponen's conjecture in special cases with additional assumptions (e.g. the theory is binary with $2$-transitive automorphism group). Conant (\cite{Co15}) proves the case where $T$ has free amalgamation with respect to the finite relational language in which it eliminates quantifiers. Koponen \cite{Kop19} solves the special case when the relations are ternary and $T$ is supersimple. However, even the $4$-ary case remains open. Following our proof of Koponen's conjecture, we formulate a generalization to $\mathrm{NSOP}_{1}$ theories.

\subsection{Stable forking and generalizations}
The ideas investigated in Section \ref{kimforking} have their origins in the \emph{stable forking conjecture}, that every instance of forking in a simple theory is witnessed by a stable formula. See Kim and Pillay \cite{kim2001around} for the history of the problem. In previous sections, we introduced a number of invariants which allowed us to quantitatively limit the complexity of forking. In Section \ref{kimforking}, we relate these measures to older qualitative limits and characterizations of forking around the stable forking conjecture including, for instance, the variants considered by Palac\'in and Wagner \cite{Pal17}. Bounding $\Fmb p$ is one of several quantitative ways in this paper in which we've bounded the complexity of forking in various specific contexts. When $\Fmb p$ is finite, one might ask for a further reduction in the complexity of forking - among all indiscernible sequences in type $p$, is the collection of Morley sequences relatively definable? This holds, for example, in the setting of the Koponen conjecture, when a theory is assumed to be $\omega$-categorical and $n$-ary. We formalize this \emph{definable Morley property} of indiscernible sequences in Definition \ref{definablemorleyproperty} and investigate its properties, showing in Theorem \ref{definabilitysimplekimforking} that it gives us nontrivial instances of the \textit{simple Kim-forking conjecture} (Question \ref{simplekimforkingconjecture}) in $\mathrm{NSOP}_{1}$ theories. As a special case, we get nontrivial instances of the simple Kim-forking conjecture for types $p$ over finite sets with $F_{Mb}(p) < \infty$, assuming a type-counting condition.

In a slightly different direction, Palac\'in and Wagner \cite{Pal17} investigated the forking relation (not a priori a definable relation) itself as a ternary relation - asking in a variety of circumstances if this relation itself is stable. For instance, Wagner and Palac\'in show that in a supersimple CM-trivial theory, the independence relation is stable, that is, $x \ind _{y_1} y_2$ can not order an infinite indiscernible sequence. In Section \ref{kimforking}, we expand this investigation, broadening and formalizing what it means to ask if the (Kim)-forking relation is simple. As a consequence of Kaplan and Ramsey's results (\cite{KR19}, stated here as Fact \ref{witnessing}) we show in Theorem \ref{global1} that with enough variables, this global analogue of the simple Kim-forking conjecture holds over a model in \textit{any} $\mathrm{NSOP}_{1}$ theory. In the setting of where $F_{Mb}(p) < \infty$, we show in Theorem \ref{global2} that only finitely many variables are required for this global version of the simple Kim-forking conjecture.  As we shall also see, in the $\omega$-categorical setting, some of the boundaries between these properties and our earlier notions dissolve. 

\subsection{Overall Organization of the paper}
In Section \ref{nmdegfmb}, we introduce various invariants, and prove some basic facts about the notions. In Section \ref{Thebasics} we give examples of some nontrivial computations of the invariants introduced in the previous section. We also establish some more involved inequalities. Section \ref{Mbfinitestable} centers around the invariant $\Fmb p$, introduced in Section \ref{nmdegfmb}, relating it to existing work on psudolinearity and giving several nontrivial examples. We also prove some results about the robustness of the invariant. For instance, we show it is preserved under equidomination of types. We use the theory of free projective planes of Hyttinen and Paolini \cite{HP21} to show that in a stable theory, $\Fmb p$ can take on any positive integer value. Section \ref{nm>} is devoted to answering the problems of \cite{freitag2023degree, freitag2023bounding}, showing that the degree of nonminimality can take on any positive integer value.
Section \ref{Kop} resolves Koponen's conjecture. Section \ref{kimforking} introduces new variations of the stable forking conjecture which go beyond simple theories. Several of these variations naturally arise from the study of the invariant $\Fmb p$.

\section{Degrees, nonminimality, and invariants of indiscernibility}\label{nmdegfmb}

In this section we  define several ranks for the complexity of forking
and compare them.

\subsection{Degree of nonminimality}

In \cite{freitag2023bounding}, the following notion was defined: 
\begin{definition} \label{RahimJim}
(Degree of nonminimality of a stationary type in a stable theory) Given a stationary type $p \in S(A)$, in a stable theory $T$ with $\U (p) >1$, the degree of nonminimality of $p$ is the minimal length $n$ of a sequence $a_0, \ldots , a_{n-1}$ of realizations of the type $p$ such that $p$ has a nonalgebraic forking extension over $a_0, \ldots , a_{n-1}$. 
\end{definition}

The same definition makes sense in the case of a general type in an arbitrary theory (though it seems very little could be proved in that case!). We will essentially only develop the notion in the settings where $T$ is a simple theory (or later, an $\mathrm{NSOP}_{1}$ theory): 

\begin{definition} \label{gennm} 
Let $T$ be simple. Given a type $p \in S(A)$ the degree of nonminimality of $p$ ($\nmdeg(p)$) is the minimal length $n$ of a sequence of realizations of the type $p$, say $a_0, \ldots , a_{n-1}$ such that $p$ has a nonalgebraic forking extension over $Aa_0, \ldots  a_{n-1}$. If there is no such $n$ then we let $\nmdeg (p) = \infty.$ If $p$ is a minimal type and thus has no nonalgebraic forking extension, then we set $\nmdeg (p) = 1$ as a matter of convention.  
\end{definition}

\begin{remark} \label{calcnm}
To calculate $\nmdeg (p)$, it is sufficient to consider $\bar a = a_0, \ldots , a_{n-1} $ to be independent over $A$. 
To see this, note that if $a_i \nind_ A \bar a_{<i}$, then if $n$ is $\nmdeg (p)$ and $i<n$, it must be that $a_i \in \acl (A \bar a_{<i}).$ In this case, if $p$ has a nonalgebraic forking extension over $\bar a$, then $p$ has a nonalgebraic forking extension over $\bar a _{-i} = (a_0, \ldots , a_{i-1} , a_{i+1} , \ldots , a_{n-1} )$. But this is impossible if $n = \nmdeg(p)$. So, it must be that $\bar a$ is an independent tuple of realizations of $p$.
\end{remark}

Next, we consider a small variation on the idea of the degree of nonminimality. 

\begin{definition}\label{finddef}
Let $T$ be simple, and $p \in S(A)$. Let $F_{ind} (p)$ be the least $n \in \m N$ such that there is a nonconstant $A$-indiscernible sequence $(a_i)_{i < \omega}$ in the type $p$ such that $\tp (a_{n} /A a_{<n} )$ is a forking extension of $p$. If no such $n$ exists, then define $F_{ind} (p) = \infty$. When $p$ is minimal, any nonconstant indiscernible sequence is a Morley sequence, and we set $F_{ind} (p) =1 $ as a matter of convention. 
\end{definition}

\begin{remark}\label{findbigger}
Note that in Definition~\ref{finddef}, 
since the sequence $(a_i)$ is indiscernible and nonconstant, $\tp (a_{n} / a_{< n} )$ must be a nonalgebraic forking extension of $p$. Because of this, one can see immediately that $F_{ind} (p) \geq \nmdeg (p).$ Moreover, if $\mathrm{SU}(p) > 1$, for $ p \in S(A)$, $F_{ind}(p)$ is finite: let $q \in S(B)$ be a nonalgebraic forking extension, and let $(a_i)_{i < \omega}$ be a Morley sequence over $B$ where for $i < \omega$, $a_{i} \models q$. Then $(a_{i})_{i < \omega}$ is not constant, but is not a Morley sequence over $A$: if it were, by Kim's lemma, $B \ind_{A} a_{0}$, but by symmetry, $a_{0} \ind_{A} B$, contradicting that $q$ is a forking extension of $p$. So $\tp (a_{n} / a_{< n} )$ is a forking extension of $p$ for some $n$, and $F_{ind}(p) \leq n$. In Section \ref{Thebasics} we will develop the properties of $F_{ind} (-).$ 

In the original primary applications of $\nmdeg (-)$ (differentially closed fields and compact complex manifolds), we will see in Proposition \ref{dcfsame} that the two quantities are identical. In Section \ref{nm>} we will show that the two quantities are not equal in general, giving examples where each is a different finite number. It is, however, interesting to note that, as we show in Section \ref{basicineq}, the two quantities are bounded in terms of each other. 
\end{remark}

So, the quantity $F_{ind}(p)$ is the least $n$ so that \emph{there exists} an indiscernible sequence $(a_i)$ in the type $p$, for which we can see that $(a_i)$ is not a Morley sequence by looking only at $a_0, a_1, \ldots , a_{n}$. By changing the \emph{existential} to a \emph{universal quantifier} over the collection of indiscernible sequences in $p$, we arrive at the following definition: 

\begin{definition} \label{Mbr}
Let $T$ be simple. Let $p \in S(A)$ be nonalgebraic. Let $F_{Mb} (p)$ be the least $n \in \m N$ such that for all $A$-indiscernible sequences $(a_i)_{i \in \omega}$ in the type $p$ that are not Morley sequences, $\tp (a_{n} /A a_{<n} )$ is a forking extension of $p$. If no such $n$ exists, then define $F_{Mb} (p) = \infty.$
\end{definition}

Note that when $\mathrm{SU}(p) > 1$, because there is, as in Remark \ref{findbigger}, a nonconstant $A$-indiscernible sequence in $p$ that is not a Morley sequence over $A$, we can take Definition \ref{Mbr} to quantify over all \textit{nonconstant} $A$-indiscernible sequences that are not Morley sequences, as in Definition \ref{finddef}. We may also state the relationship between $F_{ind}(p)$ to $F_{Mb}(p)$ as follows: $F_{ind}(p)$ is the least $n$ so that there is a nonconstant $A$-indiscernible sequence $\{a_{i}\}$ in $p$ that is $n$-independent over $A$ (i.e. any $n$ of the $a_{i}$ are independent over $A$) but not $n+1$-independent over $A$, while $F_{Mb}(p)$ is the \textit{greatest} $n$ so that there is an $A$-indiscernible sequence $\{a_{i}\}$ in $p$ that is $n$-independent over $A$ but not $n+1$-independent over $A$.

We first compute $F_{Mb}(p)$ in the most basic case:

\begin{definition}\label{algfork}
Suppose the type $p \in S(A)$ has the property that for all $a \models p$, if $a \nind_{A} B$ then $\mathrm{acl}^{eq}(aA) \cap \mathrm{acl}^{eq}(BA) \supsetneq \mathrm{acl}^{eq}(A)$. In this case, let us say that $p$ has \emph{algebraic forking}.    
\end{definition} 
This generalizes one-basedness in the following sense:

\begin{fact}\label{forkingonebased}
   (see, e.g. \cite{P96} or \cite{Wag00}) A simple theory with elimination of hyperimaginaries is one-based if and only if every type has algebraic forking.
\end{fact}

In general, whenever $\mathrm{acl}^{eq}(aA) \cap \mathrm{acl}^{eq}(BA) \supsetneq \mathrm{acl}^{eq}(A)$, we have $a \nind_{A} B$, so: 

\begin{prop} \label{algforkingprop}
If the type $p \in S(A)$ has algebraic forking and $a \models p$, then 
$$a \nind_{A} B \iff \mathrm{acl}^{eq}(aA) \cap \mathrm{acl}^{eq}(BA) \supsetneq \mathrm{acl}^{eq}(A).$$    
\end{prop}

For example, if $SU(p)=1$ or $p$ is one-based, then $p$ has algebraic forking.  

\begin{prop}
    \label{typeswithalgebraicforkinghavembone}
    
    Suppose $p$ has algebraic forking. Then $F_{Mb}(p) = 1$.
\end{prop}

\begin{proof}
Suppose $(a_{i})_{i < \omega}$ is an $A$-indiscernible sequence where $a_{0} \models p$, and is not a Morley sequence over $A$. If $p$ has algebraic forking, then for some $n < \omega$, $\mathrm{acl}(a_{n}A) \cap \mathrm{acl}(a_{0} \ldots a_{n-1}A) = B \supsetneq A$. The sequence $(a_{n+i})_{i < \omega}$ is indiscernible over $\mathrm{acl}(a_{0} \ldots a_{n-1}A)$, so in particular, is indiscernible over $B$. Then $B \subseteq \mathrm{acl}(a_{n}A)$, so $B \subseteq \mathrm{acl}(a_{n+1}A)$, and $\mathrm{acl}(a_{n}A) \cap \mathrm{acl}(a_{n+1}A) \supseteq B \supsetneq A$. So $a_{n+1} \nind_{A} a_{n}$, and by indiscernibility, $a_{1} \nind_{A} a_{0}$.
\end{proof}

Except for edge cases in which $\nmdeg(p)$ either doesn't make sense or is a matter of convention (e.g. when $\U(p) =1$), the definitions and above discussion make it transparent that $\nmdeg(p) \leq F_{ind} (p) \leq  \Fmb{p}$, but as we'll show in Section \ref{Thebasics}, the quantities can all be distinct. 

\subsection{Related notions involving canonical bases}
In this subsection, in order to formulate the notions, we must work in a setting with a suitable notion of canonical bases for types. For instance, it would suffice to work in a stable theory with elimination of imaginaries (or else in our formulations, work in $T^{eq}$). More generally, one might work in a simple theory with elimination of hyperimaginaries. 

Each of the three notions given above in Definitions \ref{gennm}, \ref{finddef}, \ref{Mbr} can be related to the rank of the canonical base of certain forking extensions of the type $p$; we will explain the connection in the coming pages. 

\begin{definition} \label{CBbounddef} Given a type $p,$ let 

\begin{enumerate} 
 \item $\fcb{p}$ be the minimal number $m$ of realizations $a_1, \ldots a_m$  of $p$ such that the canonical base of $p$ is contained in the algebraic closure of $a_1, \ldots , a_m$.
\item $\ficb{p}$ be the minimal number $m$ of \emph{independent} realizations $a_1, \ldots a_m$  of $p$ such that the canonical base of $p$ is contained in the algebraic closure of $a_1, \ldots , a_m$.
\end{enumerate}
\end{definition}

Though it isn't obvious, the quantity $\ficb{p}$ is intensively studied in the context of differential equations and modeling, where the notion is closely related to parameter identifiability, as we will explain\footnote{We are not saying that the notion itself is defined there or studied in these terms. Parameter identifiability problems have recently been considered in differential algebra, but have a much longer and more extensive history more generally.}. In parameter identifiability problems, one typically considers variations on the following question: given a fixed system of differential equations with unknown parameters and some output functions, can one identify the unknown parameters from a series of observations of \emph{independent} solutions to the system? Numerous variations on this general idea are extensively studied in the differential equations literature. See for instance \cite{dong2023differential, ovchinnikov2021computing, jeronimo2023weak} for a few recent examples in the differential algebraic context.

Specifically, when $p$ is the generic type of the solution set of a system of differential equations, $\Sigma$, the notion $\ficb{p}$ is closely related to the minimum value of $r$ such that the parameters of $\Sigma$ are multi-experiment identifiable in $r$ many experiments \cite[Section 2.2]{ovchinnikov2022multi}. There is a slight difference in that we are considering the algebraic closure, while in parameter identifiability problems, the rational expressions of the functions and their derivatives are typically used. This is a very small difference, and in the setting of parameter identifiability would yield values of $r$ which differ by at most one. In the differential context, the emphasis has often been on the design of efficient algorithms for checking the identifiability of a system.

It is natural to define the relative versions of these notions for forking extensions: 

\begin{definition} \label{cbrkdef} Given a type $p \in S( \emptyset )$ and $q$ a forking extension of $p$ let: 

\begin{enumerate}
 \item $\fcb{q/p}$ be the minimal number $m$ of realizations $a_1, \ldots a_m$  of $p$ such that the canonical base of $q$ is contained in the algebraic closure of $a_1, \ldots , a_m$
\item $\ficb{q/p}$ be the minimal number $m$ of independent realizations $a_1, \ldots a_m$  of $p$ such that the canonical base of $q$ is contained in the algebraic closure of $a_1, \ldots , a_m$.
\end{enumerate}
\end{definition}

Given a type $p$ of rank larger than one, there is a type $q$ such that $\nmdeg (p) = \fcb {q/p}$; in fact: 

\begin{eqnarray*}
\nmdeg(p) &=& \min \left\{\fcb {q/p} \, | \, q \text{ a nonalgebraic forking extension of }p \right\} \\
&=& \min \left\{\ficb {q/p} \, | \, q \text{ a nonalgebraic forking extension of }p \right\}.
\end{eqnarray*}

\section{Basic properties of the invariants} \label{Thebasics}
\subsection{Basic Examples} 
We next give a series of examples centered around the invariants of the previous section.

\begin{example} In any strongly minimal theory, if $a,b,c$ are independent, the
generic $2$-type $p$ has $\nmdeg(p) = F_{ind}(p)= 1$, witnessed by $(a,b)$
and $(a,c)$. This is also true when $p$ is any $n$-type of rank at least $2$ in a strongly minimal theory. 
\end{example}

\begin{example} \label{ACFMB}
For $2$-types in strongly minimal theories, the analysis of $\mb$ is more interesting. For instance, as we will explain, in the theory of algebraically closed fields, the generic $2$-type $p$ has $\mb(p)=\infty$.

Let $T=ACF_0$, the theory of algebraically closed fields of characteristic zero. Let $p \in S_2 ( \emptyset )$ be the generic type of affine $2$-space. 
Consider the indiscernible sequence in $p$ given by independent generic points of a generic affine curve of degree $d$, 
$$\sum_{i,j} c_{i,j} x^i y^j  +1 = 0,$$
where $c_{i,j}$ are independent transcendentals. 
Let $q$ be the generic type of the curve. It isn't hard to see that the tuple $\bar c$ of the transcendentals $c_{i,j}$ is the canonical parameter of the forking extension $q$ of $p$, and that $\bar c$ is in the field generated by $(d+1)^2-1$ many independent realizations of $q$, but no fewer (each point yields one independent linear condition on the coefficients). 
So, taking an indiscernible sequence in $q$, $(a_i)$, over $\emptyset$, $(a_i)_{i \leq (d+1)^2-1}$ is an initial segment of a Morley sequence (ie this collection of points in independent), but $(a_i)_{i \leq (d+1)^2}$ is not. Since a similar example exists for any $d \in \m N$, we can see that $\mb {(p)} = \infty$. 
\end{example}

The above example, while familiar and illustrative, turns out to not be particular to $ACF_0$, strongly minimal types, or even types in stable theories. In general, in Corollary \ref{pseudolinearitymb}, we show that if $p$ is a Lascar type with $\mathrm{SU}(p) = 1$ (in a stable or $\omega$-categorical simple theory) then either $\Fmb{p \otimes p} = \infty$ or $p$ is \emph{linear}.\footnote{See Definition \ref{deflin}.} 
In Section \ref{ocat} we give an example of a type of $\mathrm{SU}$-rank $2$ in an $\omega$-categorical supersimple theory of finite rank in which $\nmdeg (p) < F_{ind} (p)$, but this is rather more involved than the basic examples of this section.
Examples of stable theories with $\nmdeg (p) < F_{ind} (p)$ are given in Section \ref{nm>}.
In many cases, the two invariants are identical - see remarks following the proof of our next result. 

Let $DCF_0$ denote the theory of differentially closed fields of characteristic zero in the language of rings together with a distinguished derivation $\delta$. The proof of the following result is quite similar to that found in Section 2 of \cite{freitag2023degree}.

\begin{prop} \label{dcfsame}
    For a type $p$ in the theory $DCF_0$, $\nmdeg(p) = F_{ind} (p).$
\end{prop}
\begin{proof}
Let $p \in S(A)$. If $p_n \rightarrow p_{n-1} \rightarrow \ldots \rightarrow p_1$ is an analysis of the type $p$ over $A$, $F_{ind} (p)$ is bounded by $F_{ind} (p_1)$ and also $F_{ind}$ of the types given by the generic fibers of the maps in the analysis. Thus, it suffices to prove the result in the case that $p$ has \emph{no proper fibrations} \cite[Definition 2.1]{moosa2014some}. It follows by \cite[Proposition 2.3]{moosa2014some} that $p$ is either 
\begin{enumerate}
    \item interalgebraic over $A$ with a finite collection of independent realizations of a locally modular minimal type, or 
    \item interalgebraic with a type $q$ which is internal to a nonlocally modular type. 
\end{enumerate}
    In the first case, both $F_{ind} (p)= \nmdeg{(p)} =1$. In the second case, if the binding group of $p$ relative to the constants is not generically $2$-transitive \cite[for the defintion of generic transitivity]{freitag2023bounding}, then $F_{ind}(p) = 1$, and $\nmdeg(p) \leq F_{ind} (p) $ in general (see Remark \ref{findbigger}), so the quantities are equal. In the second case, since $F_{ind} (-) $ is invariant up to interalgebraicity, we can assume $p$ is internal to the nonlocally modular type - in this case, the constants. 

    Suppose now that $F_{ind} (p) >1$. Then it must be that $p$ and also $p \otimes p$ are weakly orthogonal to the constants over $A$ (if not, it is easy to see that $F_{ind} (p) =1$ by taking a Morley sequence whose image in $\mc C^n$ coming from nonorthogonality lies on an algebraic curve). It follows that both $p$ and $p \otimes p$ are isolated types. We wish now to reduce to the case that the action of the binding group $G$ on $p$ is $2$-transitive. This is equivalent to any two realizations of $p$ being independent. As the binding group is generically $2$-transitive, with $p \otimes p$ being isolated, there must be an $A$-definable equivalence relation $\sim$ on $p$ such that $p/ \sim$ has a $2$-transitive binding group. 

    But now the binding group is $2$-transitive and definably isomorphic to an algebraic group action acting on an algebraic variety over a field of characteristic zero, and so the classification theorem of Knopp \cite{knop1983mehrfach} tells us that the binding group action on $p$ is isomorphic to either a group of affine transformations of affine space or the action of $PGL_{n+1}$ on $\m P^n.$ Here it must be that $n>1.$ Both of these group actions preserve collinearity of points. Taking the Morley sequence of points on a line, we can see that $F_{ind} (p) =2.$ $\nmdeg(p)=2$ in this case as well since the binding group acts $2$-transitively. 
\end{proof}

Again, similarly to the results of Section 2 of \cite{freitag2023degree}, in the previous result, one does not really need to assume that the theory $T$ is $DCF_0$.
Rather, it applies more generally when $T$ is $\omega$-stable, eliminates $\exists ^ \infty$, and any nonlocally modular minimal type over $A$ is nonorthogonal to a pure algebraically closed field over $A$. So, for instance, the proof applies to the theory of compact complex manifolds or more generally the class of theories laid out in Section 3 of \cite{freitag2023degree}.

The proof of the previous result also gives a reasonable indication of the difficulties associated with giving 
examples in which $\nmdeg(p) \neq F_{ind}(p)$ - one either has to give an example of a theory whose binding groups have different transitivity properties than algebraic groups or give a theory in which any minimal type in some nonorthogonality class of minimal types requires parameters to define. We don't really know examples of such groups - indeed conjectures like those of Cherlin-Zilber, Borovik-Cherlin \cite{bcorig}, and Borovik-Deloro \cite{borovik2019binding} are towards the nonexistence of such groups. Most natural theories don't seem to have the second property, but we build such a (simple) theory in Section \ref{ocat} using a construction of Evans, and show that for our example there are types $p$ with $\nmdeg(p) \neq F_{ind}(p)$.

\begin{example}
Work in the theory $DCF_0$, and let $p \in S(\mathbb{Q}(t))$ be the generic type of the equation $x''=0$, for $t$ a realization of $p$. It isn't hard to see that over $\m Q(t)$   the solutions of $x''=0$ are in bijective correspondence with $\m A^2 (\m C)$ with the induced structure of subsets definable in $ACF_0$ (e.g. by stable embeddedness). Let $q \in S(c_{1} \ldots c_{2d})$ be the forking extension of $p$ realized by solutions to $x''=0$ of the form $at+b$ where $a$ and $b$ are constants and $(a,b)$ is a generic point on the curve $x=\prod_{i=1}^{2d} (y-c_i)$ where $c_i$ are independent transcendental constants. Then taking solutions of $q$, one can see that $2d$ many are required to define the canonical base of $q$ (each solution gives another point on the curve $x=\prod_{i=1}^{2d} (y-c_i)$ and imposes one linear condition on the coeficients $c_i$). On the other hand, taking realizations of $p$ given by $c_i t+ c_{2d-i}$, it is easy to see that a canonical base of $q$ is contained in the $d$ many realizations as $i$ ranges from $1$ to $d$. In this example we have:  $$\fcb{q/p}=\ficb{q/p} < \fcb{q} = \ficb{q}.$$

Especially with regard to differential algebra, the quantities $\fcb{q/p}, \, \ficb{q/p} , \, \fcb{q} , \, \ficb{q}$ have natural interpretations in terms of existing work. For instance, when $q$ is the generic type of a solution to a system of ODE, and the parameters of the system are unknown constants which form a canonical base, $\ficb{q}$ or $\ficb{p} +1 $ is the number of experiments required to identify the parameters. In general,
\begin{equation*}
\nmdeg {p} = \min \left\{  \fcb{q/p} \, | \, q \text{ is a nonalgebraic forking extension of } p \right\} 
\end{equation*}

\end{example}

\subsection{Inequalities and basic results} \label{basicineq}

In this section, we assume that $T$ is simple. We will see that in a simple theory, $\mathrm{nmdeg}(p)$ can be strictly less than $F_{ind}(p)$. Nonetheless, in the finite-rank setting, the two quantities are related (can be bounded in terms of each other). The inspiration for these arguments comes from problems of parameter identifiability of differential equations; thanks to Scanlon and Pillay for discussions around this topic at various times. 

\begin{lemma}\label{cbmorleysequence} 
 Let $T$ be a simple theory, and\footnote{$c\in\mathrm{bdd}(A) $ means $c$ has boundedly many images over A under automorphisms of $M$; for the more intricate definition of $S(\mathrm{bdd}(A))$, see \cite[Definition 1.2]{HartKP}.} $q \in S(\mathrm{bdd}(A))$. Let $\{a_{i}\}_{i < \omega}$ be a Morley sequence in $q$ over $A$. Suppose $\mathrm{SU}(\mathrm{Cb}(q)) = n < \omega$. Then  $\mathrm{Cb}(q) \subseteq \mathrm{bdd}(\{a_{i}\}_{i < n})$.  
\end{lemma}

\begin{proof}
  Let $C= \mathrm{Cb}(q)$, and suppose otherwise; we show that, for $0 \leq i \leq n$, $a_{i} \nind_{a_{< i}} C$, contradicting $\mathrm{SU}(\mathrm{Cb}(q)) = n$. We may assume $A = \mathrm{bdd}(A)$, so $C \subseteq A$. Suppose that for some $0 \leq i \leq n$, $a_{i} \ind_{a_{< i}} C$. Since $a_{i} \models q$, $a \ind_{C} A$. Since $\{a_{i}\}_{i < \omega}$ is a Morley sequence in $p$ over $A$, $a_{i} \ind_{A} a_{< i} $. So by transitivity, $a_{i} \ind_{C} A a_{<{i}}$, and therefore $a_{i} \ind_{C a_{< i}} A a_{<{i}}$. By the assumption and transitivity, $a_{i} \ind_{a_{< i}} A a_{<_{i}}$. Since $a_{i} \ind_{A} a_{< i} $, $C = \mathrm{Cb}(q) \subseteq \mathrm{bdd}(\mathrm{Cb}(\mathrm{tp}(a_{i}/A a_{< i})))$. But by $a_{i} \ind_{a_{< i}} A a_{<_{i}}$, $\mathrm{Cb}(\mathrm{tp}(a_{i}/A a_{< i})) \subseteq \mathrm{bdd}(a_{< i})$, so $C \subseteq \mathrm{bdd}(a_{< i}) \subseteq \mathrm{bdd}(a_{< n})$, a contradiction.
\end{proof}

This allows us to bound $F_{ind}(p)$ in terms of $\mathrm{nmdeg}(p)$, when $p$ is of finite rank.

\begin{cor}\label{eqnmfin}
       Let $T$ be simple and $\mathrm{SU}(p) = n < \omega $. Then $\mathrm{nmdeg}(p) \leq F_{ind}(p) \leq n \cdot \mathrm{nmdeg}(p)$.
\end{cor}

\begin{proof}
   Assume without loss of generality that $p \in S(\emptyset)$. For the remaining inequality (by Remark~\ref{findbigger}, $\nmdeg (p)  \leq F_{ind}(p)$), let $\mathrm{nmdeg}(p) = N$. Then there are $a_{1}, \ldots, a_{N} \models p$ so that $p$ has a nonalgebraic forking extension $q$ over $\mathrm{bdd}(a_{1} \ldots a_{N})$. Let $I = \{a'_{i}\}_{i < \omega}$ be a Morley sequence in $q$ over $\mathrm{bdd}(a_{1} \ldots a_{N})$; then $I$ is nonconstant because $q$ is nonalgebraic. Because $\mathrm{Cb}(q) \subseteq \mathrm{bdd}(a_{1} \ldots a_{N})$, by the Lascar inequalities, $U(\mathrm{Cb}(q)) \leq n \cdot N$. So by the previous proposition, $C = \mathrm{Cb}(q) \subseteq \mathrm{bdd}(a'_{<Nn})$. But $a_{Nn} \ind_{A} a_{<Nn}$ because $I$ is a Morley sequence over $A$, and $a_{Nn} \ind_{C} A$ because $a \models q$, so $a_{Nn} \ind_{C} Aa_{<Nn}$ by transitivity, and therefore $a_{Nn} \ind_{a_{< Nn}} Aa_{<Nn}$ because $C  \subseteq \mathrm{bdd}(a'_{<Nn})$. Thus $a_{Nn} \ind_{a_{< Nn}} A$. But because $a_{Nn} \models q$ and $q$ forks over $p$, $a_{Nn} \nind A$, so $a_{Nn} \nind a_{< Nn}$, showing $F_{ind}(p) \leq n N$.
\end{proof}

In the finite-rank case (or the finite $\mathrm{SU}_{p_{0}}$-rank case, where $p_{0}$ is a regular type), the ranks of canonical bases will bound how much independence an indiscernible sequence has. In particular, each indiscernible sequence will be a Morley sequence in a type over some set, and the canonical base of that type will give upper and, in some cases, lower bounds on $k$ so that this sequence is $k$-dependent if it is not a Morley sequence.

\begin{prop} \label{dependencecanonicalbases}
   Let $A \subseteq B$ with $B=\mathrm{bdd}(B)$, $q \in S(B)$, $p = q|_{A}$. Let $I = \{a_{i}\}_{i < \omega}$ be a Morley sequence in $q$ over $B$.
   
   (1) If $\mathrm{SU}(\mathrm{Cb}(q)/A)\leq k$, where $1 \leq k < \omega$, then $a_{k} \nind_{A} a_{< k}$; that is, $I$ is $k+1$-dependent. 

     (2) If $\mathrm{SU}(q)=n < \omega$, $\mathrm{SU}(p) = n+1$, and $\mathrm{SU}(\mathrm{Cb}(q)/A) > k$ where $1 \leq k < \omega$, then $a_{k-1} \ind_{A} a_{< k-1} $; that is, $I$ is $k$-independent.
\end{prop}

\begin{proof}
    (1). Since $a_{k} \ind_{B} a_{< k}$, $\mathrm{bdd}(\mathrm{Cb}(\mathrm{tp}(a/\mathrm{bdd}(Ba _{< k})) = \mathrm{bdd}(\mathrm{Cb}(q)) \subseteq \mathrm{bdd}(Aa_{< k})$, the inclusion by Proposition \ref{cbmorleysequence}, so $a_{k} \ind_{Aa_{< k}} B a_{< k}$. Assume for a contradiction that $a_{k} \ind_{A} a_{< k}$; then by transitivity, $a_{k} \ind_{A} B$. But $a_{k} \models q$ and $q$ is a forking extension of $p$ (since $\mathrm{SU}(\mathrm{Cb}(q)/A) \geq 1$), a contradiction.

    (2)  Suppose, for a contradiction, that $a_{k-1} \nind_{A} a_{< k-1} $. Then $n = \mathrm{SU}(a_{k}/B) = \mathrm{SU}(a_{k}/Ba_{k-1}) \leq \mathrm{SU}(a_{k}/Aa_{k-1}) \leq n$, the first two equalities by $a_{k} \models q$ and the fact that $I$ is a Morley sequence, the last by the assumption. So $a_{k} \ind_{Aa_{<k}}Ba_{< k}$. Then $ \mathrm{bdd}(\mathrm{Cb}(q)) =\mathrm{bdd}(\mathrm{Cb}(\mathrm{tp}(a/\mathrm{bdd}(Ba _{< k}))\subseteq \mathrm{bdd}(Aa_{< k})$, so $\mathrm{SU}(\mathrm{Cb}(q)/ a_{< k})=0$, and it suffices to show that, for $0 < i < k$, $\mathrm{SU}(\mathrm{Cb}(q)/ A{a_{< i}}) - \mathrm{SU}(\mathrm{Cb}(q)/ A{a_{< i+1}}) \leq 1$.   By the Lascar inequalities, the quantity $\mathrm{SU}(\mathrm{Cb}(q)/ A{a_{< i}}) - \mathrm{SU}(\mathrm{Cb}(q)/ A{a_{< i+1}})$ is equal to the quantity $\mathrm{SU}(a_{i}/ Aa_{<i}) - \mathrm{SU}(a_{i}/ Aa_{<i}\mathrm{Cb}(q))$.  But $n+1= \mathrm{SU}(a_{i}/B)=\mathrm{SU}(a_{i}/a_{<i}B)$ while $n=\mathrm{SU}(a_{i}/A)$, so because $A \subseteq Aa_{<i} \subseteq Aa_{<i}\mathrm{Cb}(q) \subseteq \mathrm{bdd}(a_{<i}B)$, it follows that $\mathrm{SU}(\mathrm{Cb}(q)/ A{a_{< i}}) -\mathrm{SU}(\mathrm{Cb}(q)/ A{a_{< i+1}}) \leq 1$.  
    
\end{proof}

Note the resemblance between part (2) of this proposition and the direction ((b) $\Rightarrow$ (a)) of Theorem 3.3 of Kim (\cite{Kim10}), and the resemblance, in the case where $\mathrm{SU}(p)=2$, between part (1) of this proposition and the direction ((a) $\Rightarrow$ (b)) of that theorem of Kim, whose proof, as stated in \cite{Kim10}, uses similar ideas to the proof of Theorem 3.6 of \cite{DPK03}. We expect the above proofs to be similar to a proof of Theorem 3.3 of Kim (\cite{Kim10}).

When $\mathrm{SU}(p) = 1$, $\mathrm{nmdeg}(p)=F_{ind}(p)=F_{Mb}(p)=1$. So $\mathrm{SU}(p) = 2$ is the first case in which $\mathrm{nmdeg}(p)\leq F_{ind}(p) \leq F_{Mb}(p)$ can be greater than $1$. In this case, by the previous proposition, the possible ranks of the canonical bases of the forking extensions of $p$ will give both upper and lower bounds for $F_{ind}(p)$ and $F_{Mb}(p)$.

\begin{cor}\label{rank2}
    Let $p \in S(A)$. Let 
    $$N_{-}(p) = \mathrm{min}(\{\mathrm{SU}(\mathrm{Cb}(q)/A): p \subset q \in S(B), A \subseteq B = \mathrm{bdd}(B), \mathrm{SU}(q) > 0, \mathrm{SU}(\mathrm{Cb}(q)/A) \geq 1 \})$$

    and

    $$N_{+}(p) = \mathrm{max}(\{\mathrm{SU}(\mathrm{Cb}(q)/A): p \subseteq q \in S(B), A \subseteq B = \mathrm{bdd}(B) \}.)$$ (If not defined, define $N_{+}(p)= \infty$.)

    Then 

    (1) $F_{ind}(p) \leq N_{-}(p)  $, and if $\mathrm{SU}(p) = 2$, $N_{-}(p) = F_{ind}(p)$.

    (2) $F_{Mb}(p) \leq N_{+}(p)  $, and if $\mathrm{SU}(p) = 2$, $N_{+}(p) = F_{Mb}(p)$.
\end{cor}

\begin{proof}
    Note that $I$ is an indiscernible sequence over $A$ that is not a Morley sequence over $A$, if and only if $I$ is a Morley sequence over $B$ in $q \in S(B), A \subseteq B = \mathrm{bdd}(B)$, for $q \supset p$ some forking extension of $p$ (i.e. $\mathrm{SU}(\mathrm{Cb}(q)/A) \geq 1$). If additionally $\mathrm{SU}(p) = 2$, then $I$ is a nonconstant indiscernible sequence over $A$ that is not a Morley sequence over $A$, if and only if $I$ is a Morley sequence over $B$ in $q \in S(B), A \subseteq B = \mathrm{bdd}(B)$, with $\mathrm{SU}(q)=1$--thereby satisfying the hypothesis of (2) of the previous proposition. So the corollary follows from the previous proposition by the definitions of $F_{ind}$ and $F_{Mb}$.

\end{proof}

Even if $p$ is only assumed to be of finite rank, rather than rank $2$, the existence of some $N$ so that $F_{Mb}(q) \leq N$ \textit{for all extensions} $q \supseteq p$ is equivalent to there existing some $N'$ so that \textit{all extensions} $q \supseteq p$ have the property that the canonical bases of forking extensions of $q$ have rank at most $N$.

\begin{prop}
       \label{Mbbounds}
        
        Let $p \in S(A)$, $\mathrm{SU}(p) < \omega $. Suppose there is some $N$ so that for all $q \supseteq p$, $q \in S(B)$, $A \subseteq B $, $F_{Mb}(q) \leq N$. Then there is some $N'$ so that for all $q \supseteq p$, $q \in S(B)$, $A \subseteq B $, $N_{+}(q) \leq N'$  (where $N_{+}(q)$ is as in the previous corollary). (So the converse also holds, by the clause  $F_{Mb}(p) \leq N_{+}(p)  $ of part (2) of the previous corollary.)
\end{prop}

\begin{proof}
   We claim that suffices to show the following: Suppose that $p$ has the property that for $q \supseteq p$, $q \in S(B)$, $A \subseteq B$, $F_{Mb}(q) \leq N$, and that $\mathrm{SU}(p)\leq k$. Then if $I=\{a_{i}\}_{i < \omega}$ is an $A$-indiscernible sequence of realizations of $p$, $a_{i} \ind_{Aa_{< kN}} a_{<i}$ for $i \geq kN$. To show this suffices, assume we have shown this and let $r \supseteq p$, $r \in S(B)$, $A \subseteq B = \mathrm{bdd}(B)$; we show that $\mathrm{SU}(\mathrm{Cb}(r)/A) \leq k^{2}N$, so that $N_{+}(p) \leq k^{2}N$--this is enough, since $\mathrm{SU}(q) \leq k$ for $q \supseteq p$ so we will also have shown that for all $q \supseteq p$ $q \in S(B)$, $A \subseteq B $, $N_{+}(q) \leq k^{2} N$. To show $\mathrm{SU}(\mathrm{Cb}(q)/A) \leq k^{2}N$, let $I$ be a Morley sequence in $r$ over $B$, so that $a_{i} \ind_{Aa_{< kN}} a_{<i}$ for $i \geq kN$ by assumption. Choose $a_{\omega}$ so that $I \smallfrown \langle a_{\omega }\rangle$ remains a Morley sequence over $B$; by strong finite character, we also have that $a_{\omega} \ind_{Aa_{< kN}} I$. Then since $I \smallfrown \langle a_{\omega }\rangle$ is a Morley sequence over $B$, $\mathrm{tp}(a_{\omega}/BI)$ is a nonforking extension of $q$, so $\mathrm{bdd}(\mathrm{Cb}(\mathrm{tp}(a_{\omega}/BI)))=\mathrm{bdd}(\mathrm{Cb}(q))$, but by finite satisfiability, $\mathrm{tp}(a_{\omega}/BI)$ is itself a nonforking extension of $\mathrm{tp}(a_{\omega}/AI)$, so by $a_{\omega} \ind_{Aa_{< kN}} I$ and transitivity, of  $\mathrm{tp}(a_{\omega}/Aa_{< kN})$. So $\mathrm{bdd}(\mathrm{Cb}(q))=\mathrm{bdd}(\mathrm{Cb}(\mathrm{tp}(a_{\omega}/BI))) \subseteq \mathrm{bdd}(Aa_{<kN})$, so by the Lascar inequalities, $\mathrm{SU}(\mathrm{Cb}(q)/A) \leq k^{2}N$. (Note that, using the ideas of the proof of part (2) of Proposition \ref{dependencecanonicalbases}, we could have improved this bound.) 

   Therefore, it remains to show that if $p$ has the property that for $q \supseteq p$, $q \in S(B)$, $A \subseteq B$, $F_{Mb}(q) \leq N$, and that $\mathrm{SU}(p)\leq k$, and if $I=\{a_{i}\}_{i < \omega}$ is an $A$-indiscernible sequence of realizations of $p$, then $a_{i} \ind_{Aa_{< kN}} a_{<i}$ for $i \geq kN$. We show this by induction on $k$. If $p$ is already a Morley sequence, then we are done. If not, then by $F_{Mb}(q) \leq N$, $p'= \mathrm{tp}(a_{N}/Aa_{< N})$ is a forking extension of $p$ and has $\mathrm{SU}(p)\leq k-1$; $I=\{a_{i}\}_{ N \leq i < \omega}$ is an $Aa_{< N}$-indiscernible sequence of realizations of $p'$. Moreover, $p'$ continues to have the property that for $q \supseteq p'$, $q \in S(B)$, $Aa_{< N} \subseteq B$, $F_{Mb}(q) \leq N$, so by the induction hypothesis we have that $a_{i} \ind_{Aa_{< N}a_{N} \ldots a_{Nk}} a_{<i}$ for $i \geq kN$.

\end{proof}

\begin{remark}
    When $p =\mathrm{tp}(a/A)$ where $a$ is a set of realizations of a regular type $p_{0} \in S(B)$, $\mathrm{bdd}(C)= C \subseteq A$, some of these results hold replacing $\mathrm{SU}$ with $\mathrm{SU}_{p_{0}}$, namely analogues of Proposition \ref{Mbbounds} giving lower bounds on independence:  when $I = \{a_{i}\}_{i < \omega }$  is a Morley sequence 
 in an extension $q$ of $p$ to $B=\mathrm{bdd}(B)$ with $\mathrm{SU}_{p_{0}}(q)=n < \omega$, $\mathrm{SU}_{p_{0}}(p) = n+1$, and $\mathrm{SU}_{p_{0}}(\mathrm{Cb}(q)/A) > k(n+1)$ where $1 \leq k < \omega$, then $a_{k-1} \ind_{A} a_{< k-1} $.

    Therefore, for 
    $$N^{p_{0}}_{+}(p) = \mathrm{max}(\{\mathrm{SU}_{p_{0}}(\mathrm{Cb}(q)/A): p \subseteq q \in S(B), A \subseteq B = \mathrm{bdd}(B) \})$$

    $$N^{p_{0}}_{-}(p) = \mathrm{min}(\{\mathrm{SU}_{p_{0}}(\mathrm{Cb}(q)/A): p \subseteq q \in S(B), A \subseteq B = \mathrm{bdd}(B) \})$$

    we have $N^{p_{0}}_{-}(p) \leq 2f_{Mb}(p)$, $N^{p_{0}}_{+}(p) \leq 2f_{Mb}(p)$ whenever $\mathrm{SU}_{p_{0}}(p) = 2$.
    
    Moreover, when $\mathrm{SU}_{p_{0}}(p)$ is finite and there is some $N$ so that for all $q \supseteq p$, $q \in S(B)$, $A \subseteq B $, $F_{Mb}(q) \leq N$, then there is some $N'$ so that for all $q \supseteq p$, $q \in S(B)$, $A \subseteq B $, $N^{p_{0}}_{+}(q) \leq N'$. The reason why these proofs work for $\mathrm{SU}_{p_{0}}$ is that, because $q$ and $\mathrm{lim}^{+}(I/I)$ have a common nonforking extension, what we are really showing is that when  $F_{Mb}(q) \leq N$ ,  and $\mathrm{SU}(p) - \mathrm{SU}(\mathrm{lim}^{+}(I/I)) =1$ (or $\mathrm{SU}(p)= k$), then $\mathrm{lim}^{+}(I/I)$ does not fork over $\{a_{i}\}_{i < N}$ ($\{a_{i}\}_{i < kN}$), which implies the conclusion about the rank of $\mathrm{Cb}(q)$, which satisfies $\mathrm{bdd}(\mathrm{Cb}(q))=\mathrm{bdd}(\mathrm{Cb}(\mathrm{lim}^{+}(I)))$. (The former case, where $\mathrm{SU}(p) - \mathrm{SU}(\mathrm{lim}^{+}(I/I)) =1$, is not Proposition \ref{Mbbounds} but is contained within the proof.)  Since this will all happen inside of $p_{0}$, the proof works for $\mathrm{SU}_{p_{0}}$ as well.

\end{remark}

In rank $2$, one special case where the ranks of canonical bases are of interest is that of \textit{plane curves}, or minimal forking extensions of the type of two independent realizations of a regular (or minimal) type. The following extension of linearity was introduced for minimal types in \cite{P96}, then generalized to regular types in \cite{TW03} (an excellent exposition of which can be found in \cite{BTW02}; see \cite{Kim10} for further results on $k$-linearity); this is easily seen to be equivalent to Definition 16 of \cite{BTW02}.

\begin{definition} \label{deflin}
    Let $p$ be a regular type over $A$. Then $p$ is $k$\emph{-linear} if, when $p'=\mathrm{tp}(ab/A)$ where $a \ind_{A} b$, $a, b \models p$, every regular extension $q \supseteq p'$, with $q \in S(B)$, $A \subset B = \mathrm{bdd}(B)$, and $\mathrm{SU}_{p}(q) = 1$, has $\mathrm{SU}_{p}(\mathrm{cb}(q)/A) \leq k$, and moreover, there exists such a type $p'$ and extension $q$ with $\mathrm{SU}_{p}(\mathrm{cb}(q)/A) = k$. The type $p$ is \emph{linear} if and only if it is $1$-linear.
\end{definition}

Note that in the case where $\mathrm{SU}(p)=1$, $\mathrm{SU}_{p}$ is the usual $\mathrm{SU}$-rank and we can omit the assumption that $q$ is regular, because $\mathrm{SU}(q) = 1$ already implies that $q$ is regular.

\begin{remark}
   \label{independenceofpotimesp}
    
    At least in the case of $\mathrm{SU}(p)=1$, $k$-linearity does not depend on $p'$ in the above definition when $p \in S(A)$ is a Lascar type: if $p \otimes p$ is a type over $A$ realized by two independent realizations of $p$ over $A$, and $p\otimes p$ has a forking extension $q$ with $\mathrm{SU}(\mathrm{Cb}(q)/A ) \geq k$, then for $(p \otimes p)'$ some other type over $A$ realized by two independent realizations of $p$ over $A$, $(p \otimes p)'$ has a forking extension $q'$ with $\mathrm{SU}(\mathrm{Cb}(q)/A ) \geq k$. To see this, we apply the argument used by \cite{TW03} to get the group configuration from pseudolinearity; see Fact \ref{pseudolinearity} below and the sketch of its proof. By Lascarness the independence theorem holds, so there are $a, b, c$ independent over $A$ so that $ab \models p \otimes p$, $bc \models p \otimes p$, $ac \models (p \otimes p)'$. There are also $c_{ab}$, $c_{bc}$ so that $\mathrm{SU}(c_{ab}/A) = k$,  $\mathrm{SU}(c_{bc}/A) = k$, $c_{ab}= \mathrm{Cb}(\mathrm{tp}(ab/c_{ab}))$, $c_{bc}= \mathrm{Cb}(\mathrm{tp}(bc/c_{bc}))$; additionally we may choose $c_{ab}$ and $c_{bc}$ so that $ac_{ab} \ind_{A b} cc_{bc}$. Let $c_{ac}=\mathrm{Cb}(ac/c_{ab}c_{ac})$; we show that $\mathrm{SU}(c_{ac}/A) \geq k$, indeed that  $\mathrm{SU}(c_{ac}/Ac_{ab}) \geq k$.  We see that $c_{bc} \subseteq \mathrm{bdd}(c_{ab}c_{ac})$, since $b$ and $c$ are interalgebraic over $c_{ab}c_{bc}$. But $c_{bc} \ind_{A} c_{ab}$, so  $\mathrm{SU}(c_{bc}/Ac_{ab}) = k$, so $\mathrm{SU}(c_{ac}/Ac_{ab}) \geq k$.

    We expect this is also true for $p$ a regular type, though we only really need the above for Proposition \ref{Mboneinfinity}, where $T$ is of finite rank and all regular types have $\mathrm{SU}(p) = 1$.
\end{remark}

The following is then a consequence of Corollary \ref{rank2}, and the above Remark \ref{independenceofpotimesp}:

\begin{cor}\label{Mbpsquared}

Let $p \in S(A)$ with $\mathrm{SU}(p) = 1$. Then $p$ is $k$-linear if and only if the maximum value of $F_{Mb}(p \otimes p)$ is equal to $k$, where $p \otimes p$ ranges over the types over $A$ realized by two independent realizations of $p$. So if $p$ is a Lascar type (e.g. if $A=\mathrm{acl}^{eq}(A)$ and $T$ is supersimple), and $p\otimes p $ is any type over $A$ realized by two independent realizations of $p$ over $A$, $p$ is $k$-linear if and only if $F_{Mb}(p \otimes p) = k$. 
\end{cor}

This just says that Theorem 3.3 (1) of \cite{Kim10}, the equivalence of $k$-linearity and being ``$k$-based" for a Lascar strong type $p$, is witnessed by two independent realizations of $p$; again, as stated above, we expect that its proof makes use of similar ideas. In fact, using Theorem 3.3 (1) of \cite{Kim10}, we can do even better:

\begin{fact}
Let $p \in S(A)$ with $\mathrm{SU}(p) =1$, and $n \geq 2$. Then  $p$ is $k$-linear if and only if the maximum value of $F_{Mb}(p^{(n)})$ is equal to $k$, where $p^{n}$ ranges over the types over $A$ realized by $n$ independent realizations of $p$. If $p$ is a Lascar type, and $p^{(n)}$ is any type over $A$ realized by $n$ independent realizations of $p$ over $A$, then $p$ is $k$-linear if and only if $F_{Mb}(p^{(n)}) = k$. 
\end{fact}

\begin{proof}
Direction (a) $\Rightarrow$ (b) of Theorem 3.3 (1) of \cite{Kim10} says that if $p$ is $k$-linear and $I = \{a_{i}\}_{i < \omega}$ is an $A$-indiscernible sequence of realizations of $p^{(n)}$, for $p^{(n)}$ any type realized by $n$ independent realizations of $p$, then $I= \{a_{k-1 +i}\}_{i < \omega}$ is a Morley sequence over $Aa_{< k}$. (Note that Theorem 3.3 (1) of \cite{Kim10} says this explicitly only when $p$ is a Lascar type, but the proof should extend to any complete type.)  Therefore, if $I$ is not a Morley sequence over $A$, $a_{k} \nind_{A} a_{< k}$ and  $F_{Mb}(p^{(n)}) \leq k$. Since $n \geq 2$, if $p$ is $k$-linear, there is some type $p^{(n)}$ realized by $n$ independent realizations of $k$ so that $F_{Mb}(p^{(n)}) \geq k$; if $p$ is additionally a Lascar type, this holds for any $p^{(2)}$ and thus any $p^{(n)}$, by Remark \ref{independenceofpotimesp}.

\end{proof}

When $p$ is minimal, note that this definition is equivalent replacing $\mathrm{SU}_{p}$ with just $\mathrm{SU}$.

In some classes of theories, $k=1$ is the only possible value for which a regular type can be $k$-linearity. This is shown by \cite{P96}, \cite{TW03} and is highly nontrivial, relying on the \textit{group configuration theorem} in stable or $\omega$-categorical simple theories. 

\begin{fact} \label{pseudolinearity}(\cite{P96}, \cite{TW03})
    Let $T$ be a stable theory and $p$ a regular type, or let $T$ be an $\omega$-categorical simple theory, and let $p$ be a regular type over a finite set. Then if $p$ is $k$-linear, $k=1$.
\end{fact}
\begin{proof}
    (Sketch) We sketch the $\mathrm{SU}$-rank $1$ case. Suppose $p$ is $k$-linear. Then as in Remark \ref{independenceofpotimesp}, we can get $a, b, c, c_{ab}, c_{bc}, c_{ac}$, where $a, b, c$ have $\mathrm{SU}$-rank $1$ over $A$, and $c_{ab}, c_{bc}, c_{ac}$ have $\mathrm{SU}$-rank $k$ over $A$; these will be arranged in a \textit{group configuration}, where any three of the $a, b, c, c_{ab}, c_{bc}, c_{ac}$ will be independent, except for $a, b, c_{ab}$; $b, c, c_{bc}$; and $a, c, c_{ac}$. By \cite{Hr92} in the stable case, \cite{TW03} in the $\omega$-categorical case, we then obtain a type-definable group of $\mathrm{SU}$-rank $k$ acting transitively on a type of $\mathrm{SU}$-rank $1$. In the $\omega$-categorical case, \cite{TW03} show that there can be no such group when $k > 1$; in the stable case, by Hrushovski's thesis, it is known that any group of $\mathrm{U}$-rank $k > 1$ acting on a type $p$ of $\mathrm{U}$-rank $1$ must be have $\mathrm{U}$-rank $1$, $2$ or $3$ (this proves a special case of the Borovik-Cherlin conjecture), and in the case where the $\mathrm{U}$-rank is $2$ or $3$ the type $p$ is nonorthogonal to a field, so can be shown not to be $k'$-linear for any $k'$ (see Example \ref{ACFMB} for a related observation).
\end{proof}

From Fact \ref{pseudolinearity}, Corollary \ref{Mbpsquared}, and Remark \ref{independenceofpotimesp} follows:

\begin{cor}\label{pseudolinearitymb}
Let  $p \in S(A)$ be a regular type in a stable theory, or a regular type over a finite set in an $\omega$-categorical simple theory. Then either there is no finite bound on $F_{Mb}(p \otimes p)$, where $p \otimes p$ ranges over all types over $A$ realized by two independent realizations of $p$ over $A$, or $p$ is linear.

Moreover, let $p \in S(A)$ be a regular \emph{Lascar} type, with $\mathrm{SU}(p) = 1$ in a stable theory or an an $\omega$-categorical supersimple theory, and let $p \otimes p$ be any type over $A$ realized by two independent realizations of $p$ over $A$. Then either $p$ is not $k$-linear for any $k$ and $F_{Mb}(p \otimes p) = \infty$, or $p$ is linear and $F_{Mb}(p\otimes p) =1$,

 \end{cor}
\begin{proof}
    The only case not yet clear from the facts quoted above is the case of a Lascar type $p(x) \in S(A)$ in an $\omega$-categorical supersimple theory where $A$ is not finite. (We include this case because it will be needed for Proposition \ref{Mboneinfinity}, showing that $F_{Mb}$ has values $1$ and $\infty$ when $T$ is $\omega$-categorical supersimple of finite rank.) We may clearly replace $A=\mathrm{bdd}(A)$. By supersimplicity, we may find some finite $A'_{0} \subset A$ over which $p(x)$ does not fork; we may also add the ($A_{0}$-definable, by $\omega$-categoricity) Lascar equivalence class of realizations of $p(x)$ over $A'_{0}$ to get a finite set $A_{0} \subseteq A$; then $p(x)$ will not fork over $A_{0}$, and $p(x) |_{A_{0}}$ will be a Lascar strong type over $A$ because all realizations of $p(x) |_{A_{0}}$ will be Lascar equivalent over $A'_{0}$, so over $A_{0} $ because $A_{0} \subset \mathrm{acl}^{{eq}}(A'_{0})$.  So by the independence theorem and the fact that $p(x)$ does not fork over $A_{0}$, every type $p|_{A_{0}}(x) \otimes p|_{A_{0}}(x) $ realized by two $A_{0}$-independent realizations of $p|_{A_{0}}(x)$ extends to a type $p(x) \otimes p(x)$ realized by two $A$-independent  realizations of $p(x)$, and any type of the form $p(x) \otimes p(x)$  restricts to a type of the form $p|_{A_{0}}(x) \otimes p|_{A_{0}}(x) $. So the case of $p(x)$ follows from the case of $p(x)|_{A_{0}}$, a Lascar type over a finite set, by Corollary \ref{Mbpsquared} and the below Claim \ref{Mbpreservedundernonforkingextension}.
\end{proof}

In Proposition \ref{Mboneinfinity} we will show that \textit{any} type $p$ in a supersimple theory of finite rank, whenever the conclusion of Fact \ref{pseudolinearity} holds, must have $F_{Mb}(p) =1$ or $F_{Mb}(p) = \infty$.

A good illustration of the case $F_{Mb}(p \otimes p) = \infty$ of Corollary \ref{pseudolinearitymb} is given by $T=ACF_0$, the theory of algebraically closed fields of characteristic zero. We saw in Example \ref{ACFMB} that $\Fmb {p} = \infty$, when $p \in S_2 ( \emptyset )$ is the generic type of affine $2$-space.


\begin{remark} \label{rankofcanonicalbaseinvariant}
    For an indiscernible sequence $I$ of realizations of $p \in S(A)$, the value $\mathrm{SU}(\mathrm{Cb}(q)/A)$ when $T$ is supersimple (or of $\mathrm{SU}_{p}(\mathrm{Cb}(q)/A)$ when $p$ is a regular type), where $ q \in S(B)$, $A \subseteq B = \mathrm{bdd}(B)$ is an extension of $p$ such that $I$ is a Morley sequence over $B$ of realizations of $q$, can be easily seen to be an invariant of $I$ that does not depend on $q$, without making reference to how many terms of $I$ are independent over $A$.
\end{remark}

When $T$ is simple but no longer supersimple, it will be of interest to distinguish indiscernible sequences, but $\mathrm{SU}(\mathrm{Cb}(q)/A)$ may no longer be defined as in the above remark. To circumvent this, we define some invariants based on the ``$\mathrm{G}$-rank," $\mathrm{G}(\mathrm{Cb}(q)/A)$, of a canonical base. For our applications, we will only need to define this for $\emptyset$-indiscernible sequences, though the definitions can easily be generalized to any base.

\begin{definition}
   \label{resolvable}
    
    (1) Let $I$ be an indiscernible sequence over $\emptyset$. Then $I$ is \emph{$\kappa$-resolvable} if there is some set $A \subset \mathbb{M}$ (i.e. of real elements), $|A| \leq \kappa$ so that $I$ is a Morley sequence over $A$.
   
    (2) Let $p \in S(B)$ be a type. Then $p$ is \emph{$\kappa$-resolvable} if there is some set $A \subset \mathbb{M}$, $|A| \leq \kappa$ and $ q \in S(AB)$, $q \supseteq p$ so that $q$ forks over neither $A$ nor $B$.

\end{definition}

\begin{definition}\label{grank}
    Let $ B \subset \mathbb{M}^{\mathrm{heq}} $. Then $G(B)$ (the ``\emph{$\mathrm{G}$-rank}" of $B$) is the least $\kappa$ so that there is $A \subseteq \mathbb{M}$, $|A| \leq \kappa$ with $B \subseteq \mathrm{bdd}(A)$.
\end{definition}

The following proposition, connecting these definitions, is an exercise:

\begin{prop}\label{resolvabilitygrank}
    Let $T$ be simple, and let $I$ be an indiscernible sequence. Then the following are equivalent:

    (1) The indiscernible sequence $I$ is $\kappa$-resolvable.

    (2) For some type $p \in S(B)$, so that $I$ is a Morley sequence in $p$ over $B$, $p$ is $\kappa$-resolvable.

    (2') For every type $p \in S(B)$, so that $I$ is a Morley sequence in $p$ over $B$, $p$ is $\kappa$-resolvable.

    (2'') The type $\mathrm{lim}^{+}(I/I) \in S(I)$ is $\kappa$-resolvable.

    (3) For some type $p \in S(B)$, so that $I$ is a Morley sequence in $p$ over $B$, $G(\mathrm{Cb}(p)) \leq n$.

    (3') For every type $p \in S(B)$, so that $I$ is a Morley sequence in $p$ over $B$, $G(\mathrm{Cb}(p)) \leq n$.

    (3'') $G(\mathrm{Cb}(\mathrm{lim}^{+}(I/I))) \leq n$

\end{prop}

For our application to the Koponen conjecture in Section \ref{Kop}, we will need the equivalence of (1), (2) and (2'); (3), (3') and (3'') are included to justify the analogy with the other invariants, and (3'') is included to obtain a canonical invariant distinguishing indiscernible sequences.

In Section \ref{nmdegfmb}, we noted that for $p \in S(A)$ in a simple theory, $\mathrm{nmdeg}(p) \leq F_{Mb}(p) $. The following notions, introduced in \cite{oberwolfachFdeg}, also generalize the degree of nonminimality:

\begin{definition}
Let $p \in S(A)$ with $\mathrm{SU}(p) > \lambda$. Let the \emph{$\lambda$-forking degree of $p$}, $F_{\lambda}(p)$, be the least $n$ so that there are $a_{0}, \ldots a_{n-1} \models p$, so that $p$ has a forking extension $q \in S(Aa_{0} \ldots a_{n-1})$ with $\mathrm{SU}(q) \geq \lambda$.
\end{definition}

It follows as in Remark \ref{findbigger} that $F_{\lambda}(p)$ is well-defined, and that

\begin{prop} \label{flambda}
    Let $p \in S(A)$ with $\mathrm{SU}(p) > \lambda$. Then $F_{\lambda}(p) \leq  F_{Mb}(p)$. So if $\mathrm{SU}(p)= \alpha$ is a limit ordinal and $\mathrm{lim}_{\lambda \to \alpha} F_{\lambda}(p) = \infty$, $F_{Mb}(p) = \infty$.
\end{prop}

\begin{example}
\label{vectorspace}

The following example of a non-one-based theory, every type of which is linear, appears in as Example 6.2.14 in \cite{P96}: let $T$ be the two-sorted theory of an infinite-dimensional vector space $V$ over an algebraically closed field $K$, in the language with the field operations, vector addition, and scalar multiplication. This is superstable, with a single type $p \in S^{1}_{V}
(\emptyset)$ with $\mathrm{U}(p)= \omega$. The type $p$ is a linear regular type, and the characterization of forking in \cite{P96} implies that for $a_{0}, \ldots a_{n} \models p$, $a_{n} \nind a_{0} \ldots a_{n-1}$ if and only if $a_{0}$ is in the linear span of $a_{0} \ldots a_{n-1}$; if $q \in S(a_{0} \ldots a_{n-1})$ is the generic of the linear span of $a_{0} \ldots a_{n-1}$, then $\mathrm{U}(q) = n$. It follows that for $n < \omega$, $F_{n}(p) = n$, so in particular, $\mathrm{lim}_{\lambda \to \omega} F_{\lambda}(p) = \infty$ and $F_{Mb}(p) = \infty$.

\end{example}

An open problem, mentioned in \cite{Wag00}, is whether there exists an $\omega$-categorical supersimple theory of infinite $\mathrm{SU}$-rank. If $T$ is $\omega$-categorical and supersimple, let $p(x) \in S(A)$ have $SU(p) = \alpha$, for $\alpha$ a limit ordinal; then $\mathrm{lim}_{\lambda \to \alpha} F_{\lambda}(p) = \infty$. To see this, assume, without loss of generality because $T$ is supersimple, that $A$ is finite, and suppose $\mathrm{lim}_{\lambda \to \alpha} F_{\lambda}(p) = n$ for $n < \omega$; then for $\lambda < \alpha$, there are $a_{0} \ldots a_{n} \models p$ with $\mathrm{SU}(a_{n}/a_{0} \ldots a_{n-1}A) > \lambda$. Since $\lambda$ is a limit ordinal, this implies there are infinitely many types in $x_{0}, \ldots, x_{n-1}, x_{n}$ over the finite set $S(A)$, contradicting $\omega$-categoricity. So, we've shown: 

\begin{cor}
\label{omegacategoricalsupersimple}
    Let $T$ be $\omega$-catgorical and supersimple, and $SU(p)$ a limit ordinal. Then $F_{Mb}(p) = \infty$.
\end{cor}
Note the similarity of this proof to that of Corollary 3.12 of Palacín \cite{Pal17}, exposited in Koponen \cite{Kop19}.

\section{$F_{Mb}$ can be finite, but larger than $1$, in a stable theory}\label{Mbfinitestable}
\numberwithin{theorem}{section} \setcounter{theorem}{0}
So far, we have only discussed examples of types $p \in S(A)$ where $p$ has the property that has algebraic forking (Definition \ref{algfork})--which fully characterizes the forking extensions of $p$, and implies $F_{Mb}(p) =1 $-- or 
 where $F_{Mb}(p) = \infty $. This leads us to ask: in a simple theory, under what conditions is it possible that $1 < F_{Mb}(p) < \infty$, and when $F_{Mb}(p)= 1$, when must $p$ have algebraic forking? 

 We first show that in a supersimple theory with finite rank, the first question is reducible to the following conjecture of Kim.

 \begin{conj}
    (Pseudolinearity conjecture, \cite{Kim10}.)

     Let $T$ be simple, and let $p \in S(A)$ be a Lascar type with $\mathrm{SU}(p) =1$. Then if $p$ is $k$-linear for some $k < \omega$, $p$ is linear.
 \end{conj}

 \begin{thm}
      \label{Mboneinfinity}
 Let $T$ be a supersimple theory of finite rank. Then the following are equivalent:

 (1) $T$ satisfies the conclusion of the pseudolinearity conjecture.

 (2) For every $p \in S(A)$, either $F_{Mb}(p) =1$ or $F_{Mb}(p) = \infty$.
  \end{thm}

Here, (2) $\Rightarrow$ (1) is just Corollary \ref{Mbpsquared} and the definition of $k$-linearity.

Note that in (2) of this proposition, we may assume $p$ is finitary\footnote{Recall a hyperimaginary is finitary if it is the class of a finite tuple modulo
a type-definable equivalence relation.}. We will see that \ref{Mboneinfinity} is essentially a corollary of a powerful classical result of geometric stability theory, Fact \ref{equidominancetheorem}. To state the form of the result we will need we first require a definition: 

\begin{definition}
    \label{equidominance}

    Two types $p, q \in S(A)$ are \emph{equidominant} over $A$ if there are realizations $a$ and $b$ of $p$ and $q$ respectively if, for any set $B$, $a \nind_{A} B$ if and only if $b \nind_{A} B$. Two complete types $p$, $q$ over any sets are \emph{equidominant} if there is some set $A$ containing the domains of $p$ and $q$ so that there are nonforking extensions of $p$ and $q$ to $A$ which are equidominant over $A$.
\end{definition}

\begin{fact}
   \label{equidominancetheorem}
    
    Let $T$ be supersimple of finite rank, and $p \in S(A)$ a finitary type. Then $p$ is equidominant with the type, over some set $B$, of a finite $B$-independent set of realizations of minimal types over $B$.
\end{fact}

The previous result follows from Corollary 5.2.19, Proposition 5.1.12, and Lemma 5.2.11 of \cite{Wag00}, the latter two of which imply that the regular types of finite $\mathrm{SU}$-rank are exactly the types of $\mathrm{SU}$-rank $1$.

The theorem will then follow from the following three claims and Fact \ref{equidominancetheorem}:

\begin{claim}
\label{Mbpreservedundernonforkingextension}
Let $p$ and $q$ be complete types so that $q$ is a nonforking extension of $p$. Then $F_{Mb}(p) = F_{Mb}(q)$.

\end{claim}

\begin{claim}
\label{Mbpreservedunderequidominance}
Let $p, q \in S(A)$ be equidominant over $A$. Then $F_{Mb}(p) = F_{Mb}(q)$.

\end{claim}

\begin{claim}
\label{Mboneorinfinityforproductofminimaltypes}

Let $p \in S(A)$ be the type of an $A$-independent set of realizations of minimal types over $A$, and assume the conclusion of the pseudolinearity conjecture holds. Than either $F_{Mb}(p)=1$ or $F_{Mb}(p) = \infty$.

\end{claim}

We first prove Claim \ref{Mbpreservedundernonforkingextension}; this only requires simplicity. Let $F_{Mb}(p) = n$, and choose an $A$-indiscernible sequence $I$, consisting of realizations of $p$, any $n$ terms of which are independent over $A$, but which is not a Morley sequence over $A$. By the chain condition in simple theories, and the fact that $q$ does not fork over $p$, there is some $I' \equiv_{A} I$, with $I' \ind_{A} B$, consisting of realizations of $q$ and indiscernible over $B$. But then any $n$ terms of $I'$ are independent over $B$, and $I'$ is not a Morley sequence over $B$, so $F_{Mb}(q) \geq n$. Now suppose $F_{Mb}(q) = n$; it suffices to show $F_{Mb}(p) \geq n$. Let $I$ be an indiscernible sequence over $B$ of realizations of $q$, that is not a Morley sequence over $B$, but so that any $n$ terms of $I$ are independent over $B$. Since $q$ is a nonforking extension of $p$, so each term of $I$ is independent from $B$ over $A$, it follows from the forking calculus that every $n$ terms of $I$ (which consists of realizations of $p$) are independent over $A$. But $I=\{a_{i}\}_{i < \omega}$ is not a Morley sequence over $A$, because if it were, for each $N < \omega$,  $\{a_{Ni} \ldots a_{Ni + (N-1)}\}_{i < \omega}$ would be a Morley sequence over $A$ that remains indiscernible over $B$, so the first term of this would be independent from $B$ over $A$ by Kim's lemma and symmetry. So $I \ind_{A} B$ (see, say, the proof of the independence clause of Proposition 3.21 of \cite{KR17}), and then since $I$ is assumed to be a Morley sequence over $A$, it will be a Morley sequence over $B$, a contradiction.  Thus $F_{Mb}(p) \geq n$.

To prove Claim \ref{Mbpreservedunderequidominance}, supppose $F_{Mb}(p) = n$; by symmetry, it suffices to show $F_{Mb}(q) \geq n$. Let $I = \{a_{i}\}_{i < \omega}$ be an indiscernible sequence over $A$, $a_{i} \models p$, $a_{n-1} \ind_{A} a_{< n-1}$, $I$ not a Morley sequence. We may then find $I' = \{a'_{i}\}_{i < \omega}$, $a'_{i} \models q$, so that for each $i < \omega$, $a_{i}$ forks over $A$ with exactly the same sets as $a'_{i}$.  Then it is easily seen by successive replacements that $a'_{n-1} \ind_{A} a'_{< n-1}$, and $I'$ is not a Morley sequence. So $F_{Mb}(q) \geq n$.

Finally, it remains to prove Claim \ref{Mboneorinfinityforproductofminimaltypes}. There are two cases, first where all of the independent minimal types are orthogonal, and second where some two are not orthogonal. In the first case, this follows from Proposition \ref{typeswithalgebraicforkinghavembone} and the following subclaim:

\begin{subclaim}
    Let $p(x), q(y) \in S(A)$ be orthogonal types and $r(x,y)$ the type over $A$ of some (necessarily independent) realizations of $p$ and $q$. Then $F_{Mb}(r) = \mathrm{max}(F_{Mb}(p), F_{Mb}(q))$. 
\end{subclaim}

\begin{proof}
   Obviously $F_{Mb}(r) \geq \mathrm{max}(F_{Mb}(p), F_{Mb}(q))$ (take a witness to $F_{Mb}(p)$, and extend the terms independently to realizations of $r$; then extract an $A$-indiscernible sequence.) To see   $F_{Mb}(r) \leq \mathrm{max}(F_{Mb}(p), F_{Mb}(q))$, let $\{a_{i}b_{i}\}_{i < \omega}$ be an $A$-indiscernible sequence of realizations of $r$ that is not a Morley sequence. It suffices to show that either $I = \{a_{i} \}_{i < \omega}$ or $J = \{b_{i} \}_{i < \omega}$ is not a Morley sequence over $A$, because, supposing without loss of generality that $I$ is not a Morley sequence, $n+1$ terms of it will be dependent for $n \leq F_{Mb}(p)$, so $n \leq \mathrm{max}(F_{Mb}(p), F_{Mb}(q)) + 1$ terms of $\{a_{i}b_{i}\}_{i < \omega}$ will be dependent. If both $I$ and $J$ are Morley sequences over $A$, then (see, say, \cite{P96}) $I \ind_{A} J$, so $\{a_{i}b_{i}\}_{i < \omega}$ will be a Morley sequence over $A$, a contradiction.
\end{proof}

This proves the case where the minimal types are orthogonal. When the minimal types are nonorthogonal, we may assume, by the subclaim, that $p$ is realized by an $A$-independent set $\{a_{i}\}_{i < n}$ for $n \geq 2$ where the $p_{i}= \mathrm{tp}(a_{i}/A)$ all belong to the same nonorthogonality class; we may assume $A = \mathrm{bdd}(A)$ (so we can apply the independence theorem). So there is some set $B$ with $\mathrm{bdd}(B) = B$ and nonforking extensions $q_{i}$ of $p_{i}$ to $B$ so that no two of the $q_{i}$ are weakly orthogonal, and by the independence theorem we may additionally chose $B$ so that for each $i$, $a_{i} \models q_{i}$, and $B \ind_{A} \{a_{i}\}_{i < n}$.  By Claim \ref{Mbpreservedundernonforkingextension},  it suffices to show that $F_{Mb}(\mathrm{tp}(\{a_{i}\}_{i < n} ))$ is either $1$ or $\infty$. By Claim \ref{Mbpreservedunderequidominance} we can replace each of the $a_{i}$ by tuples interalgebraic over $B$, thereby assuming that each $a_{i}$ realizes the same type over $B$. But then by the conclusion of the pseudolinerity conjecture either this type is one-based, so $F_{Mb}(\mathrm{tp}(\{a_{i}\}_{i < n} )) = 1$, or it is not $k$-pseudolinear for any $k$, so because $n \geq 2$, by Corollary \ref{Mbpsquared},  $F_{Mb}(\mathrm{tp}(\{a_{i}\}_{i < n} )) = \infty$. 

This completes the proof of Theorem \ref{Mboneinfinity}. It follows from Corollary \ref{pseudolinearitymb} that:

\begin{cor}\label{infMb}
Let $T$ be either superstable of finite rank or supersimple of finite rank and $\omega$-categorical. Then for $p \in S(A)$, either $F_{Mb}(p)=1$ or $F_{Mb}(p) = \infty$.

\end{cor}

Note that this proof does not tell us that types $p$ with $F_{Mb}(p) = 1$ have algebraic forking. For example, it is not known that algebraic forking is preserved under nonforking extension, so the analogue of Claim \ref{Mbpreservedundernonforkingextension} may not hold.  We may then ask:

\begin{question}
Let $T$ be supersimple of finite rank and let $p \in S(A)$, $F_{Mb}(p) = 1$. Does $p$ have algebraic forking?

\end{question}

If the answer to this question is positive, by Propositions \ref{typeswithalgebraicforkinghavembone} and \ref{Mboneinfinity}, then, assuming the pseudolinearity conjecture, in a supersimple theory with finite rank, the types $p$ with $F_{Mb}(p) $ finite will be exactly those with algebraic forking.

\begin{example}
\label{Mbonenonalgebraic}

There are superstable theories $T$ so that for all $p \in S(A)$, $F_{Mb}(p) = 1$, and which are not one-based, and therefore have types without algebraic forking. In particular, there are superstable theories, which are not one-based, which have the following property (See e.g. \cite{P96}, where stable theories with the below property are called ``trivial'').

\begin{definition}
\label{geometricallytrivialforking} A simple theory $T$ has \emph{geometrically trivial forking} if $a \nind_{A} bc$ implies $a \nind_{A} b$ or $a \nind_{A} c$. A regular type $p \in S(A)$ in a simple theory has \emph{trivial pregeometry} if for $a, b, c \subseteq p(\mathbb{M})$, $a \nind_{A} bc$ implies $a \nind_{A} b$ or $a \nind_{A} c$.

\end{definition}
The following is an obvious consequence of the definitions:

\begin{prop}
If $T$ is a simple theory with geometrically trivial forking, then $F_{Mb}(p)=1$ for all complete types $p$. If $p$ is a regular type in a simple theory with trivial pregeometry, then $F_{Mb}(p)=1$.

\end{prop}

The $\omega$-stable free pseudoplane, Example 4.6.1 of \cite{P96}, is an example of a superstable theory with geometrically trivial forking which is not one-based. This is the theory of undirected graphs without loops, each vertex of which has infinite degree; it is the case $n =1$ of the construction of Theorem  \ref{Mbn} below. As noted in example 7.2.10 of \cite{P96}, the unique complete type $p$ over the empty set is regular, so has trivial pregeometry; however, the proof from \cite{P96} that $T$ is not one-based should show that $p$ does not have algebraic forking.

The Farey graph is another example of an $\omega$-stable theory (Corollary 1.15 of \cite{DKG23}) which is geometrically trivial (Lemma 1.14 of \cite{DKG23}) but is not one-based (it has weak elimination of imaginaries, Corollary 11.15 of \cite{DKG23}, but $A \ind_{C} B$ does not coincide with $\mathrm{acl}(AC) \cap \mathrm{acl}(BC) =\mathrm{acl}(C)$ for the real algebraic closure, Lemma 11.4 of \cite{DKG23}.)

\end{example}
However, even though there are examples of types $p$ in a supersimple theory with $F_{Mb}(p) < \infty$ without algebraic forking, there are no known examples of types $p$ in a supersimple theory with $F_{Mb}(p)$ other than $1$ or $\infty$.

\begin{question}
   \label{Mboneinfinitysupersimple}
    
    Let $T$ be supersimple, and $p$ a complete type. Is $F_{Mb}(p)$ either $1$ or $\infty$?
\end{question}

We have also given two conditions under which a regular type $p$ in a supersimple theory must have $F_{Mb}(p)=1$. Do these account for all examples?

\begin{question}
    Let $T$ be a supersimple theory, and $p$ a regular type with $F_{Mb}(p)=1$. Must $T$ have either algebraic forking, or trivial pregeometry?
\end{question}

Though we do not have an answer to the question of whether $F_{Mb}(p)$ is $1$ or $\infty$ for $p$ a type in a supersimple (or even superstable) theory (Question \ref{Mboneinfinitysupersimple}), we discuss this question further, in particular the case of regular types. Regular types are an important special case in supersimple theories, because of the fact (see e.g.  Corollary 5.2.19 of \cite{Wag00}, and compare Fact \ref{equidominancetheorem}) that every type in a supersimple theory is equidominant to the type of a set of independent realizations of regular types. More precisely, if all regular types $p$, and types $p$ of independent sets of realizations of a \textit{linear} regular type, have $F_{Mb}(p) = 1$ or $F_{Mb}(p) = \infty$, we can use the proof of Proposition \ref{Mboneinfinity} to prove that all types $p$ in superstable theories have $F_{Mb}(p) = 1$ or $F_{Mb}(p) = \infty$.

However, unlike types $p$ with $\mathrm{SU}(p)=1$, which have $F_{Mb}(p) =1$ because they have algebraic forking, it is not necessary that a regular type $p$ has $F_{Mb}(p) =1$ (the unique type over the empty set in the vector space sort in Example \ref{vectorspace}, for example.) 

We give some characterizations of $F_{Mb}(p)$, where $p$ is a regular type. To motivate this characterization, we reproduce the following definition from \cite{Kop11}, \cite{AKop15}, which generalizes geometric triviality:

\begin{definition}
   Let $T$ be simple. Then $T$ has \emph{$n$-degenerate dependence} if $a \nind_{C} B$ implies that, for some $B_{0} \subset B$ with $|B_{0}| \leq n$,  $a \nind_{C} B_{0}$.
\end{definition}

This suggests the following definition, where we assume that $B$ is a sufficiently saturated model, rather than just any set. For the purposes of our chracterization of $F_{Mb}(p)$ for $p$ regular, we present this relative to a type in the theory $T$, fixing the base set $C$ as well as taking $a$ to be a single realization of $p$.

\begin{definition}
    Let $T$ be simple, and $p(x) \in S(C)$.  Then $p$ has \emph{weakly $n$-degenerate dependence over $C$} if for $a \models p(\mathbb{M})$ and $|C|^{+}$-saturated models $M \supset C$, if $a \nind_{C} p(M)$ then for some $B_{0} \subset p(M)$ with $|B_{0}| \leq n$, $a \nind_{C} B_{0}$.
\end{definition}

\begin{theorem}
    Let $T$ be simple, and let $p \in S(C)$ be a regular type. Then the following are equivalent:

    (1) $F_{Mb}(p) \leq n$.
    
    (2) $p$ has weakly $n$-degenerate dependence over $C$.
    
    (3) For every forking extension $q$ over $p$, the solution set of $q$ is covered by finitely many sets of the form $\mathrm{cl}_{p}(a_{0} \ldots a_{n-1})$, where $a_{0}, \ldots ,a_{n-1} \models p$.
    
\end{theorem}

\begin{proof}
    (3 $\Rightarrow$ 2) Suppose (2) fails, so there is some $a \models p$ and $|C|^{+}$-saturated model $M$ with $a \nind_{C} p(M)$, and $a \ind_{C} B_{0}$ for every $B_{0} \subset p(M)$ with $|B_{0}| \leq n$. It suffices to show that $q = \mathrm{tp}(a/p(M))$, which is a forking extension of $p$, fails the conclusion of (3). Suppose otherwise, so there are $B^{0}_{0}, \ldots B_{0}^{k} \subset p(M)$ with $|B^{i}_{0}| \leq n$ for $0 \leq i \leq k$, so that $q(\mathbb{M}) \subseteq \cup^{k}_{i=0} \mathrm{cl}_{p} (B^{i}_{0})$. Then $a \in q(\mathbb{M})$, so $a \in \mathrm{cl}_{p}(B^{i}_{0})$ for some $0 \leq i \leq k$, so $a \nind_{C} B^{i}_{0}$, a contradiction.

    (1 $\Rightarrow$ 3) Suppose the negation of (3) holds, for some forking extension $q \in S(C')$ of $p$. We first construct inductively a sequence $\{a_{i}\}_{i < \omega}$, $a_{i} \models q$, so that any $n+1$ terms of $\{a_{i}\}_{i < \omega}$ are independent over $C$. Suppose $\{a_{i}\}_{i < k}$ is a sequence, though of length $k$, satisfying these same requirements. We want to find $a_{k} \models q$ that is independent from every $n$-element subset of $\{a_{i}\}_{i < k}$, thereby preserving the requirements.  Let $B^{0}_{0}, \ldots B_{0}^{m} \subset p(M)$ enumerate the  $n$-element subsets of $\{a_{i}\}_{i < k}$. Then by assumption, $q(\mathbb{M}) \not\subseteq \cup^{m}_{i=0} \mathrm{cl}_{p} (B^{i}_{0})$; choosing $a_{k} \in q(\mathbb{M}) \backslash \mathrm{cl}_{p} (B^{i}_{0})$ will give us the desired $a_{k}$. 

  Now choose a $C'$-indiscernible sequence $\{a'_{i}\}_{i < k}$ with the same EM-type over $C$ as $\{a_{i}\}_{i < k}$. The condition that any $n+1$ terms are independent over $A$ is type-definable (see Subclaims \ref{typedefinablemorley} and \ref{nindependencetypedefinable}), so $\{a'_{i}\}_{i < k}$ also has this condition. However, $\{a'_{i}\}_{i < k}$ is not a Morley sequence, since if $\{a'_{i}\}_{i < \omega}$ were a Morley sequence over $C$, by Kim's lemma and symmetry $a'_{0} \ind_{C} C'$, contradicting the fact that $a'_{0} \models q$, a forking extension of $p$. So since $\{a'_{i}\}_{i < k}$ is an $n+1$-independent sequence of realizations of $p$ that is not a Morley sequence, $F_{Mb}(p) > n$.

  (1 $\Rightarrow$ 2). Suppose that $p$ has $n$-degenerate independence over $C$, and that $F_{Mb}(p) > n$. Then there is a $C$-indiscernible sequence $\{a_{i}\}_{i< \omega}$, $a_{i} \models p$, so that for $N > n$, $a_{0}, \ldots a_{N-1}$ form an independent set over $C$ but $a_{i} \nind a_{0} \ldots a_{N-1}$ for $i > N-1$. The sequence $\{b_{i}\}_{i < \omega}$ where $b_{i}=a_{i+N-1}$ has, by $C$-indiscernibility of $\{a_{i}\}_{i< \omega}$, the property that $b_{0}, \ldots b_{N-1}$ forms an independent set over $C$, and is indiscernible over $C' = Ca_{0} \ldots a_{N-1} $, so we may choose some $|C|^{+}$-saturated model $M \supset C'$ so that  $\{b_{i}\}_{i < \omega}$ remains indiscernible over $Cp(M)$. By $n$-degenerate independence and indiscernibility over $p(M)$, there is then some $|B_{0}| \subset p(M)$ with $B_{0} \leq n$ so that $b_{i} \in \mathrm{cl}_{p}(B_{0})$, in particular for $0 \leq i \leq N-1$. But because $SU_{p}(B_{0}) \leq n$, this contradicts independence of $b_{0}, \ldots b_{N-1}$.

\end{proof}
The direction   (3 $\Rightarrow$ 1) of the above suggests a new way of obtaining $n$-independent indiscernible sequences that are not Morley sequences, possibly giving a strategy for answering Question \ref{Mboneinfinitysupersimple}; note that its proof uses a different construction method than the arguments in the results on canonical bases in the previous section (Proposition \ref{dependencecanonicalbases}, Corollary \ref{rank2}), where a Morley sequence is taken in a forking extension of $p$. 

We now give our first examples of types $p$ with $1 < F_{Mb}(p) < \infty$. In
general, it seems quite difficult to construct such examples - for instance,
since our constructions in Section \ref{nm>} are either superstable or
$\omega$-categorical supersimple of finite rank, by Corollary \ref{infMb}
every type $p$ in these theories has either $F_{Mb}(p) = 1$ or $F_{Mb}(p) =
\infty$. We show that for any $n < \omega$, it is possible in a stable theory
that $F_{Mb}(p) = n $.

\begin{theorem}
   \label{Mbn}

    Let $n < \omega + 1$. Then there is an stable theory $T $ with a type $p$ so that $F_{Mb}(p) = n$, and so that,
    for all types $q$, $F_{Mb}(q) \leq n$.
\end{theorem}

Fix $n < \omega$. We define $T$, based on the free projective planes
discussed by Hyttinen and Paolini (\cite{HP21}); for $n =1$ this will
coincide with the free $\omega$-stable pseudoplane (see e.g. \cite{P96}). In
fact, this theory will itself work for $n = 2$, and we expect that the open
generalized $n$-gons of \cite{AT23} will work for any value of $n$. However,
these examples are more complicated than necessary; for example, in the free
projective planes of \cite{HP21}, the algebraic closure is determined not
just by the requirement that the incidence relation has no finite subgraphs
of minimum degree at least $n+1$, but also the requirement that there  is
exactly one line incident to any two points, and vice versa.
We only need a
requirement of the first kind to obtain the desired theory. So we work in a vocabulary
with a single binary relation and study undirected graphs to build a family of theories $T_n$.

\begin{definition}\label{newopen} For our fixed $n$,
\begin{enumerate}
\item An undirected graph is \textit{open} if every finite subgraph $B$ has
    an element with degree in $B$ at most   $n$.

\item For $A \subseteq B$ open graphs and $C$ with $A\subseteq C \subseteq
    B$ finite, $C$ is \textit{closed} over $A$ if every point of $C
    \backslash A$ has at least $n+1$ neighbors in $C$.  If not, it is open
    over $A$.
    \item Write $A \leq B$ if there is no $C \subseteq B$, $C$ not
    contained in $A$, which is closed over $A$.
    \end{enumerate}
\end{definition}

We borrow the $\leq$ terminology from \cite{HP21}.
Observe that if $C \leq A$, $C \leq B$ and $A$ and $B$ are freely amalgamated
over $C$, $C \leq AB$. This is because, under these assumptions, if $D \subseteq AB$ is
closed over $C$, then so is $D \cap A$ and $D \cap B$, and if $D$ contains
points not in $C$ then so does one of these two. By the same reasoning, if
$A$ and $B$ are open then $AB$ is an open graph.

%

\begin{notation}\label{Tndef} $T_n$ is the theory of
infinite open undirected graphs, every $n$ nodes of which have infinitely
many common neighbors.
\end{notation}


By varying the classic construction of Hall (\cite{HP21}), there is an
infinite model of $T_n$. At stage $0$, start with an infinite independent set
$I$, and at stage $k+1$, for each $n$ nodes of stage $k$, add infinitely many
nodes connected to each of those nodes and no others. The graph is open
because at each stage the new points have degree $n$.

We rely on Fact~\ref{HP}, whose proof is a straightforward generalization of
the content of the proof of Theorem 1.1 of \cite{HP21}:

\begin{fact}\label{HP}
    Let $M \models T_n$ be saturated, and let $A, B$ be small open undirected
    graphs so that $A \leq M$, $A \leq B$. Then there is an embedding $\iota: B \to M$,
    which is the identity on $A$, so that $\iota(B) \leq M$.
\end{fact}

By a standard back-and-forth argument, we then see that:
\begin{cor} Each $T_n$ is complete. Moreover, for $\mathbb{M} \models T_{n}$ a sufficiently saturated model of $T_{n}$, two isomorphic sets $A \leq \mathbb{M}$, $B \leq \mathbb{M}$ have the same type.
\end{cor}

In order to show each $T_n$ is stable we need further information.

\begin{lemma}\label{aclopen} If $\mathbb{M}$ is a sufficiently saturated model of $T_n$
then an open graph $A$ with $A \subseteq \mathbb{M}$  is algebraically closed
if and only if $A \leq \mathbb{M}$.
\end{lemma}
 \begin{proof} If  $A \leq\mathbb{M}$ then $A\leq B$ for any $B$, $A\subseteq B \subseteq
 \mathbb{M}$.
 As noted before Notation~\ref{Tndef}, we can then iterate  free
 amalgamations over $A$ to get infinitely many copies of $B$ in $\mathbb{M}$.
For the converse, it suffices to show that if $A$ is an open undirected
graph, and $A \subsetneq B$ with $B\backslash A$ a finite set such that $B$ is closed over
$A$, then there cannot 
be infinitely many distinct $B_{i}$  isomorphic to $B$ over $A$. For this note that by
compactness we can go from infinitely many to uncountably many and then by
the $\Delta$-system lemma (taking sufficiently as $\aleph_2$) assume that $\{B_{i}\}_{i < \omega}$ forms a countable
sequence of sets that are disjoint and isomorphic over a common $D \supseteq A$. Since $B$ is closed over $A \subset D \subset B$, we can find a finite set $D_{0} \subseteq D$ such that for all $i$ and $b \in B_{i} \backslash D$, $b$ has at least $n + 1$ neighbors in $B_{i} \cup D_{0}$. Then we see that $F = D_{0} \cup \cup^{n}_{i = 0} (B_{i} \backslash D)$ is a finite subset of $\mathbb{M}$ such that each point of $F$ has at least $n+1$ neighbors in $F$, contradicting that $\mathbb{M}$ is an open graph.

\end{proof}

\begin{notation}\label{indnot}
Let $A, B, C \subseteq \mathbb{M}$ be sets, where $C \leq \mathbb{M}$. Let
$\ind$ denote forking independence and $A \ind_{C}^{\otimes} B$ denote
$\mathrm{acl}(ABC) = \mathrm{acl}(AC) \cup \mathrm{acl}(BC)$ and
$\mathrm{acl}(AC) \cup \mathrm{acl}(BC)$ is the free amalgamation of
$\mathrm{acl}(AC) $ and $\mathrm{acl}(BC)$ over $C$ (so in particular, $\mathrm{acl}(AC) \cap \mathrm{acl}(BC) = C$.
\end{notation}

Since forking independence is  stationary over models, Lemma~\ref{indeq}
implies  $T$ is stable. This is similar to the characterization of forking in
free projective planes in \cite{HP21}, but we include it to illustrate a
property of forking that may not be familiar from other examples of stable
theories.

\begin{lemma}\label{indeq} $A \ind_{C} B$  if and only if $A
\ind_{C}^{\otimes} B$.
\end{lemma}


\begin{proof}
First, if $A \ind^{\otimes}_{C} B$ with $C \leq M$, then $A \ind_{C} B$
because clearly, $\mathrm{tp}(A/ CB)$ extends to a $C$-invariant global type.
Now suppose that $A \ind_{C} B$. Then, $A \ind_{C} \mathrm{acl}(BC)$, so we
may assume $C \leq B$ and $B$ is algebraically closed. Now, let $A' =
\mathrm{acl}(AC)$ and suppose that $A \nind_{C}^{\otimes} B$. We show that $A
\nind_{C} B$. We may produce a $C$-indiscernible sequence $\{B_{i}\}_{i <
\omega}$ with $B_{0}= B$ so that, for $i < \omega$, $B_{i} \ind_{C}^{\otimes}
\{B_{j}\}_{j < i}$; for $p(X, B) = \mathrm{tp}(A/B)$, it suffices to show
that $\cup \{p(X, B_{i})\}_{i < \omega}$ is inconsistent. Suppose otherwise;
then we may in fact assume that $\{B_{i}\}_{i < \omega}$ is indiscernible
over $A'$. We then get a contradiction from the following claim: \end{proof}

\begin{claim}\label{forcontra} (Here we still assume $A \nind_{C}^{\otimes} B$ with $C \leq B$, $B$
algebraically closed.) Let $\{B_{i}\}_{i < \omega}$ be a
$A'=\mathrm{acl}(AC)$-indiscernible sequence with $B_{0} = B$.
Then $B_{n} \nind_{C}^{\otimes} \{B_{i}\}_{i < n}$.
\end{claim}

\begin{proof} We may assume that $B_{i} \cap B_{j} = C$ for $i \neq j < \omega$;
then clearly, $A' \cap B = C$. It will suffice to show that
$\mathrm{acl}(\{B_{i}\}_{i < n+1}) \supsetneq \cup_{i < n+1} B_{i}$. If $A
\nind_{C}^{\otimes}B$, either $\mathrm{acl}(A'B) = A'B$ and $A'$ and $B$ are
not freely amalgamated over $C$, or  $\mathrm{acl}(A'B) \supsetneq A'B$. In
the first case, there is some $a \in A' \backslash C$ with a neighbor $b$ in
$B \backslash C$. Then $\{a\}$ will have $n+1$ neighbors in $\{B_{i}\}_{i <
n+1}$, so will belong to  $\mathrm{acl}(\{B_{i}\}_{i < n+1})$ but will not
belong to $\cup_{i < n+1} B_{i}$. Therefore, $B_{n} \nind_{C}^{\otimes} \{B_{i}\}_{i < n}$. In the second case, since $\{B_{i}\}_{i <
\omega}$ and thus $\{\mathrm{acl}(A'B_{i})\}_{i < \omega}$ is indiscernible
over $A'$, there is some fixed $D \supseteq A'$ so that
$\mathrm{acl}(A'B_{i}) \cap \mathrm{acl}(A'B_{j}) = D$ for $i \neq j \leq
\omega$, and so that $\{B_{i}\}_{i < \omega}$ is indiscernible over $D$. As
we are in the second case, there are also $\{F_{i}\}_{i < \omega}$, so that
$F_{i} \subseteq \mathrm{acl}(A'B_{i}) \backslash A'B_{i}$ is finite and
closed over $A'B_{i}$ for $i < \omega$, and so that
$\mathrm{acl}(A'B_{i})F_{i} \equiv \mathrm{acl}(A'B_{j})F_{j}$ for $i \neq j
< \omega$. Suppose first that for some (every) $i < \omega$, $F_{i} \subseteq
D$. Each of the $F_{i}$ must have neighbors in $B_{i}$, since $F_i$ is closed
over $A'B_{i}$ and $A'$ is algebraically closed {\bf and so $A'\leq \mathbb{M}$ by
Lemma~\ref{aclopen}}. So if $F_{i} \subseteq D$, then $D$ and $B$ are not
freely amalgamated over $C$, and $B_{n} \nind_{C}^{\otimes} \{B_{i}\}_{i <
n}$ by the first case. Otherwise, for some (every) $i < \omega$, $F_{i}
\not\subseteq D$. There is some finite $D_{0} \subseteq D$ so that, for all
$i < \omega$, $D_{0}$ consists of all neighbors of $F_{i} \backslash D$ in
$D$. Then $F = D_{0} \cup^{n}_{i = 0} (F_{i} \backslash D)$ is a finite set,
not contained in $\cup_{i < n+1} B_{i}$, so that every point of $F$ has at
least $n+1$ neighbors in $F \cup (\cup_{i < n+1} B_{i})$: each point of
$F_{i}$ has $n+1$ neighbors that are either in $D \supseteq A$ (so in $D_{0}
\subseteq F$), in $F_{i} \backslash D$, or in $B_{i}$, while each point of
$D_{0}$ has a neighbor in each of the $F_{i} \backslash D \subseteq F$, and
therefore has at least $n+1$ neighbors in $F$. So $\mathrm{acl}(\{B_{i}\}_{i
< n+1}) \supsetneq \cup_{i < n+1} B_{i}$, because $F \subseteq
\mathrm{acl}(\{B_{i}\}_{i < n+1})$, and again $B_{n} \nind_{C}^{\otimes} \{B_{i}\}_{i < n}$.

\end{proof}
This completes the proof of Claim~\ref{indeq}. We have proven that each
$T_n$ is stable, with $A \nind_{C}^{\otimes} B$ characterizing forking over
algebraically closed sets. We now complete the proof of Theorem~\ref{Mbn}.

\begin{proof}
We may assume that $A$ is algebraically closed and that $p=\mathrm{tp}(B/A)$
for $A \leq B$ and $B$ algebraically closed. Let $\{B_{i}\}_{i < \omega}$ be
an $A$-indiscernible sequence, with $B_{i} \models p$ for $i < \omega$, and
suppose it is not a Morley sequence over $A$. We must show that $B_{n}
\nind_{A} \{B_{i}\}_{i < n}$. But there is some $N$ so that  $B_{j}
\nind_{A}^{\otimes} \{B_{i}\}_{i < N}$ for all $j \geq N$, because
$\{B_{i}\}_{i < \omega}$ is not a Morley sequence over $A$. Since $\{B_{i +
N}\}_{i < \omega}$ is indiscernible over $\mathrm{acl}(\{B_{i}\}_{i < N})$,
by Claim~\ref{forcontra}, $B_{n+N} \nind_{A}^{\otimes} \{B_{i+N}\}_{i < n}$,
so $B_{n} \nind_{A} \{B_{i}\}_{i < n}$, as desired.

Finally, we find a type $p \in S(\emptyset)$ with $F_{Mb}(p) \geq n$, so
$F_{Mb}(p) = n$. Let $A= \{a_{i}\} \cup \{*\}$ be an undirected graph with
the $a_{i}$, $*$ distinct, where there is an edge between each of the $a_{i}$
and $*$ but no edges between the $a_{i}$. Then there is an embedding $\iota:
A \to \mathbb{M}$ with algebraically closed image; for $i < \omega$, let
$b_{i} = \iota(a_{i})$. Then $\{b_{i}\}_{i<\omega}$ is an
$\emptyset$-indiscernible sequence, so that $b_{n-1} \ind \{b_{i}\}_{i <
n-1}$ but $b_{n} \nind \{b_{i}\}_{i < n}$. Therefore, $F_{Mb}(p) \geq n$, as
desired, proving Theorem~\ref{Mbn}.
\end{proof}

\section{$\nmdeg$ and $F_{ind}$ in stable and simple theories}\label{nm>}
\numberwithin{theorem}{subsection} \setcounter{theorem}{0}

In this section, we give  examples showing that the degree of nonminimality
of types can be any $n \in \m N$ for $\omega$-categorical theories that
 are supersimple or strictly stable,   and for almost strongly minimal theories.
This solves the analog of a problem of the second author and Moosa in the
context of simple theories. Examples of Freitag and Moosa (\cite[discussion
following Theorem C, page 5]{freitag2023bounding}, \cite{freitag2023degree}
and \cite{freitag2022any}) previously showed that the degree of nonminimality
can be as large as two in a stable theory.

Before giving the constructions, we explain the results of
\cite{freitag2023degree} and its implications for constructing such examples.
In \cite{freitag2023degree}, it is shown that for finite rank types in a
theory with the property that \emph{any non-locally modular minimal type is
non-orthogonal to a minimal type over the emptyset}\footnote{Such as the
theory of differentially closed fields or compact complex manifolds.},
$\nmdeg (p)\geq d$ implies that the type is internal to a non-locally modular
minimal type and has a binding group which is generically $d$-transitive
\cite[Section 2]{freitag2022any}. In the case that the theory satisfies the
Zilber trichotomy, this binding group and its action on the type is
isomorphic to the action of an algebraic group acting regularly on an
algebraic variety; the arguments of \cite{freitag2023degree} can then be used
to show that {\em the degree of nonminimality is at most two}. Even without
the trichotomy assumption, to a great extent, by the O'Nan-Scott theorem in
this setting \cite{macpherson1995primitive}, the binding group actions can be
(to a large extent) reduced to actions of simple groups on the left cosets of
definable subgroups (see e.g. Theorem 2.3 of \cite{freitag2023differential}).

So, any finite rank types with $\nmdeg(p)\geq n >2 $ involve either:
\begin{enumerate}
    \item \label{first} A theory in which there are nonorthogonality
        classes of minimal types $p$ such that for any type in the class,
        $\fcb{p}>2$.


    \item \label{second} A definable simple group serving as a binding
        group which has transitivity characteristics not possessed by
        algebraic group actions.\footnote{Specifically, analyzing the
        arguments of \cite{freitag2023degree}, a high degree of generic
        transitivity and the property that for large values of $n$, every
        non-generic orbit of the binding group on $p^n$ is finite.}
\end{enumerate}
A (finite rank) $\omega$-stable example along the lines of \ref{second} would
involve a counterexample to the Cherlin-Zilber conjecture, so we pursue an
example along the lines of \ref{first} via a Hrushovski construction of
non-locally modular regular types. Indeed the most natural method to get
large degree of nonminimality along the lines of \ref{first} involves
ensuring that we have minimal types $p$ such that for any minimal type $q$ in
the nonorthogonality class of $p$, $\fcb{q}=n$ to build a type with
degree of nonminimality $n$.

Lemma~\ref{transimpnmfin} shows there are simple theories of $r$-spaces where
$k-1$-transitivity implies there exist types $p$ with $\nmdeg(p)$ arbitrarily
large (depending on the theory). We use quite different variants of the
Hrushovski construction  to build such examples that are $\aleph_1$-  but not
$\aleph_0$-categorical, $\aleph_0$-categorical simple, or strictly stable
$\aleph_0$-categorical.

\subsection{Hrushovski constructions}

We outline the properties of two constructions  Hrushovski gave in the late
1980's to refute a) Zilber's trichotomy conjecture for strongly minimal sets
and b) Lachlan's conjecture that an $\aleph_0$-categorical stable theory is
$\aleph_0$-stable. The key to both examples is to define a `pre-dimension'
that describes a family of matroids. These examples will all be {\em
$r$-hypergraphs}. The unique relation symbol $R$ in the vocabulary holds only
of distinct $r$-tuples and
    in any order.

The constructions are all from amalgamation classes of finite structures.

\begin{definition}[$\prec$-amalgamation classes]\label{amalclassdef}
A $\prec$-amalgamation class $(\bK_0,\prec)$ is a collection of finite
structures
 for a vocabulary $\sigma$ (which may have function and relation symbols) satisfying:
\begin{enumerate}

\item $\prec$ is a partial order refining $\subseteq$.
\item $\prec$ satisfies joint embedding and amalgamation.
\item $A,B,C\in \bK_0$, $A \prec B$, and $C\subseteq B$ then $A\cap C \prec
    C$.
\item  $\bK_0$ is countable
\end{enumerate}
\end{definition}

Condition Definition~\ref{amalclassdef}.3 is not essential for the next
result.  However it follows when the amalgamation class arises from a
dimension function as in Definition~\ref{predim} and is used for the later
analysis.

\begin{theorem}\label{genexist}
For a $\prec$-amalgamation class, there is a  countable structure $M$,  the
{\em $\prec$-generic model}, which is a union of members of $\bK_0$, each
member of $\bK_0$ embeds in $M$, and $M$ is $\prec$-homogeneous.
\end{theorem}

For Fra{\" \i}ss\'{e}, the language is finite relational, the class is closed
under substructure, and $\prec$ is $\subseteq$. We describe below two
instances of $\prec$ used in this paper.

\medspace

The notion of pre-dimension and dimension establish a matroid structure on
the generic model.
\begin{definition}\label{predim} Work in a relational vocabulary $\{R\}$.
Let $\delta: {\bf K}_0 \rightarrow \NN$ be a function with $\delta(\emptyset)
= 0$ that give the properties of a matroid. Extend $\delta$ to $d: {\hat \bK}
\times {\bK}_0
 \rightarrow \Re^+$ by for each $N \in {\hat \bK}$,
$d_N(A)= \inf \{\delta(B): A \subseteq B \subseteq_{\omega} N\}$.

Require for  $A,B,C \in \bK_0$  that are  subsets of $N\in \hat \bK$.
\begin{enumerate}
\item $d_N(A/C) \geq 0$.
\item $d_N(AB/C) = d_N(A/BC) + d_N(B/C)$.
\item If $A \subseteq A'$ then $d_N(A'/C) \geq d_N(A/C)$ (equivalent to
    submodularity).
\item For every $A$ and there is a finite $A_0$ with $A\subseteq
    A_0\subseteq  N$ such that $\delta(A_0) = d_N(A)$.
\end{enumerate}
\end{definition}

The original examples were of the form $\epsilon(A) = m |A| - n |R(A)|$,
where $|R(A)|$ is the number of instances of tuples satisfying $\tau$
relations in $A$.  A necessary extension for the counterexample to Lachlan's
conjecture is to allow a
carefully chosen irrational coefficient. 


\medspace

   Each conjecture requires a different interpretation of $\prec$ on
$\bK_0$.

\begin{definition}\label{defsstr} \begin{enumerate}

     \item For Zilber conjecture: $A\leq B $ (strong substructure) if $\forall B'$ with $A\subseteq B' \subseteq B$, $
         \delta ( B'/ A)\geq 0$.


\item For Lachlan conjecture: $A \leq_* B$ (*-strong substructure) if
    $\forall
    B'$
    with
    $A\subsetneq B' \subseteq
    B$, $\delta(A) < \delta(B')$:

    \item   $\icl(A)$ is the
minimal extension $B$ of $A$ with $B \leq N$ for any $N\in \hat \bK$. It can
be obtained as union of a chain $B_i$ such that $\delta(B_{i+1}/B_i) < 0$, is
contained in $\acl(A)$ and is $\leq $ any extension in $\bK_0$.  Without the
control functions discussed below $|\icl(A)|$ may not be bounded in terms of
$|A|$ \cite{BaldwinShiJapan,BaldwinShelahran, Laskowax}.

\end{enumerate}
\end{definition}

\cite{BaldwinShiJapan} and \cite[p 159]{Ev02} call the first interpretation $\leq$ and the second
$\leq^*$; \cite{Evansmini} calls  the first $\leq_s$ and the second $\leq_d$.

\begin{definition}\label{deffc} $M\in \hat \bK_0$ {\em has finite closures} with respect to $(\bK_0,\prec)$  if
for every finite $A \subseteq M$ there is a finite $B \prec M$ with $A
\subseteq B$.
\end{definition}

For rational coefficients both $\leq$ and $\leq_*$ have finite extensions
that are ``strong" -- i.e. ``closed in the universe", but for $\leq$ such
extensions are not bounded in size, while in practice (see the next section), they are bounded in size for $\leq_*$. For
irrational coefficients $\leq$ need not have finite closures.

The crucial distinction between the two is that while for rational
coefficients both satisfy Definition~\ref{predim}.iv), only $\leq_*$  has a
unique minimal $*$-strong extension which can then be thought of as a
$*$-closure\footnote{This is often,  somewhat misleadingly, called a
$d$-closure, since it is the choice of strong, not of $d$ which determines
the precise closure.}.

\medspace

The position of the theory of the generic in the stability classification
depends on rational vs irrational coefficients of $\delta$ and on the choice
between $\leq$ and $\leq_*$. Here $\icl(A)$ is the minimal extension $B$ of
$A$ with $B \leq N$ for any $N\in \hat \bK$. We will write {\em generic} for
$\leq$-generic and {\em $*$-generic} for $\leq_*$-generic.

\begin{definition}

[d-independence] \label{dinddef}
$c$ is \textit{$d$-independent} from $B$ over $A \subseteq B$ if 
\begin{enumerate}
\item \cite[\S 3]{BaldwinShiJapan} For $\leq$:
 $d(c/B) = d(c/A)$ and $\icl{c A} \cap \icl{B} \subseteq \icl{A}$.

\item \cite[Cor. 2.20]{Ev02}  $\leq_*$: $d(c/B) = d(c/A)$ and $\acl(cA)
    \cap B = A$.
\end{enumerate}

The minimal $B\supseteq A$
with $B\leq M$ ($\icl (B)$)   can be strictly smaller than the minimal
$B\supseteq A$ with $B\leq_* M$. The second  $B$ is $\acl(A)$ in the
$\aleph_0$-categorical (rational coefficients)     $\leq_*$-case.

When the choice of $\prec$ is clear we write $c \ind_C^d B$ in either case.
\end{definition}

\begin{lemma}\label{d-nf} Let $R(A)$ be the set of of instances of the relation of $R$ in $A$.
\begin{enumerate} 
\item $\leq$ \cite[3.39]{BaldwinShiJapan}:  The theory of an  generic for  $d(A) = \alpha |A|-|R(A)|$ is stable if
    $\alpha$ is irrational and $\omega$-stable if rational.
\item $\leq$ \cite[2.20,2.24,3.9]{Ev02}:  The theory of a *-generic model
    of a $\aleph_0$-categorical theory  for $d(A) = \alpha |A|-|R(A)|$
    is stable if $\alpha$ is irrational.  

\end{enumerate}
\end{lemma}

Lemma~\ref{d-nf} 1) and 2) agree for irrational $\alpha$ since then $\leq$ coincides with $\leq_{*}$.

Evans works with expansions of an underlying $\aleph_0$-categorical theory
but we restrict to the {\em ab initio} case, beginning with a collection of
finite structures.

\subsection{Supersimple and strictly stable $\aleph_0$-categorical theories}\label{ocat}

In this section we study Hrushovski constructions on various classes
$\bK_0 = \mathcal{C}_{f}$ of $k$-hypergraphs, where $f$ is a control function, to obtain $\aleph_0$-categorical theories of various
stability classes with maximal $\nmdeg(p) =k$. The distinction will arise
from different choices of the pre-dimension $d_0$, control function, and
notion of strong substructure.

In \cite{EP02}, Evans and Pantano construct $\omega$-categorical simple
theories which they show are $n$-transitive. In an $n$-transitive simple
theory, if $a_{1}, \ldots a_{n-1} \models p$ for $p \in S(\emptyset)$ a type
with $\mathrm{SU}(p) > 1$, then there must be exactly one nonalgebraic $1$-type
over $a_{1} \ldots a_{n-1}$, because the types of any $n$ distinct
realizations of $p$ are all the same. So this nonalgebraic $q \in S(a_{1}
\ldots a_{n-1})$ must be a nonforking extension of $p$, so $\mathrm{nmdeg}(p)
\geq n$. We expect that the construction of \cite{EP02} has finite
$\mathrm{SU}$-rank and $\mathrm{SU}$-rank greater than $1$, giving us an
example of an $\omega$-categorical supersimple theory of finite rank with a
type $p$ with $\mathrm{nmdeg}(p) \geq n$. However, their construction is more
complicated than necessary to get $n$-transitivity (because they also have
the goal of making the algebraic closure grow arbitrarily fast), so we show
that we can also get $\mathrm{nmdeg}(p) \geq n$ in the less complicated
constructions of  \cite{Ev02}.

In \cite{GFA}, there is an overview of the constructions of Evans, which we
reproduce here in modified form. Let $\mathcal{L}$ be a language with
finitely many relations, and for each relation symbol $R_{i}$, let
$\alpha_{i}$ be such that $\alpha_{i} \cdot n_{i}!$ is a non-negative integer
associated to $R_{i}$, where $n_{i}$ is the arity of $R_{i}$.  For $A$ a
finite $\mathcal{L}$-structure, define a \textit{predimension} $d_{0}(A)=|A|
- \sum_{i} \alpha_{i}|R_{i}(A)|$, with $R_{i}(A)$ the set of tuples of
$R_{i}$ with elements of $A$ as coordinates. As usual in Hrushovski constructions of
hypergraphs, we write $|R(A)|$ for the number of tuples realizing the relation
$R$. As in Definition \ref{defsstr}, define the relation $A \leq_{*} B$,  for
$B$  any $\mathcal{L}$ -structure and $A$ a finite substructure of $B$, to
mean that every finite superstructure of $A$ within $B$ has predimension
greater than $A$. Let $f$ be an increasing continuous positive real-valued
function and let $\mathcal{C}_{f}$ be the class of finite
$\mathcal{L}$-structures $A'$, so that the $R_{i}$ are irreflexive (i.e.,
only contain tuples of distinct elements) and symmetric (invariant under
permutation of coordinates) on $A'$, and any substructure $A\subseteq A'$
satisfies $d_{0}(A) \geq f(|A|)$. (Note that symmetry is not required by
\cite{Ev02}, but is useful for the presentation here.)

Then Evans (\cite{Ev02}) shows that when $f'(x)$ exists and $0 < f'(x) <
\frac{1}{x}$ for all $x \geq 0$, $\mathcal{C}_{f}$ is a $\leq_{*}$-amalgamation class as in Definition \ref{amalclassdef}, so the $*$-generic of the class $\mathcal{C}_{f}$ exists. That is, there is a $\mathcal{L}$-structure $M$ such that every finite substructure of $M$ belongs to $\mathcal{C}_{f}$, and such that for all finite $A$, $B$ with $A \leq_{*} M$ and $A \leq_{*} B$, there is an embedding $\iota: B \to M$ with $\iota|_{A} = \mathrm{id}_{A}$ and $\iota(B) \leq_{*} M$. Let $T$ be the theory of $M$, which will be $\omega$-categorical. Then for $\mathbb{M} \models T$, the algebraic closure (or just \textit{closure}) of any finite set $A \subset M$ is the $*$-closure of $A$, or the minimal $B \leq_{*} M$ containing $A$. If additionally $f(3x) \leq f(x) + 1$, Evans (\cite{Ev02}) shows that $T$ is simple. For $A, B$ finite subsets of an ambient (sufficiently saturated) model $\mathbb{M} \models T$, define $d(A) = d_{0}(\mathrm{acl}(A))$ and $d(A/B)= d(\mathrm{acl}(AB))-d(\mathrm{acl}(B))$, as in Definition \ref{predim}. Then for $A, B, C $ finite subsets of $\mathbb{M}$, Evans characterizes forking in $T$, exactly as in Definition \ref{dinddef}. That is, for all finite sets $A$, $B$ and $C$, $A \ind _{C} B$ if, and only if, $d(A/C) = d(A/BC)$.


We first show that the degree of nonminimality can be arbitrarily large in the theories defined by Evans. We will then improve this, but will show this weaker statement first to give intuition for the proof.

\begin{fact}
    For any $n < \omega$, there is an $\omega$-categorical supersimple theory $T_n$
    admitting a type $p$ with $\mathrm{nmdeg}(p) \geq n$.
\label{simpnm}
\end{fact}

\begin{proof} Let $\mathcal{L} = \{R\}$ for $R$ a
$k$-ary relation  with $k = 3^{2n+1}$, and let $\delta(A) = 2|A| - \frac{|R(A)|}{k!}$. (Note that the factor $\frac{1}{k!}$ accounts for ths symmetry, so $\frac{|R(A)|}{k!}$ is equal to the number of tuples with coordinates in $A$, up to permutation, that realize $R(A)$).  Then let $T_n$ be the theory of the *-generic $M$ for the class $\mathcal{C}_{f}$,
where $f(x) = \frac{\mathrm{log}(x)}{\mathrm{log}(3)}$.

Now the criteria for $T$ being well-defined and simple, $0 < f'(x) <
\frac{1}{x}$ and $f(3x) \leq f(x) + 1$, will hold for all $x \geq 0$. By Evans's characterization of forking ($A \ind _{C} B$ if and only if $d(A/C) = d(A/BC)$ for $A$, $B$, $C$ finite sets), and the choice of $d_{0}$, $T_{n}$ is supersimple of $\mathrm{SU}$-rank at most $2$. It remains to show that there is a unique type in $S^{n}(\emptyset)$ realized by an $n$-tuple of distinct elements (i.e. $T_{n}$ is $n$-transitive), and that for $p \in S^{1}(\emptyset)$ the unique $1$-type, $\mathrm{SU}(p) = 2$. To show the first claim, let $a_{1}, \ldots, a_{n} \in \mathbb{M}$ be distinct; it suffices to show that $\{a_{1}, \ldots, a_{n}\}$ is algebraically closed, because isomorphic algebraically closed sets have the same type. But $A=  \mathrm{acl}(a_{1}, \ldots, a_{n})$ must have predimension at most $d_{0}(\{a_{1}, \ldots a_{n}\}) = 2n$. So if $\{a_{1}, \ldots, a_{n}\}$ is not algebraically closed, so $|A| > n$, it must then be the case that $|A| \geq k$, because otherwise $A$ would have no instances of $R$ and $d_{0}(A)= 2|A| > 2n$. But then by choice of $f$ and $k$, $f(|A|) > 2n$, contradicting $d_{0}(A) \leq 2n$.

To show $SU(p) \geq 2$, so $\mathrm{SU}(p) = 2$, because $f(x) \leq 2x -1$ for $x \leq k$, one can see there is a structure in  $\mathcal{C}_{f}$ of size $k$ with exactly one instance (up to permutation) of the relation $R$. Since $T_{n}$ is the theory of the $*$-generic of $\mathcal{C}_{f}$, there is an algebraically closed set $\{a_{1}, \ldots, a_{k}\} \subseteq \mathbb{M}$ isomorphic to this structure. Since $\{a_{1}, \ldots, a_{k}\}$ and therefore $\{a_{1}\}$, $\{a_{2} \ldots, a_{k}\}$ and $\emptyset$ are algebraically closed, $d(a_{1}/a_{2} \ldots a_{k}) = d_{0}(a_{1} \ldots a_{k}) - d_{0}(a_{2} \ldots a_{k})= 1$ and $a_{1} \notin \mathrm{acl}(a_{2}, \ldots a_{k})$, while $d(a_{1}/\emptyset)=d_{0}(a_{1}) - d_{0}(\emptyset) = 2$, so $a_{1} \nind_{\emptyset} a_{2} \ldots a_{k}$. Since $a_{1} \models p$, $\mathrm{SU}(p) \geq 2$, so $\mathrm{nmdeg}(p) \geq n$ by the $n$-transitivity shown above.

\end{proof}

We modify this example to show the strict inequality $n=\mathrm{nmdeg}(p) <
F_{ind}(p)$ for arbitrary $n$.

\begin{theorem}\label{simpfind}
Let $n < \omega$. There is an $\omega$-categorical supersimple theory $S_n$
admitting a type $p$ with $\mathrm{nmdeg}(p) = n$ and $F_{ind}(p) >
\mathrm{nmdeg}(p)$.

\end{theorem}

\begin{proof} 
Let $\mathcal{L} = \{R\}$ for $R$ an
$k$-ary relation with $k = n+1$, and $\delta(A) = 2|A| - \frac{|R(A)|}{k!}$. Then let $S_n$ be the theory of the *-generic $M$ for the class $\mathcal{C}_f$ where $f(x) = 2x -2$ for $x \leq
n+2$, and $f(x) = 2(n+2) -2 + (\frac{\mathrm{log}(x)}{\mathrm{log}(3)}-
\frac{\mathrm{log}(n+2)}{\mathrm{log}(3)})$ for $x \geq n + 2$.

Now the criteria for simplicity and well-definedness of the $*$-generic of $\mathcal{C}_f$, $0 < f'(x) < \frac{1}{x}$ and $f(3x)
\leq f(x) + 1$, will hold for all $x \geq n+2$, but not for all $x \geq 0$.
However the $*$-generic of $\mathcal{C}_f$ is still well-defined and simple. As in Remark 3.8 of \cite{Ev02}, we may
individually check the instances of amalgamation where some of the structures
being amalgamated have size at most $n+1$, since these structures either have
no relations, or have size $n+1$ and one relation. Moreover, because $f(x) \geq 2$ for $x \geq k$, we again
see that there is only one type $p \in S^{1}(\emptyset)$. And because $f(x) \leq 2x -1$ for $x \leq k$, so there is again a structure in $\mathcal{C}_{f}$ of size $k$ with exactly one instance of $R$ up to permutation, $\mathrm{SU}(p) = 2$ with
$\mathrm{nmdeg}(p) \leq k-1$, so here $\mathrm{nmdeg}(p) \leq n+1-1 = n$.

To see  $\mathrm{nmdeg}(p) \geq n$, so $\mathrm{nmdeg}(p) = n$, we show there
is   a unique type in $S^{n}(\emptyset)$ realized by distinct $a_{1}, \ldots,
a_{n}$.
Because $f(x) > 2n \geq d_{0}(a_{1}, \ldots,
a_{n})$ for $x >n + 1$
$|\mathrm{acl}(a_{1}, \ldots, a_{n})| \leq n+1$. But $|\mathrm{acl}(a_{1} \ldots a_{n})| = n+1$ is
impossible, because the least predimension of a set of size $n+1$ is $2n+1$, which is larger than $d_{0}(a_{1} \ldots a_{n}) = 2n$.



  Now to show
that $n=\mathrm{nmdeg}(p) < F_{ind}(p)$, note that if $ \{a_{i}\}_{i < \omega}$ is a nonconstant indiscernible sequence, either every
$n+1$-tuple satisfies $R$ or none does. The first is clearly impossible since
long enough such sequences have negative predimension. In the second case,
$d_{0}(a_{1} \ldots a_{n+1}) = 2n + 2$. Because $f(x) > 2n + 2$ for $x >
n+2$, $|\mathrm{acl}(a_{1} \ldots a_{n+1})| \leq n+2$, so either $a_{1},
\ldots a_{n+1}$ is closed or $|\mathrm{acl}(a_{1} \ldots a_{n+1})| = n+2$ and
$d_{0}(\mathrm{acl}(a_{1} \ldots a_{n+1})) \geq f(n+2) = 2n + 2$. For either
option $a_{n+1} \ind a_{1}, \ldots a_{n}$ and $F_{ind}(p) > n$.
\end{proof}

%

We apply Evans's (\cite{Ev02}) modification of the construction  used by
\cite{Herwig}, \cite{Hrustableplane} to obtain the original non-superstable
examples of $\omega$-categorical stable theories. He avoids the requirement
that $\alpha$ has infinite index.
%
Evans's more general criteria for $T$ to be a well-defined simple theory are
as follows: define $\kappa(x)$, for $x > 0$, to be the minimum of $1$ and the
smallest possible positive value of $\delta(D)-\delta(A)$ where $A \leq_{*}
D$ and $|D| \leq x$ (Definition 3.1 of \cite{Ev02}.) Using the more general
statement of  Lemma 3.3 and Theorem 3.6(ii) of \cite{Ev02}, $T$ is
well-defined and simple whenever $0 < f'(x) < \frac{\kappa(x)}{x}$ and $f(3x)
\leq f(x) + \kappa(x)$.  Again, we may apply Remark 3.8 of \cite{Ev02}: we
will be able to show that $T$ is a well-defined simple theory when $f(x) = 2x
-2$ for $x \leq n+2$, and on $x \geq n+2$, $f$ is any continuous function
satisfying the criteria just listed.  Then the proof that $T$ has a type $p$
with $\mathrm{nmdeg}(p) = n$ and $F_{ind}(p) > \mathrm{nmdeg}(p)$ is as
above; the only difference is that the $\mathrm{SU}$-rank of $T$ will no
longer be $2$, because it is no longer the case that $d(a/\emptyset)$ can
only strictly decrease twice when enlarging the base set. However, in the
case where $\alpha$ is irrational, one sees that the following condition
holds: if $A \supsetneq B$, then $\delta(A) \neq \delta(B)$. By Lemma 3.9 of
\cite{Ev02} (essentially reproduced as Lemma \ref{d-nf}(2) here), whenever $\delta$ satisfies this condition, $T$ (if
well-defined) is stable, regardless of $f$. So the proof of
Theorem~\ref{simpfind} also gives us:

\begin{theorem} For every $n< \omega$ there is an $\omega$-categorical strictly stable theory,
 admitting a type $p$ with $\mathrm{nmdeg}(p) = n$ and $F_{ind}(p) > \mathrm{nmdeg}(p)$.
\end{theorem}

The theory is not superstable as an $\aleph_0$-categorical superstable theory
is $\omega$-stable and then must be locally modular. But the Hrushovski
examples are CM-trivial.


%
%

\subsection{$\aleph_1$-categorical theories} \label{nm1cat}

To construct $\aleph_1$-categorical theories with types of arbitrary degree of nonminimality, we consider
a specific type of design reflecting the intersection theory of curves in algebraic geometry.

\begin{definition}\label{design}  A $t-(v, k,\lambda)$ design is a set $P$ of $v$ points,
together with a set $\BB$ of blocks each of which is a $k$-subset of $P$, and
which has the property that any set of $t$ points is a subset of exactly
$\lambda$ blocks.
\end{definition}

\begin{theorem}\label{asmthm} For each $r,k< \omega$ there is an
$\aleph_1$-categorical (almost strongly minimal) theory $T_{r,k}$
that has a  ($1$)-type $p\in S(\emptyset)$ with $\nmdeg(p) =r$; the
nonforking extension of $p$ over $r$ elements is strongly minimal. Each model
of cardinality $\kappa$ is an $(r-1) -(\kappa,\kappa,1)$ design and the
restriction to a strongly minimal subset $P$, is an $(r-1) - (\kappa,k,1)$ design
($k$-Steiner system).
\end{theorem}

\begin{definition}[$r$-spaces]\label{defls}
A $\tau$-structure $(M,R)$ is

\begin{enumerate}
\item An {\em $r$-space} if $R$ defines an $r$-hypergraph in which each
    $r-1$ -tuple of distinct points is contained in a unique maximal
    $R$-clique (curve).  


   This is enforced by requiring that for any $r$-space $A$, if $a$ and $b$
    are each on a line through  $\cbar =\langle c_0,c_1, \ldots
    c_{r-1}\rangle$ then $R(a,b,\cbar')$ for any $\cbar' \subset \cbar$
    with $\cbar' = r-2$.

\item $\bK^*_0 = \bK^*_0(r)$ denotes the collection of all finite
    $r$-spaces.

\item An $r$-space is a {\em $(r,k)$-Steiner system} if all curves have the
    same length $k$.

        \end{enumerate}
\end{definition}

For infinite $\lambda$ we built $(r-1)-(\lambda, \lambda,\lambda)$ block
designs in the $\aleph_0$-categorical case and now build  both $(r-1)-(\lambda,
\lambda,\lambda)$ and $(r-1)-(\lambda, k,\lambda)$ (for fixed $k$) in the
$\aleph_1$-categorical case.

Recall Definition~\ref{RahimJim}.

\begin{definition} (Degree of nonminimality)  \label{nmdegdef} Suppose $p \in S(A)$ is a
nonalgebraic and nonminimal stationary type. By the degree of nonminimality
of $p$, denoted by $\nmdeg(p)$, we mean the least k such that $p$ has a
nonalgebraic forking extension over $A \cup \{a_1, \ldots , a_k\}$, for some
$a_1, \ldots , a_k$ realising $p$.

So for $m < \nmdeg(p)$, if
$\abar$ is a  sequence of $m$ distinct  realizations of
$p$ and $q$ with $p \subseteq q\in S(A\abar)$ forks over $A$, $q$ has only
finitely many realizations.
\end{definition}


The Hrushovski construction \cite{Hrustrongmin} created examples of strongly
minimal sets whose geometries were neither discrete, group-like nor
field-like. The modifications in \cite{Baldwinasmpp,  BaldwinPao,
BaldwinHolland} to obtain $\aleph_1$-categorical theories of non-Desarguesian
geometries, Steiner systems, and Morley rank 2 fields with definable infinite
subsets, respectively, are adapted here.


Now we define our geometrically based pre-dimension function by modifying
\cite{BaldwinPao}. Although the following definitions each depend on $r$,
since we will fix $r$ for each example, after describing the vocabulary, we
have suppressed the $r$.
 The unary predicate $P$ is used to guarantee that the theory is almost strongly minimal
 and so $\aleph_1$-categorical. This allows a $k$-Steiner system on $P$ while
 the lines are infinite in the entire structure.
 
 \begin{definition}\label{lines}
 \begin{enumerate}
 \item  We denote the cardinality of an $R$-clique $\ell \subseteq A \in
     \bK^*_0 $ by $|\ell|$, and, for $B \subseteq A$, we denote by
     $|\ell|_B$ the cardinality of $\ell \cap B$.
   \item We say that a non-trivial curve $\ell$ contained in $A$ is {\em
       based in} $B \subseteq A$ if $|\ell\cap B| \geq r-1$, in this case we
       write $\ell \in L(B)$.
\item	\label{nullity} The {\em nullity of a line} $\ell$ contained in a
    structure $A \in \mathbf{K}^*_0$ is:
	$$\mathbf{n}_A(\ell) = |\ell| - (r-1).$$

\end{enumerate}
\end{definition}

Each $\aleph_1$-categorical Hrushovski  construction depends on a number of parameters: the vocabulary, the choice of a pre-dimension, and an algebrizing function.
Succesive classes  $\bK^*$ of finite structures defined in terms of the pre-dimension and the  algebrizing function lead to the definition of an amalgamation class of finite structures with respect to a notion of strong embedding. The generic for this class is the prototypical model. A preliminary classification of such constructions appears in \cite{BaldwinVer}.

\begin{definition}\label{defdelrank} We define the appropriate $\tau, \bK^*$ and $\bK_0$.
\begin{enumerate}
\item 
$\tau$ has an $r$-ary relation $R$ and a unary predicate $P$.
 \item  $\bK^*$ is the class of  $r$-spaces
 and  $\bK^*_0$ is the class of  finite $r$-spaces.

 \item 
 For $(A,R) \in  \bK^*_0$ let:
	$$\delta(A) = 2(|A|) -|P(A)| - \sum_{\ell \subseteq A} \mathbf{n}_A(\ell).$$

\item  
$(A,R) \in  \bK'_0 \subseteq \bK^*_0$ if for any $A' \subseteq A, \delta(A') \geq 0$.

\item $(A,R) \in  \bK_0$ if $A\in  \bK'_0$ and there is no $B\subseteq A$
    containing a subset $C$ of a single realization of $R$ with
    $\delta(B/C) < \delta(C)$.

\end{enumerate}
\end{definition}

 The switch from $\bK'_0$ to
$\bK_0$ in Definition~\ref{defdelrank} is patterned on
\cite[Proposition 18]{Hrustrongmin} and \cite{Baldwinautpp}; $\bK_0$ has amalgamation just
 as $\bK'_0$ does,
 as the amalgamation procedure
introduces none of the forbidden structures. The result is $<r-1$-transitivity.

\cite[page 6]{Kaveh} gives an
example of a $3-(8,4,1)$ design, $A$.  Note that there is a copy of $A
\subseteq \neg P$ in $\bK'_0$ as for such an $A$, no subset of $A$ has
negative rank,  and $\delta(A) =2 $ while the $4$ element curve has rank $7$
and so is not strong in $A$. Thus, the restriction to $\bK_0$ is essential
for the transitivity.   On page 8 of the same chapter is
a $3-(10,4,1)$-design  with 30 blocks; thus it is forbidden in $\bK'_0$, even in
$\neg P$.


\begin{definition}[Strong Extensions]\label{stext}
\begin{enumerate}
\item For any class $\bK$ of finite structures, $\hat \bK$ denotes the
    collection of structures of arbitrary cardinality that are direct
    limits of models in $\bK$.

    \item Extend $\delta$ to $d: { \hat \bK}_0 \times {\bK_0} \rightarrow
        \omega$ by for each $N \in {\hat \bK}_0$ and $A \subset_{<\omega}
N$, $d(N,A) = \inf \{\delta(B): A \subseteq B \subseteq_{\omega} N\}$,
$d_N(A/B) = d_M(A \cup B) -d_M(B)$. We usually write $d(N,A)$ as $d_N(A)$
and omit the subscript $N$ when clear.

%

    \[d_N(A) = d_N(A/\emptyset).\]
    \item For any $N \in \hat\bK_0$ and finite  $B \subseteq N$, we write
        $B \leq N$ (read $N$ is a \emph{ strong extension} of $B$) when  $B
        \subseteq A\subseteq N$ implies $d_N(A) \geq d_N(B)$.

\item We write $B < A$ to mean that $B \leq A$ and $B$ is a proper subset
    of $A$.

  \item A structure $A$ is an $(\bK_0, \leq)$-{\em generic} or $\leq$-{\em
      homogeneous} if $A$ is a $(\bK_0, \leq)$-union and for any $B\leq C$
      each in $\bK_0$ and $B\leq A$ there is a $\leq$-embedding of $C$ into
      $A$.
\end{enumerate}
\end{definition}

In the situation here, if $A
    \subseteq M$ and $|A| \leq r-1$ then for any $B \supseteq A$, $A\leq
    B$.  In other situations, we can restrict the number of  finite structures by adding
    such a
    condition.

\begin{observation} Counting automorphism types: {\rm If $|X| = s \leq r-1$, $X$ is strong
 in any extension. Thus,
in the generic model the quantifier free type  of $X$  determines  the
automorphism type of $X$; {\em a priori} there are $2^s$ such types depending
which elements of $X$ are in $P$. If $|X| =r$, the other quantifier free
question asks whether $R$ or $\neg R$ holds of the tuple, yielding $2^{r+1}$
$r$-types. Note however that since the truth of $R(\abar)$ does not depend on
the order of $\abar$, there are actually only $r+1$ automorphism types of
$r$-tuples depending on the number of $a_i$ in $P$.}
\end{observation}

\begin{definition}{\bf Canonical Amalgamation}\label{canamdef} For any class $(\bK_0,\delta)$,
    if $A\cap B = C$, $C\leq A$ and $A,B, C
    \in \bK^*$,  $G$ is a {\em free (or canonical) amalgamation},  $G = B \oplus_C
    A$,
    if
      $ G \in
    \bK^*$, $\delta(A/BC) = \delta(A/C)$ and $\delta(B/AC) = \delta(B/C)$

\end{definition}

It follows \cite[Lemma 3.10]{BaldwinPao} that the amalgam satisfies: $\delta(
A \oplus_{C} B)= \delta(A) +\delta(B) -
    \delta(C)$ and any $D$ with $C \subseteq D \subseteq A \oplus_{C} B$ is
    also free.  Thus, $B \leq G$.
    In addition to submodularity, \cite{BaldwinPao} proves that the
    constructed strongly minimal set is `flat' (\cite{Hrustrongmin}).

    The crux of $\aleph_1$-categoricity is the `algebrizing' function $\mu$ which makes
    all non-generic types in $P$ algebraic.

 \begin{definition}\label{prealgebraic} Let $A, B \in {\bK}_0$ with
 $B \subseteq A$ and\footnote{For ease of exposition, we insist primitive extensions
 are proper.} $B-A \neq \emptyset$.
	\begin{enumerate}
\item $A$ is a \emph{primitive extension} of $B$ if $B \leq A$ and there is
    no $B \subsetneq A_0 \subsetneq A$ such that $B\leq A_0 \leq A$. $A$ is
a {\em  $k$-primitive extension} if, in addition, $\delta(A/B) =k$.

	\item  We say that the $0$-primitive pair $A/B$ is {\em good} if there
is no $B' \subsetneq B$ such that  $(A/B')$ is $0$-primitive.
(\cite{Hrustrongmin} called this a minimal simply algebraic or m.s.a.
extension.)
	\item If $A$ is $0$-primitive over $B$ and $B' \subseteq B$ is such
 that we have that $A/B'$ is good, then we say that $B'$ is a {\em base}
 for $A$ (or sometimes for $A-B$).
	\item If the pair $A/B$ is good, then we also write $(B,A)$ is a {\em good pair}.
\end{enumerate}
\end{definition}

	\begin{definition}     \label{Kmu} [The algebrizing function]	
	\begin{enumerate}
	\item  Let $\Uscr$ be the collection of functions $\mu$
 assigning to every isomorphism type ${\beta}$ of a good pair $C/B$ in
$\bK_0$ a natural number
$\mu({\beta}) = \mu(B,C) \geq \delta(B)$. 

\item For any good pair $(B,C)$ with $B \subseteq M$ and $M \in \hat
    \bK_0$, $\chi_M(B,C)$ denotes the number of disjoint copies of $C$ over
 $B$ in $M$. A priori,  $\chi_M(B,C)$ may be $0$.

	\item\label{Kmuitem} Let $\bK_\mu$ be the class of structures $M$ in
$\bK_0$ such that  if $(B,C)$ is a good pair then $\chi_M(B,C) \leq
\mu((B,C))$

\end{enumerate}	
\end{definition}

Minor modifications of the proof in \cite{BaldwinPao} yield.
\begin{theorem} Let $\mu \in \Uscr$;
the class $\bK_{\mu}$ has amalgamation and thus a generic model $M$.
\end{theorem}


We now describe an additional parameter $k$ for $\mu$ which guarantees that
each line in $M$ intersects $P(M)$ in $(r+k)$-points.

\begin{notation} \label{notbeta} Let $\beta_s$ be the isomorphism type of $s \leq r$ elements
 and let $\beta^{-}_s$ denote the isomorphism type of the structure obtained by dropping any single point.
\end{notation}

Since $R$ is symmetric it is harmless to regard the dropped point as the last
one.

\begin{observation}\label{shline}
\begin{enumerate}
\item If $s < r$, $\delta(\beta/\beta^-) > 0$
\item If $s=r$ and the last element is not in $P$, $\delta(\beta/\beta^-) >
    0$.
    \item    If $s=r$ and the last element is  in $P$, then
        $\delta(\beta/\beta^-) = 0$ and  $\beta/\beta^-$  is a good
pair.
\end{enumerate}
\end{observation}

\begin{definition}\label{kdep} For $t\geq r$ in $\NN$, let $\bU_t$ be the subset of $\Uscr$ such that if $\beta$ is
one of the $2^{r-1}$ isomorphism types of a sequence satisfying
Observation~\ref{shline}.3, then $\mu(\beta/\beta^-) = t$.
\end{definition}

\begin{observation}\label{linelength} If $\mu \in \bU_t$ then each line in the generic model
$M$ for $\bK_\mu$ has $k=r+t-1$ elements in $P$.
\end{observation}

 While lines in models of  $T_\mu$ are infinite,
 each line intersects $P(M)$ in $k$ points, for $k$ as in Notation~\ref{linelength}; this motivates the following notation.

\begin{notation} If we fix the arity of the relation $R$ as $r$ and $\mu$ is
in $U_t$ with $k = r+t$, we write $T_{r,k,\mu}$ for the theory of the generic
model for $\bK_\mu$.
\end{notation}

\begin{observation}\label{ansm} The restriction of $T_{r,k,\mu}$ to $P$ is the
analog (with larger $r$ and line length) to the strongly minimal Steiner
system in a vocabulary with a single ternary relation of \cite{BaldwinPao}.
So $T\upharpoonright P$ is strongly minimal.
\end{observation}

%


$T_{r,k,\mu}$ is the theory of generic for  $\bK_\mu$ with this choice of
$\mu$.

\begin{lemma}\label{curvesintersect} Any two distinct curves intersect in  $r-2$ points.
\end{lemma}
\begin{proof} In the class $\bK_0$, given $A$, a set of $2r-4$ distinct points, with $\abar, \bbar$ on infinite lines
that do not intersect in  $A$; if $B =Aa$ where $a$ is on each line then $B$ is in  $\bK_0$ since $\delta(Aa) > \delta(A) \geq 2r-4$. Thus in the
generic (by `what can happen does') the two infinite lines intersect. But that
intersection has at most $r-2$ points since the lines are distinct and $r-1$
points determine a line.\end{proof}

Recall that a theory $T$ is {\em almost strongly minimal} \cite{Baldwinasmii}
if there is a finite set of constants $\abar$ realizing a principal type and
a strongly minimal formula $\phi$ such that for every $N\models T$, $N
=\acl(\abar \cup \phi(M))$.


\begin{theorem}\label{asm} The theory $T_{r,k,\mu}$ is almost strongly minimal and for $M \models
T_{r,k,\mu}$ with $|M|=\kappa$. $P(M)$ is a strongly minimal $(r-1,k)$-Steiner
system (i.e., an $(r-1)-(\kappa, k,1)$ design), and $(M,R)$ is an
$(r-1)-(\kappa,\kappa,1)$-design.
\end{theorem}

\begin{proof} Let $M$ be the countable generic model and  fix two equal length
sequences $\abar, \bbar$ from the prime model that make up $2r-4$ distinct points not in $P(M)$.   Note that for each $c \in M$, the lines $\abar c$
and ${\bbar} c$ intersect $P$ in $c^1$ and $c^2$ respectively,
(where each $c^i$ is one of the
$k$-points in $P(M)$ on $\abar c^1$ ($\bbar c^2$)  and $c$ is one of the $r-1$ points on
the intersection of $\abar c $ and ${\bbar} c $). (Without loss, $c^1 \neq
c^2$ since the number of points in $P$ on either line is $k>r-2$.)  Thus,
$c\in \acl(\abar,{\bbar},c^1,c^2)$ and $T$ is almost strongly minimal.

\end{proof}

We now show that each of the curves is strongly minimal.
\begin{definition}
    An  $(n,m)$-relation  $T$ on $X \times Y$ is a definable binary relation such that
    each $x$ is related to $m$ elements of $Y$  and each $y$ is related to $n$ elements of $X$.
\end{definition}

\begin{lemma}\label{smcurves} Each curve defined by a formula $R(\abar,x)$
 ($|\abar| = r-1$) is strongly minimal.

\end{lemma}
\begin{proof}
 Fix $r-2$ points   $\cbar$ in $\neg P$.
Now,
 as $b$ ranges through $R(\abar,x)$, there must be a point $d\in  P(M) $ ($d \neq b$ since $k>1$) with $d,b$ both on the intersection of the curves determined by $\abar$ and $\cbar d$.
 By Lemma~\ref{curvesintersect} the two curves  intersect in $r-3$ further points and any of the $k$ points on $R(\mathbf{c},y,b) \wedge P(y)$ determines (with $\cbar$ the curve through $\cbar d$.
 Thus, the formula
 $\phi(x,y)=(R(\mathbf{a},x) \wedge R(\mathbf{c},y,x) \wedge P(y)$ defines a  $(k,r-2)$ relation between    $P(y)$ and $R(\abar,x)$;
so $R(\abar,x)$ is also strongly minimal.
 \end{proof}

We can choose $\mathbf{c} \in P(M)$, guaranteeing that the generic types of $P(M)$
and $R(\mathbf{c},M)$ are not weakly orthogonal.\

\begin{theorem}\label{2,2} For $T_{r,k,\mu}$ the  complete type over $\emptyset$, $p$,
  given by $\neg P(x)$ has $\nmdeg(p) =r-1$.
 \end{theorem}

\begin{proof} Choose distinct $a_1, a_2, a_3,\ldots a_r$ realizing $p(v)$ and such that $R(a_1, a_2,
a_3, \ldots a_r)$.
Let $p_i = \tp(a_i/A_{<i})$. Each $p_i$ for $i < r$ does not fork over the
empty set because  $\mathrm{Aut}(M)$ acts $i$-transitively on $\neg P$ guaranteeing
each $p_i$ is the unique non-forking extension of $p$ to its domain $A_{<i}$.
The type $p_r$ forks over $\emptyset$ because there are infinitely many
distinct choices of $a_{r-1}^j$ giving different lines. Any pair of the lines
$\ell_i = \langle a_1, a_2 \dots a_{r-1}, a^j_r \rangle$ are infinite and intersect only in
$a_1, a_2 \dots a_{r-1}$ since any further point of intersection would imply
the lines are identical.
%
\end{proof}
%

For r-spaces we have a general argument from r-transitivity to $\nmdeg =
(r-1)$.

\begin{lemma}\label{transimpnmfin} Suppose $T$ is a simple theory of an $r$-space, $p\in S(\emptyset)$
 with  $\mathrm{SU}(p) = n < \omega $, and for any model $M$ of $T$,
$\aut(M)$ acts $r-1$ transitively on $p(M)$.  Then $\nmdeg(p) =r-1$.
\end{lemma}

\begin{proof} Choose distinct $a_1, a_2, a_3,\ldots a_r = A_{<r+1}$ realizing $p$ and such that $R(a_1, a_2,
a_3, \ldots a_r)$.  Let $p_i = \tp(a_i/A_{<i})$. Each  $p_i$ for $i < r$ does
not fork over the empty set because  $\aut(M)$ acts $i$-transitively on
$p(M)$ guaranteeing each $p_i$ is the unique non-forking extension of $p$ to
its domain;
the only forking extensions over $A_{<i}$ are realized in   $A_{<i}$. 
The type $p_r$ forks over
$\emptyset$ because there are infinitely many distinct choices of $a_{r }^j$
giving different curves. Any pair of the curves $\ell_j = \langle a_1, a_2
\dots a_{r-1}, a^j_r \rangle$ intersect only in $a_1, a_2 \dots a_{r-1}$
since any further point of intersection would make  the curves identical.
 But then the formula
$R(a_1,a_2,\dots a_{r-1},x)$ forks over $\emptyset$ and $\nmdeg(p) = r-1$.
\end{proof}

\begin{remark}
We compare our construction with similar construction of Evans
\cite{Evansblock} who wrote:

\begin{quote}
 By a $t-(v, k,\lambda)$ design we mean a set $P$ of points, of cardinality
v, together with a set $\BB$ of blocks each of which is a k-subset of $P$,
and which has the property that any set of $t$ points is a subset of exactly
$\lambda$ blocks. An automorphism of the design $(P, \BB)$ is simply a
permutation of $P$ which preserves $\BB$. It is well known 
that if G is a group of automorphisms of the design $(P, \BB)$ and $t \geq 2$
and $v, k, \lambda$ are finite, then the number of $G$-orbits on $P$ is no
greater than the number of $G$-orbits on $\BB$: this is commonly referred to
as `Block’s Lemma'.
\end{quote}
\end{remark}

\cite{Evansblock} proves `For all ‘reasonable’ finite $t$, $k$ and $s$ we construct
a $t-(\aleph_0, k, 1)$ design and a group of automorphisms which is
transitive on blocks and has s orbits on points.' We can rephrase our current
result in this language as:

\begin{cor}

 Each model of cardinality
$\kappa$ is an $r-(\kappa,\kappa,1)$ design and the restriction to strongly
minimal subset $P$, is an $r-(\kappa,k,1)$ design. There are two orbits on
points: $P$ and $\neg P$, the action is transitive on blocks.
\end{cor}

Evans' structure  is a $k$-Steiner system so more similar to
\cite{BaldwinPao}. The infinite lines are necessary for the nmdeg result. By
adding additional unary predicates $P_i$  and letting line length in $P_i$
depend on $i$, we would recover Evans' result, but with one class of infinite
lines.

\section{The Koponen conjecture} \label{Kop}
\numberwithin{theorem}{section} \setcounter{theorem}{0}
Koponen asks the following question (\cite{Kop16}, \cite{Kop19}, \cite{Kop16slides}): 

\begin{question}
   Is every simple theory with quantifier elimination in a finite relational language one-based (in particular, supersimple of finite rank)?
\end{question}

Particularly, Koponen conjectures that the answer to the question of whether such a theory is supersimple of finite rank is yes, proposing that it should likely be possible to show this using the stability-theoretic techniques developed for simple theories. In this section, we resolve Koponen's conjecture: 

\begin{theorem}
\label{koponen}
    
    Let $T$ be a simple theory with quantifier elimination in a finite relational language. Then $T$ is one-based, so in particular, supersimple of finite rank.
\end{theorem}

To show this, we prove a result of independent interest on the structure of $\omega$-categorical simple theories.

\begin{theorem}\label{omegacategoricalsimplebutnotsupersimple}
    Let $T$ be $\omega$-categorical and simple, but not supersimple. Then there is a finite tuple $b$ and ($\emptyset$-)indiscernible sequences $I^{i}=(b^{i}_{j})_{j < \omega}$, for $i < \omega$, with $b^{i}_{0} = b$, so that $I^{i} \not\equiv I^{j}$ for $i \neq j < \omega$.
\end{theorem}

Note that for a theory to have quantifier elimination in a finite relational language is the same thing as being $\omega$-categorical and $n$-ary for some $n < \omega$.\footnote{Recall that a theory is $n$-ary if it has quantifier elimination in a language with only $n$-ary relation symbols.} Assuming $T$ as in the theorem is not just simple, but \textit{supersimple}, we will see that one proof of the theorem is just a corollary of Fact \ref{pseudolinearity} on $k$-linearity for $\omega$-categorical supersimple theories, due to Tomašić and Wagner (\cite{TW03}). When $\mathrm{SU}(p) = 1$ (i.e., in the case of \cite{TW03} used in the finite-rank case of the Koponen conjecture), that $k$-linearity of $p$ coincides with the maximum value of $F_{Mb}(p \otimes p)$ being equal to $k$ was observed in Corollary \ref{Mbpsquared}.

Another proof uses \cite{TW03} in the case that $T$ is already supersimple of finite rank, but directly uses $F_{Mb}$, in particular Corollary \ref{omegacategoricalsupersimple}, to show that if $T$ is supersimple it must have finite rank. The hardest new part of the proof of this theorem is then proving that a simple theory with quantifier elimination in a finite relational language must be supersimple.

We will need the following lemma on non-supersimple simple theories:

\begin{lemma}\label{forkingchains}
    Let $T$ be simple, but not supersimple. Then there is a limit ordinal $\lambda$, a finite tuple $a$, and sets $B_{\gamma}$, $\gamma < \lambda$, with $B_{\gamma} \subset B_{\gamma'} $ and $a \nind_{B_{\gamma}} B_{\gamma'}$ for $\gamma < \gamma' < \lambda$, so that for $B = \bigcup_{\lambda < \gamma} B_{\gamma}$, $\mathrm{SU}(\mathrm{tp}(a/B)) \in \mathrm{Ord}$.
\end{lemma}

\begin{proof}
    Suppose otherwise, so for any limit ordinal $\lambda$, finite tuple $a$, and sets $B_{\gamma}$, $\gamma < \lambda$, with $B_{\gamma} \subset B_{\gamma'} $ and $a \nind_{B_{\gamma}} B_{\gamma'}$ for $\gamma < \gamma' < \lambda$, for $B = \bigcup_{\lambda < \gamma} B_{\gamma}$, $\mathrm{SU}(\mathrm{tp}(a/B)) = \infty$. By transfinite induction, we will find a finite tuple $a$ and sets $B_{\gamma}$, $\gamma \in \mathrm{Ord}$, with $B_{\gamma} \subset B_{\gamma'} $, $a \nind_{B_{\gamma}} B_{\gamma'}$ and $\mathrm{SU}(\mathrm{tp}(a/B_{\gamma})) = \infty$ for $\gamma < \gamma' \in \mathrm{Ord}$, contradicting simplicity. The base case is just that $T$ is not supersimple. Suppose we have found $a$, and $B_{\gamma'}$ as desired for $\gamma' \leq \gamma$. Then $\mathrm{SU}(\mathrm{tp}(a/B_{\gamma})) = \infty$, so we can find $B_{\gamma+1}$ with $a \nind_{B_{\gamma}} B_{\gamma + 1}$ and $\mathrm{SU}(\mathrm{tp}(a/B_{\gamma})) = \infty$, completing the successor step. Finally, for the limit step, we must show that for $\lambda$ a limit ordinal, if $B_{\gamma} \subset B_{\gamma'} $ and $a \nind_{B_{\gamma}} B_{\gamma'}$ for $\gamma < \gamma' < \lambda$, and $B_{\lambda} = \bigcup_{\lambda < \gamma} B_{\gamma}$, $\mathrm{SU}(\mathrm{tp}(a/B_{\lambda})) = \infty$; then $B_{\lambda}$ can be chosen for stage $\lambda$ of the induction. But this was just the assumption we made in supposing our conclusion was false.
\end{proof}

We now prove Theorem \ref{omegacategoricalsimplebutnotsupersimple}:

\begin{proof} Assume $T$ is simple and $\omega$-categorical, but not supersimple. Let $a$ and $B$ be as in Lemma \ref{forkingchains}; then $\mathrm{SU}(\mathrm{tp}(a/B)) \in \mathrm{Ord}$, but for any $B_{0} \subseteq B$ with $|B_{0}| < \omega$, $\mathrm{SU}(\mathrm{tp}(a/B))= \infty$. Let $|a|=k$. Recalling Definition \ref{resolvable}, we claim that it suffices to show that, for arbitrarily large $n$, there exists a set $A_{n} \subset \mathbb{M}$ and type $p_{n} \in S^{k}(A_{n})$ that is $(n+1)$-resolvable but not $n$-resolvable.  This suffices because, by the equivalence between (1), (2) and (2') in Proposition \ref{resolvabilitygrank}, a Morley sequence $J^{n}$ in $p_{n}$ over $A_{n}$ will be $(n+1)$-resolvable but not $n$-resolvable (as an indiscernible sequence over $\emptyset$), so $J^{n} \not\equiv J^{m}$ for $n \neq m$, giving us the conclusion of Theorem \ref{omegacategoricalsimplebutnotsupersimple} (by an automorphism, and either $\omega$-categoricity, or the fact that the $p_{n}$ we get in the following argument extend $\mathrm{tp}(a/\emptyset)$.)

Suppose that these $A_{n}$, $p_{n}$ do not exist for arbitrarily large $n$. Then there is some $n < \omega$ such that for $A$ any set and $p \in S^{k}(A)$, if $p$ is $m$-resolvable for some $m < \omega$, then it is $n$-resolvable. Let $|x| = n$; then by $\omega$-categoricity, there is a formula $\varphi(a, x)$ so that $\varphi(a, \mathbb{M}^{n})=\{c \in \mathbb{M}^{n}: \mathrm{SU}(\mathrm{tp}(a/c)) = \infty\} $, and for any finite $B_{0} \subset B$ there is a formula $\varphi^{B_{0}}(a, B_{0}, x)$ so that $\varphi^{B_{0}}(a, B_{0}, \mathbb{M}^{n})= \{c \in \mathbb{M}^{n}: a \ind_{c} B_{0} \}$. We will show that the set

$$\{\varphi(a, x)\} \cup \{\varphi^{B_{0}}(a, B_{0}, x): B_{0} \subset B, |B_{0}| < \omega \}$$

is consistent. By compactness, for any finite collection $S$ of finite subsets of $B$, we need only show 

$$\{\varphi(a, x)\} \cup \{\varphi^{B_{0}}(a, B_{0}, x): |B_{0}| \in S \}$$

is consistent. But by monotonicity of independence, $\varphi^{\cup S}(a, \cup S, x)$ implies $\varphi^{B_{0}}(a, B_{0}, x)$ for all $B_{0} \in S$, so we may assume $|S| = 1$, and it suffices to show that $\{ \varphi(a, x) \wedge \varphi^{B_{0}}(a, B_{0}, x)\}$ is consistent for an arbitrary finite $B_{0} \subseteq B$. Tautologically, the type $p= \mathrm{tp}(a/B_{0})$ is $|B_{0}|$-resolvable, and therefore, by assumption, $n$-resolvable. There is therefore some $c_{0} $ with $|c_{0}| \leq n$ (enumerable as a tuple of length $n$) and $q \in S(B_{0}c_{0})$ with $p \subseteq q$ that forks over neither $B_{0}$ or $c_{0}$; by an automorphism, we may further choose $c_{0}$ and $q$ so that $a \models q$. Since $\mathrm{SU}(p) = \infty$ and $q$ does not fork over $p$, $\mathrm{SU}(q)= \infty$; therefore, $\mathrm{SU}(q|_{c_{0}})= \mathrm{SU}(\mathrm{tp}(a/c_{0}))= \infty$. Moreover, because $q$ does not fork over $c_{0}$, $a \ind_{c_{0}} B_{0}$, so $c_{0} \models \{ \varphi(a, x) \wedge \varphi^{B_{0}}(a, B_{0}, x)\}$, showing consistency of $\{\varphi(a, x)\} \cup \{\varphi^{B_{0}}(a, B_{0}, x): B_{0} \subset B, |B_{0}| < \omega \}$. Now let $c$ realize this set; then $\mathrm{SU}(a/c) = \infty$ and $a \ind_{c} B_{0}$ for all $B_{0} \subset B$ with $|B_{0}| < \omega$; by finite character of independence, $a \ind_{c} B$. Therefore, $\mathrm{SU}(\mathrm{tp}(a/Bc)) = \infty $, so $\mathrm{SU}(\mathrm{tp}(a/B)) = \infty $, a contradiction. This concludes the proof of Theorem \ref{omegacategoricalsimplebutnotsupersimple}.
\end{proof}

\begin{remark}
   \label{omegacategoricitynotnecessary} In fact, $\omega$-categoricity is not necessary in Theorem \ref{omegacategoricalsimplebutnotsupersimple}, though we only need the $\omega$-categorical case to prove Theorem \ref{koponen}. We show this in the appendix. 
\end{remark}

We now prove Theorem \ref{koponen}. Each step of the proof will be an application of the following observation:
\begin{lemma}\label{fewindseq}  If $T$ has quantifier elimination in a finite relational language, then there is no tuple $b$ with $|b| < \omega$ and, for $i < \omega$, ($\emptyset$-)indiscernible sequences $I^{i}=(b^{i}_{j})_{j < \omega}$ with $b^{i}_{0} = b$, so that $I^{i} \not\equiv I^{j}$ for $i \neq j < \omega$.
\end{lemma}

\begin{proof}
We show that there are only finitely many types which $\emptyset$-indiscernible sequences $I = \{b_{i}\}_{i < \omega}$, where $|b_{0}| = b$, can realize over $\emptyset$, proving the lemma. By $n$-arity of $T$ and indiscernibility of $I$, $\mathrm{tp}(b_{0} \ldots b_{n-1})$ determines $\mathrm{tp}(I)$. But by $\omega$-categoricity, there are only finitely many possibilities for $\mathrm{tp}(b_{0} \ldots b_{n-1})$.

\end{proof}

It is then immediate from Theorem \ref{omegacategoricalsimplebutnotsupersimple} that, if $T$ is a simple theory with quantifier elimination in a finite relational language, then $T$ is supersimple. That if $T$ has quantifier elimination in a finite relational language and is supersimple, $T$ is one-based, can be proven in two ways, which we do next. 

The first proof uses the following result of Koponen, which generalizes the main result of \cite{HKP00}:

\begin{fact}\label{nontrivialregulartype} 

(Theorem 3.1, \cite{Kop19}) Let $T$ be $\omega$-categorical and supersimple, and suppose that forking is not geometricallly trivial (Definition \ref{geometricallytrivialforking}). Then there is a regular type over a finite set with a nontrivial pregeometry.
\end{fact}

We claim that if $T$ has quantifier elimination in a finite relational language and is supersimple, it suffices to show that every regular type over a finite set is linear. Suppose we have shown this; then either forking in $T$ is trivial, or by the fact, there is a nontrivial linear regular type over a finite set. But note the following result of \cite{TW03} (see the proof of Theorem 5.5.4 of \cite{Tomasic01}, for the case of $\mathrm{SU}(p)=1$):

\begin{fact}
    \label{linearitygroup}
    
    If $T$ is an $\omega$-categorical simple theory with a regular type over a finite set that is either $k$-linear for $k > 1$, or linear with nontrivial pregeometry, there is an infinite group interpretable in $T$.
\end{fact}
(For $k > 1$, this is the main ingredient of the proof of Fact \ref{pseudolinearity}; however, we only need to use $k=1$ directly here.) However, note also the following result of MacPherson (\cite{Mac91}).

\begin{fact}\label{nogroup}
    A theory with quantifier elimination in a finite relational language cannot interpret an infinite group.
\end{fact}
Therefore, there can be no nontrivial linear regular type, so we are in the case where forking in $T$ is trivial. (Note the similarity to the observation directly below Fact 2.6 of \cite{Kop17}, where the main result of \cite{HKP00} is used in place of \ref{nontrivialregulartype}, and \cite{DPK03} is used in place of Fact \ref{linearitygroup}.) However, Corollary 3.12 of \cite{Pal17}, restated in Proposition 3.6 of \cite{Kop19}, states that an omega-categorical supersimple theory with trivial forking must have finite rank. Thus, we can apply the observation between Definition 2.5 and Fact 2.6 of \cite{Kop17}:

\begin{fact}\label{onebasedminimal}
Let $T$ be an $\omega$-categorical, $n$-ary simple theory of finite rank. Then $T$ is one-based if and only if forking is geometrically trivial in $T$.

\end{fact}

So $T$ is one-based.

So it remains to show that every regular type over a finite set in $T$ is linear. If a regular type $p(x) \in S(A)$, with $A$ finite, is not linear, then by Corollary \ref{pseudolinearitymb}, $F_{Mb}(p\otimes p) > N$ for $p \otimes p$ some type realized by two $A$-independent realizations of $p$, and arbitrarily large $N$. (In fact, when $T$ is of finite rank, $F_{Mb}(p\otimes p) = \infty$ by Proposition \ref{Mboneinfinity}, or because the type $p$ produced in the proof in \cite{Kop17} of Fact \ref{nontrivialregulartype} can be chosen to be Lascar, we can apply the Lascar case of \ref{pseudolinearitymb} to get $F_{Mb}(p\otimes p) = \infty$.) But recalling the definition of $F_{Mb}(p\otimes p) $, this produces, for arbitrarily large $N$, an $A$-indiscernible sequence $\{a_{i}\}_{i < \omega}$ with $|a_{i}| = 2|x|$, any $N$ terms of which are independent, but any $N + 1$ terms of which are dependent, contradicting (recall that $A$ is assumed to be finite) Lemma \ref{fewindseq}. (Note, however, that it is not strictly necessary to use $F_{Mb}$ here: $p$ is pseudolinear by Fact \ref{pseudolinearity}, so $p\otimes p$ has forking extensions $q\in S(B)$, $B \subseteq A$ with canonical bases of unbounded $\mathrm{SU}_{p}$-rank, where $p \otimes p$ ranges over types over $A$ of two independent realizations of $p$ over $A$. By taking Morley sequences over $B$ in the extensions $q$, we get infinitely many indiscernible sequences realizing different types over $A$, because $\mathrm{SU}_{p}(\mathrm{Cb}(q))$ is an invariant of these indiscernible sequences by Remark \ref{rankofcanonicalbaseinvariant}.) This gives the first proof of the supersimple case of the Koponen conjecture, thereby proving Theorem \ref{koponen}.

In the second proof, we avoid having to use the full force of Fact \ref{pseudolinearity} and Corollary \ref{pseudolinearitymb}, restricting ourselves to the case where $\mathrm{SU}(p)=1$. The case where $T$ is supersimple of finite rank is proved as before: If $T$ is not one-based, then by Fact \ref{onebasedminimal}, forking is not geometrically trivial, so by Fact \ref{nontrivialregulartype} produces a type $p$ over a finite set with $\mathrm{SU}(p) = 1$ and nontrivial pregeometry. Then by the rank-one case of Fact \ref{linearitygroup} (which can be found in the proof of 5.5.4 of \cite{TW03}), and Fact \ref{nogroup}, $p$ cannot be linear; note that, by the proof in \cite{Kop17} of Fact \ref{nontrivialregulartype}, the type $p$ in that fact can clearly be chosen to be Lascar, so the main theorem of \cite{DPK03} will also give us an infinite interpretable group. Since $p$ is not linear and has $\mathrm{SU}(p)=1$, we can then apply Fact \ref{pseudolinearitymb}, but this time only the $\mathrm{SU}$-rank $1$ case, to reach a contradiction as in the first proof.

In the case where $T$ is supersimple, it remains to show that $T$ is of finite rank. But otherwise, there is a type $p$, such that $\mathrm{SU}(p)$ is a limit ordinal. Then by Corollary \ref{omegacategoricalsupersimple}, $F_{Mb}(p) = \infty$, again contradicting Lemma \ref{fewindseq}.\footnote{Observe that to get from the general supersimple case to the finite-rank case, we now apply $F_{Mb}$ without referring to canonical bases. However, note that by Corollary \ref{Mbpsquared}, in the case where $\mathrm{SU}(p) = 1$, $k$-linearity of $p$ is equivalent to the maximum value of $F_{Mb}(p \otimes p)$, where $p \otimes p$ ranges over the types of two independent realizations of $p$, being equal to $k$; in this sense, even the first, finite-rank step step of the proof constitutes an application of $F_{Mb}$. We may also get from the general supersimple case to the finite-rank case without referring to $F_{Mb}$. Namely, we show that any supersimple theory that is \textit{not} of finite rank cannot satisfy the conclusion of Lemma \ref{fewindseq}. Let $x$ be a single variable; then if $T$ is supersimple but not of finite rank, for each $n < \omega$ there is a type $p_{n}(x) \in S(A_{n})$ of $\mathrm{SU}$-rank $n$. Let $I_{n}$ be a Morley sequence over $A_{n}$ consisting of realizations of $p_{n}$. Then $\mathrm{SU}(\mathrm{lim}^{+}(I/I)) = n$. But then for $m < n < \omega$, $I_{m} \not \equiv I_{n}$, so the conclusion of Lemma \ref{fewindseq} fails.} 

We can generalize this question to $\mathrm{NSOP}_{1}$ theories (or even further; see \cite{GFA}, \cite{SOPEXP}.) The independence definition of being one-based (Fact \ref{forkingonebased}) can be generalized to $\mathrm{NSOP}_{1}$ theories:

\begin{definition}\label{onebased}
    Let $T$ be $\mathrm{NSOP}_{1}$. Then $T$ is \emph{one-based} if $A \nind^{K}_{M} B$ implies (equivalently, is equivalent to) $\mathrm{acl}^{eq}(AM) \cap \mathrm{acl}^{eq}(BM) \supsetneq M$.\footnote{To be certain that this definition extends the definition of being one-based in simple theories, we should have used $\mathrm{bdd}$ instead of $\mathrm{acl}^{eq}$. However, as related to us in a personal communication with Nicholas Ramsey, there is no known $\mathrm{NSOP}_{1}$ that does not eliminate hyperimaginaries, so in all known examples $\mathrm{acl}^{eq}$ coincides with $\mathrm{bdd}$.}
\end{definition}

We ask whether Koponen's conjecture holds for all $\mathrm{NSOP}_{1}$ theories.

\begin{question}
 \label{nsop1koponenconjecture}
  
  Let $T$ be $\mathrm{NSOP}_{1}$ and have quantifier elimination in a finite relational language. Is $T$ one-based?  
\end{question}

\section{Definability of Morley sequences and the simple Kim-forking conjecture}
\label{kimforking}

In the previous section, we characterized forking in a simple theory with quantifier elimination in a finite relational language, and asked (Question \ref{nsop1koponenconjecture}) whether this characterization extends to Kim-forking in $\mathrm{NSOP}_{1}$ theories with quantifier elimination in a finite relational language. By Lemma \ref{fewindseq}, a more immediate consequence of an $\mathrm{NSOP}_{1}$ theory having quantifier elimination in a finite relational language (i.e., being $\omega$-catgorical and $n$-ary) is that $F_{Mb}(p)$ is finite for $p$ a type over a finite set; here we extend $F_{Mb}$ to Kim-independence:

\begin{definition}\label{mbdef}
    Let $T$ be $\mathrm{NSOP}_{1}$. Let $p(x)$ be a complete type over a model $M$. Let $F_{Mb}(p)$ be the least $n < \omega$ so that, if $\{b_i\}_{i< \omega}$ is an $M$-indiscernible sequence with $b_i \models p(x)$ and $b_i \ind^{K}_{M} b_{<i}$ for all $i \leq n$, then $\{b_i\}_{i< \omega}$ is an $\ind^{K}$-Morley sequence over $M$; if such $n$ does not exist, define $F_{Mb}(p) = \infty$.
\end{definition}
However, in Section \ref{Mbfinitestable}, we saw that it is possible for $F_{Mb}$ to be finite for all types in a theory $T$, even a stable one, with more complicated forking (i.e., $a \ind_{C} b$ does not just coincide with $\mathrm{acl}(aC) \cap \mathrm{acl}(bC) = \mathrm{acl}(C)$). In this section, we will show that when a type $p$ in an $\mathrm{NSOP}_{1}$ theory shares some properties with types in theories eliminating quantifiers in a finite relational language, including $F_{Mb}(p) < \omega$ or definability of $\ind^{K}$-Morley sequences in $p$ (Definition \ref{definablemorleyproperty} below), the Kim-forking extensions of $p$ must still be described by generalizations of the \emph{stable forking conjecture}. The stable forking conjecture, due to Kim (see, e.g., \cite{Kim01} for the the definition of the ``stable forking property," the conclusion of this conjecture) states:

\begin{conj}\label{stableforkingconjecture} (Stable forking conjecture)
Let $T$ be a simple $\mathcal{L}$-theory. If $a \nind_{C}b$, there is a formula $\varphi(x, c) \in \mathrm{tp}(a/Cb)$ (where $c \subset Cb$) which forks over $C$, so that $\varphi(x, y)$ is a stable $\mathcal{L}$-formula.
\end{conj}

This conjecture can be generalized to $\mathrm{NSOP}_{1}$ theories. Since an $\mathrm{NSOP}_{1}$ theory is not necessarily simple, we can also demand only a simple formula that witnesses Kim-forking, rather than a stable formula.

\begin{question}\label{simplekimforkingconjecture} (Stable (or simple) Kim-forking-conjecture)
    Let $T$ be an $\mathrm{NSOP}_{1}$ $\mathcal{L}$-theory, and suppose $a \nind^{K}_{M} b$. Must there be a formula $\varphi(x, b) \in \mathrm{tp}(a/Mb)$, with $\varphi(x, y) \in \mathcal{L}(M)$ so that: 
    
    (a) $\varphi(x, b)$ Kim-forks over $M$.
    
    (b) $\varphi(x, y)$ is stable (or simple, i.e. lacks the tree property) as an $\mathcal{L}(M)$-formula?

\end{question}

For $\mathrm{NSOP}_{1}$ theories with existence (see \cite{DKR22}), where Kim-independence $a \ind_{C} b$ is well-defined for $C$ an arbitrary set, this question is easily modified to hold over sets. (Likewise, when $T$ is an $\mathrm{NSOP}_{1}$ theory with existence, and $p$ is a type over any set, we define $F_{Mb}(p)$ similarly to the case where $p$ is a type over a model.)

\begin{remark}
    The stable Kim-forking conjecture is referred to as the \textit{weak stable Kim-forking conjecture} in \cite{Bos24}, where Bossut observes that it is in fact necessary to modify the conjecture so that $\varphi(x, y)$ ranges over the stable formulas in $\mathcal{L}(M)$, rather than in $\mathcal{L}$: a theory satisfying the latter, stronger statement (i.e., the conclusion of Conjecture \ref{stableforkingconjecture}, where any reference to forking is replaced with Kim-forking) is in fact simple itself, so that statement is in fact false for strictly $\mathrm{NSOP}_{1}$ theories. Similarly, the statement of the simple Kim-forking conjecture must be modified so that $\varphi(x, y)$ ranges over the simple formulas in $\mathcal{L}(M)$. Suppose that that $T$ is $\mathrm{NSOP}_{1}$, and the conclusion of Conjecture \ref{stableforkingconjecture} holds, but replacing any reference to forking with Kim-forking and replacing stable with simple; we show $T$ must be simple. Let $M \prec M' \subset B$ and $a \nind^{K}_{M'}B$. Then by the assumption, there is a formula $\varphi(x, b) \in \mathrm{tp}(a/B)$ that Kim-forks over $M'$ with $\varphi(x, y)$ a simple $\mathcal{L}$-formula. Then $\varphi(x, b)$ divides over $M$, so because $\varphi(x, y)$ is a simple $\mathcal{L}$-formula, $\varphi(x, b)$ Kim-divides over $M$, and $a \nind^{K}_{M}B$. So $\ind^{K}$ is base-monotone in the sense of Proposition 8.8 of \cite{KR17}, so $T$ is simple by that proposition.
\end{remark}

In this section, we will focus on the simple Kim-forking conjecture. There
are other ways to weaken the stable Kim-forking conjecture besides
requiring only a simple formula, rather than a stable formula, to witness
Kim-forking. For example,  one  can assert that the forking (or
Kim-forking) relation itself, which is not necessarily definable, lacks the
order property or tree property.
In other words, we are interested in showing the stable forking conjecture, or stable (simple) Kim-forking conjecture, is true up to definability of the independence relation. Depending on how this assertion is formulated, this will weaken the stable forking conjecture or the stable (simple) Kim-forking conjecture.

For example, in Theorem 4.1 of \cite{PW13}, Palacin and Wagner prove stability of the forking \textit{relation}, which is not a priori definable, in supersimple CM-trivial theories. More precisely, for the forking relation $$R(x; y_{1} y_{2}) :=x \nind_{y_{1}} y_{2} ,$$ which is not necessarily definable, \cite{PW13} show that $R(x; y_{1}y_{2})$ is stable, or lacks the order property: there are no $\{a^{i}, b_{1}^{i}b_{2}^{i}\}_{i < \omega}$ so that $\models R(a^{i}, b_{1}^{j}b_{2}^{j})$ exactly when $i < j$.

\begin{remark}\label{globalweakerthanlocal}
    Another weaker formulation of stability of $R(x, y_{1}, y_{2}) =: x \nind_{y_{1}} y_{2} $ is to require that there are no $\{a^{i}, b_{1}^{i}b_{2}^{i}\}_{i < \kappa}$, for $\kappa$ large enough to apply the Erdős-Rado theorem, so that $\models R(a^{i}, b_{1}^{j}b_{2}^{j})$ exactly when $i < j$. Similarly, we can formulate stability for $R(x, y) =: x \nind^{K}_{M} y$, where $M$ is some fixed model, so that $R(x, y)$ is defined to be stable if there are no $\{a^{i} b^{i}\}_{i < \kappa}$ so that $\models R(a^{i}, b^{j})$ exactly when $i < j$. Under this formulation, stability of the forking relation, or Kim-forking relation, itself is weaker than the stable forking conjecture. For example, in the case where $R(x, y)$ is the Kim-forking relation, if the conclusion of the stable Kim-forking conjecture holds we can show that there are no $\{a^{i} b^{i}\}_{i < \kappa}$ so that $R(a^{i}, b^{j})$ exactly when $i < j$. If such $\{a^{i} b^{i}\}_{i < \kappa}$ existed, by the Erdős-Rado theorem we can assume $\{a^{i} b^{i}\}_{i < \kappa}$ is indiscernible over $M$. But by the assumption, whenever $i < j$, there is some formula $\varphi(x, b^{j}) \in \mathrm{tp}(a^{i}/Mb^{j})$, which Kim-forks over $M$ so that $\varphi(x, y)$ is stable as a $\mathcal{L}(M)$-formula; by indiscernibility, for $i < j$, $\varphi(x, y)$ does not depend on $i, j $. Then $ \models \varphi(a^{i}, b^{j})$ exactly when $i < j$, contradicting stability of $\varphi(x, y)$.
\end{remark}

\begin{remark}
    On the other hand, the proof of Corollary 4.4 of \cite{PW13} says that, if $T$ is $\omega$-categorical and supersimple, and the forking relation is assumed to be stable (even in the weaker sense described in the previous remark), the conclusion of the stable forking conjecture holds. The proof from \cite{PW13} is as follows: let $a$, $b$ be finite tuples with $a \nind_{C} b$ and, by supersimplicity, find finite $C_{0} \subset C$ so that $a b\nind_{C_{0}} C $. By $\omega$-categoricity, the relation $R(x, y_{1}y_{2}) =: x \nind_{y_{1}} y_{2}$ is definable by a formula $\varphi(x, y_{1}y_{2})$, and since $R(x, y_{1}y_{2})$ is a stable relation, $\varphi(x, y_{1}y_{2})$ is a stable formula. But $\varphi(x, C_{0} b) \in \mathrm{tp}(a/Cb)$, and $\varphi(x, C_{0} b)$ can be seen to fork over $C$, since it forks over $C_{0}$ and $b \ind_{C_{0}} C$.

    Moreover, when the forking relation is assumed to be stable, it is even easier to see that when $T$ is an $\omega$-categorical simple theory, the stable forking conjecture is true over finite sets: the stable formula witnessing this is just the forking relation itself. When $T$ is $\mathrm{NSOP}_{1}$ with existence (so Kim-forking is defined over all sets, including finite sets), and the Kim-forking relation is stable, the stable Kim-forking conjecture is true over finite sets.
\end{remark}

Having discussed what it means for the forking or Kim-forking relation to lack the order property, we must make precise what it means for the Kim-forking relation over a model (with variables for sets, say, of size less than some given bound) to lack the \textit{tree property}, that is, to be simple as a relation. To this end, we give a notation for consistency and inconsistency of relations that are not necessarily definable, in order to extend the definition of simplicity to these relations.

\begin{definition}
    Let $S$ and $T$ be arbitrary sets. Let $R \subseteq S \times T$ be a relation. Let $\{t_{i}\}_{i \in I} \subseteq T$ be an indexed subset of $T$. Then we use the notation that $\{R(x, t_{i})\}_{i \in I}$ is \textit{consistent} to mean that there is some $s \in S$ so that $(s, t_{i}) \in R$ for all $i \in I$. For $k < \omega$, we use the notation that $\{R(x, t_{i})\}_{i \in I}$ is $k$-\textit{inconsistent} to mean that for every $I_{0} \subseteq I$ with $|I_{0}| \leq k$, it is not the case that $\{R(x, t_{i})\}_{i \in I}$ is consistent.
\end{definition}

\begin{definition}\label{simplerelations}
    Let $\kappa_{1}$ be infinite cardinals, $S$ and $T$ arbitrary sets, and $R \subseteq S \times T$ a relation. Then $R$ is \textit{$(\kappa_{1}, \kappa_{2})$-simple} if there is no $k < \omega$ and tree-indexed subset $\{t_{\eta}\}_{\eta \in \kappa_{2}^{< \kappa_{1}}} \subseteq T$, so that, for all $\sigma \in \kappa_{2}^{\kappa_{1}}$, $\{R(x, t_{\sigma|_{\lambda}})\}_{\lambda < \kappa_1}$ is consistent, and for $\eta \in \kappa_{2}^{< \kappa_{1}}$, $\{R(x, t_{\eta \smallfrown \langle \lambda \rangle})\}_{\lambda < \kappa_{2}}$ is $k$-inconsistent.
\end{definition}

Note that if $S = \mathbb{M}^{n}$, $T = \mathbb{M}^{m}$ and $R =: \{(a, b):\mathbb{M}\models \varphi(a, b)\}$ (where $\varphi(x, y)$ is an $\mathcal{L}
$-formula with $|x| = n$ and $|y| = m$) is an $\mathcal{L}$-definable relation, then for $\kappa_{1}, \kappa_{2}$ any infinite cardinals, $R$ is  $(\kappa_{1}, \kappa_{2})$-simple if and only if $\varphi(x, y)$ is a simple formula. That is, Definition \ref{simplerelations} extends the usual notion of simplicity for formulas.

The main tool we will use is the following theorem of Kaplan and Ramsey (Corollary 5.4 of \cite{KR19}), extended to arbitrary sets in $\mathrm{NSOP}_{1}$ theories with existence by Chernikov, Kim and Ramsey in \cite{CKR20} (the hard direction, from (b) to (a), is Theorem 2.8 of \cite{CKR20}, and the easy direction is as in the easy direction of Corollary 5.4 of \cite{KR19}.)

\begin{fact}\label{witnessing}(Kaplan and Ramsey, \cite{KR19}, and Chernikov, Kim and Ramsey, \cite{CKR20})

Suppose $T$ is $\mathrm{NSOP}_{1}$. For any $M \prec \mathbb{M}$ and $b \subset \mathbb{M}$, if $\{b_{i}\}_{i < \omega}$ is an $M$-indiscernible sequence with $b_{0} = b$, then the following are equivalent:

(a) For all formulas $\varphi(x, b)$ that Kim-divide over $M$, $\{\varphi(x, b_{i})\}_{i < \omega}$ is inconsistent.

(b) The sequence $\{b_{i}\}_{i < \omega}$ is an $\ind^{K}$-Morley sequence over $M$: $b_{i} \ind_{M}^{K} b_{< i}$ for $i < \omega$.

If $T$  is $\mathrm{NSOP}_{1}$ with existence, then the above holds even when $M$ is assumed only to be a (small) subset of $\mathbb{M}$, and not necessarily a model.
    
\end{fact}

If we allow the Kim-forking relation $R$ to have infinitely many variables, then a global analogue of the simple Kim-forking conjecture turns out to hold in \textit{all} $\mathrm{NSOP}_{1}$ theories. Contrast this with the Theorem 4.1 of \cite{PW13}, where the forking relation with finitely many variables is shown to be stable (in the sense of sequences of length $\omega$), but $T$ is required to be not just simple but supersimple and CM-trivial.

\begin{theorem}\label{global1}
    Let $T$ be $\mathrm{NSOP}_{1}$. Let $M \prec \mathbb{M}$, and let $\kappa = 2^{|T|+|M|}$. Let $S_{1} = : \{a \subset \mathbb{M}: |a| < \kappa\}$, and let $S_{2} = \mathbb{M}^{n}$ for some fixed $n$. Let $R =: \{(a, b) \in S_{1} \times S_{2} : a \nind^{K}_{M} b \}$. Then $R$, as a relation on $S_{1} \times S_{2}$, is $(\kappa, \omega)$-simple.
\end{theorem}

In this proof and in the proofs of Theorems \ref{global2} and \ref{definabilitysimplekimforking}, it may have been possible to shorten the proofs using results on indiscernible trees; see, e.g. \cite{Sc15} or \cite{TT12} for an overview, as well as some new results in this area. To keep the proofs self-contained, we instead give direct proofs.

\begin{proof}
    Suppose not, and let $\{b_{\eta}\}_{\omega^{< \kappa}}$ witness this for some degree of inconsistency $k < \omega$. We must account for the notational subtlety about limit ordinals pointed out in Remark 5.8 of \cite{KR17}, so let $\mathrm{Succ}(\kappa)$ denote the successor ordinals, counting $0$, less than $\kappa$. By transfinite induction on $\kappa'$, we construct, for $\kappa' \leq \kappa$, some $\sigma_{\kappa'} \in \omega^{\kappa'}$ so that, for all $\lambda \in \mathrm{Succ}(\kappa')$, $b_{\sigma_{\kappa'}| \lambda} \ind_{M}^{K} \{b_{\sigma_{\kappa'}| \lambda'}\}_{\lambda' \in \mathrm{Succ}(\lambda)} $, and so that $\sigma_{\kappa''} \subset \sigma_{\kappa'}$ for all $\kappa'' < \kappa'$.  The case for $\kappa' = 0$ and $\kappa'$ a limit ordinal is clear, so it remains to assume $\sigma_{\kappa'}$ already constructed and find some $m < \omega$ so that $\sigma_{\kappa' +1}=: \sigma_{\kappa'} \smallfrown \langle m \rangle $ works. By the induction hypothesis, it suffices to find some $m < \omega$ so that $b_{\sigma_{\kappa'} \smallfrown \langle m \rangle} \ind_{M}^{K} \{b_{\sigma_{\kappa'}| \lambda'}\}_{\lambda' \in \mathrm{Succ}(\kappa')} $. Suppose otherwise, so by symmetry, $  \{b_{\sigma_{\kappa'}| \lambda'}\}_{\lambda' \in \mathrm{Succ}(\kappa')}  \nind_{M}^{K} b_{\sigma_{\kappa'} \smallfrown \langle m \rangle}$ for all $m < \omega$. Then $ a = \cup  \{b_{\sigma_{\kappa'}| \lambda'}\}_{\lambda' \in \mathrm{Succ}(\kappa)} $ has cardinality less than $\kappa$, so it witnesses that $\{R(x , b_{\eta \smallfrown \langle m \rangle})\}_{m < \omega }$ is consistent, when $\eta = \sigma_{\kappa'}$. In particular, $\{R(x , b_{\eta \smallfrown \langle m \rangle})\}_{m < \omega }$ cannot be $k$-inconsistent, contradicting that $\{b_{\eta}\}_{\eta \in \omega^{< \kappa}}$ is a witness to the failure of $(\kappa, \omega )$-simplicity of $R$.

This completes the construction, so let $\sigma = \sigma_{\kappa}$. Then $\{R(x, b_{\sigma|_{\lambda}})\}_{\lambda \in \mathrm{Succ}(\kappa)}$ is consistent. Reindexing, we get a sequence $\{ b_{i}\}_{i < \kappa}$ so that $\{R(x, b_i)\}_{i < \kappa}$ is consistent, and $b_{i} \ind^{K}_{M} b_{< i}$ for $i < \kappa$. So some $a \subseteq \mathbb{M}$, with $|a| < \kappa$, witnesses the consistency of $\{R(x, b_i)\}_{i < \kappa}$, meaning that $a \nind_{M}^{K} b_{i}$ for $i < \kappa$. This contradicts the following claim, essentially Lemma C.7 of \cite{KTW22}:

\begin{claim}\label{boundedweight}
    Let $|a| < \kappa$. Then there is no sequence $\{b_{i}\}_{i < \kappa}$, with $b_{i} \ind^{K}_{M} b_{< i}$ for $i < \kappa$, and $a \nind_{M}^{K} b_{i}$ for $i < \kappa$. 
\end{claim}

\begin{proof}
    First, we may assume, by choice of $\kappa$, that $b_i \equiv_{M} b_{j}$ for $i < j < \kappa$. Fix an enumeration $a=\{a_{i}\}_{i < \lambda}$ of $a$, where $\lambda < \kappa$, and let $\overline{x} = \{x_{i}\}_{i < \lambda}$; let $\overline{y} = \{y_{i}\}_{i < n}$. For each $i < \kappa$, let $\varphi_{i}(\overline{x}, b_{i}) \in \mathrm{tp}(a/Mb_{i}) $ Kim-fork over $M$. By choice of $\kappa$, and the pigeonhole principle, we can assume that $\varphi(\overline{x}, \overline{y}) =: \varphi_{i}(\overline{x}, \overline{y})$ does not depend on $i$. We need the following subclaim, which is folklore:

    \begin{subclaim}\label{typedefinablemorley}
        The set of sequences $\{b'_{i}\}_{i < \kappa}$ with $b'_{i} \models p(\overline{y}) =: \mathrm{tp}(b_{0}/M)$ and $b'_{i} \ind^{K}_{M} b'_{< i}$ for $i < \kappa$ is type-definable over $M$.
    \end{subclaim}

\begin{proof}
    By symmetry, this is the set of sequences $\{b'_{i}\}_{i < \kappa}$ with $b'_{i} \models p(\overline{y})$ and $b'_{<i} \ind^{K}_{M} b'_{ i}$ for $i < \kappa$. Let $S$ be the set of formulas $\psi(\tilde{x}, \overline{y})$ with parameters in $M$, where $\tilde{x}=\{x_{i}\}_{i < \omega}$, so that $\psi(\tilde{x}, b)$ Kim-forks over $M$ for some (any) $b \models p(\overline{y})$. Then the above set of sequences is just the set of sequences $\{b'_{i}\}_{i < \kappa}$ with $b'_{i} \models p(\overline{y})$ and no tuple from $b'_{<i}$ satisfying $\psi(\tilde{x}, b'_{i})$ for any $\psi(\tilde{x}, \overline{y})$ belonging to $S$; we easily see this is type-definable.
\end{proof}

Find an $Ma$-indiscernible, and a fortiori $M$-indiscernible, sequence $\{b'_{i}\}_{i < \omega}$ with the same EM-type as $\{b_{i}\}_{i < \kappa}$ over $Ma$; by the subclaim, this will be an $\ind^{K}$-Morley sequence over $M$, and $a \models \{\varphi(\overline{x}, b'_{i})\}_{i < \omega} $. So $\{\varphi(\overline{x}, b'_{i})\}_{i < \omega} $ is consistent, contradicting the direction (b) $\Rightarrow$ (a) of \ref{witnessing}.
    
\end{proof}
    
\end{proof}

\begin{remark}
    The only place in the above where we specifically needed $\kappa = 2^{|T|+|M|}$ is in assuming that  $b_i \equiv_{M} b_{j}$ for $i < j < \kappa$. If we had made $S_{2}$ the set of realizations of some complete $n$-type over $M$, we could have chosen $\kappa = (|T| + |M|)^{+}$.
\end{remark}

The previous result, while true of all $\mathrm{NSOP}_{1}$ theories, required infinitely many variables $x$ in the Kim-independence relation $x \nind^{K}_{M} y$; indeed,  for $(\omega, \kappa)$-simplicity, $x$ had to be a variable ranging over all sets of size $\kappa$. This differs from \cite{PW13}, where all variables are finite. Here we show that, when $y$ is a variable for a realization of $p(y) \in S(M)$ where $F_{Mb}(p) < \infty$ (Definition \ref{mbdef}), our results extend to a \textit{finitary} Kim-independence relation over a model. (Our arguments also give us a version, for $x$ ranging over finite sets, of the previous theorem on general $\mathrm{NSOP}_{1}$ theories.)

\begin{theorem}\label{global2}
    Let $T$ be $\mathrm{NSOP}_{1}$, and let $\kappa = (|M| + |T|)^{+}$. Fix a complete $n$-type $p(y)$ over a model $M$. Let $F_{Mb}(p)\leq m < \omega+1$; if $m < \omega$, let $S_{1}= \mathbb{M}^{n(m-1)}$, and if $m = \omega$, let $S_{1}= \{a \subseteq \mathbb{M}: |a| < \omega\}$. Let $S_{2} = \{b \in \mathbb{M}^{n}: b \models p(y)\}$, and let $R =: \{(a, b) \in S_{1} \times S_{2} : a \nind^{K}_{M} b \}$. Then $R$, as a relation on $S_{1} \times S_{2}$, is $(\kappa, \kappa)$-simple.
\end{theorem}

\begin{proof}
    Suppose not, and let $\{b_{\eta}\}_{\omega^{< \kappa}}$ witness this for some degree of inconsistency $k < \omega$. By transfinite induction on $\kappa'$, we construct, for $\kappa' \leq \kappa$, some $\sigma_{\kappa'} \in \kappa^{\kappa'}$ so that, for all $\lambda \in \mathrm{Succ}(\kappa')$, $b_{\sigma_{\kappa'}| \lambda} \ind_{M}^{K} \{b_{\sigma_{\kappa'}| \lambda'}\}_{\lambda' \in I_{0}} $ for every $I_{0} \subset \mathrm{Succ}(\kappa')$ with $|I_{0}| < m$. Similarly to before, it remains to assume $\sigma_{\kappa'}$ already constructed and find some $\gamma < \kappa$ so that $\sigma_{\kappa' +1}=: \sigma_{\kappa'} \smallfrown \langle \gamma \rangle $ works.  By the induction hypothesis, it suffices to find some $\gamma < \kappa$ so that $b_{\sigma_{\kappa'} \smallfrown \langle \gamma \rangle} \ind_{M}^{K} \{b_{\sigma_{\kappa'}| \lambda'}\}_{\lambda' \in I_{0}} $ for every $I_{0} \subset \mathrm{Succ}(\kappa')$ with $|I_{0}| < m$. Suppose otherwise, so by symmetry, for all $\gamma < \kappa$, there is some $I_{0} \subset \mathrm{Succ}(\kappa')$ with $|I_{0}| < m$ so that $  \{b_{\sigma_{\kappa'}| \lambda'}\}_{\lambda' \in I_{0}}  \nind_{M}^{K} b_{\sigma_{\kappa'} \smallfrown \langle \gamma \rangle}$. By the pigeonhole principle and choice of $\kappa$, there is some $I_{0} \subset \mathrm{Succ}(\kappa')$ with $|I_{0}| < m$ so that there is an infinite set $I \subseteq \kappa$ so that, for every $\gamma \in I$, $  \{b_{\sigma_{\kappa'}| \lambda'}\}_{\lambda' \in I_{0}}  \nind_{M}^{K} b_{\sigma_{\kappa'} \smallfrown \langle \gamma \rangle}$. Then $ \cup \{b_{\sigma_{\kappa'}| \lambda'}\}_{\lambda' \in I_{0}}$ can be indexed as a tuple in $S_{1}$, so it witnesses that $\{R(x , b_{\eta \smallfrown \langle \gamma \rangle})\}_{\gamma \in I }$ is consistent, for $\eta = \sigma_{\kappa'}$ In particular, $\{R(x , b_{\eta \smallfrown \langle \gamma \rangle})\}_{\gamma < \kappa }$ cannot be $k$-inconsistent.

    This completes the construction, so let $\sigma = \sigma_{\kappa}$. Then $\{R(x, b_{\sigma|_{\lambda}})\}_{\lambda \in \mathrm{Succ}(\kappa)}$ is consistent. Reindexing, we get a sequence $\{ b_{i}\}_{i < \kappa}$ so that $\{R(x, b_i)\}_{i < \kappa}$ is consistent, and $b_{i} \ind^{K}_{M} \{b_{i}\}_{i \in I_{0}}$ for every $i < \kappa$ and $I_{0} \subset [i]$ with $|I_{0}| < m$. If $m = \omega$ then by the finite character of $\ind^{K}$, this just says $b_{i} \ind^{K}_{M} b_{< i}$ for every $i < \kappa$. Again, there is some $a \subseteq \mathbb{M}$, with $|a| < \kappa$ (indeed, $|a| \leq (m-1)n$), which witnesses the consistency of $\{R(x, b_i)\}_{i < \kappa}$, meaning that $a \nind_{M}^{K} b_{i}$ for $i < \kappa$. If $m = \omega$, we proceed exactly as in the above. In case $m < \omega$, a contradiction follows from the following version of Claim \ref{boundedweight}:

    \begin{claim}
        Let $|a| < \kappa$. Then there is no sequence $\{b_{i}\}_{i < \kappa}$, with  $b_{i} \ind^{K}_{M} \{b_{i}\}_{i \in I_{0}}$ for every $i < \kappa$ and $I_{0} \subset [i]$ with $|I_{0}| < m$, and $a \nind_{M}^{K} b_{i}$ for $i < \kappa$. 
    \end{claim}

    \begin{proof}
        That $b_i \equiv_{M} b_{j}$ for $i < j < \kappa$ is now by definition on $S_{2}$. Again, we may assume there is some formula $\varphi(\overline{x}, \overline{y})$ so that $a \models \{\varphi(x, b_{i})\}_{i < \kappa}$ and $\varphi(x, b)$ Kim-forks over $M$ for $b \models p(\overline{y})$. The proof of the following is a straightforward generalization of the proof of Subclaim \ref{typedefinablemorley}:

\begin{subclaim}
   \label{nindependencetypedefinable}
    
    The set of sequences $\{b'_{i}\}_{i < \kappa}$ with $b'_{i} \models p(\overline{y}) =: \mathrm{tp}(b_{0}/M)$, and with  $b_{i} \ind^{K}_{M} \{b_{i}\}_{i \in I_{0}}$ for every $i < \kappa$ and $I_{0} \subset [i]$ with $|I_{0}| < m$, is type-definable over $M$.
\end{subclaim}

    Again, extract an $Ma$-indiscernible sequence $\{b'_i\}_{i < \omega}$ sharing the same EM-type as $\{b_i\}_{i < \kappa}$ over $Ma$. By the subclaim, $\{b'_i\}_{i < \omega}$ will in particular be an $M$-indiscernible sequence with $b'_{i} \ind^{K}_{M} \{b'_{i}\}_{i \in I_{0}}$ for every $i < \omega$ and $I_{0} \subset [i]$ with $|I_{0}| < m$. Because $F_{Mb}(p) \leq m$, $\{b'_i\}_{i < \omega}$ will be an $\ind^{K}$-Morley sequence over $M$. Conclude as in the proof of Claim \ref{boundedweight}.
        
    \end{proof}

\end{proof}

\begin{remark}
     The proof of the general part of Theorem \ref{global2} should also hold when the argument $x$ of $R(x, y)$ is allowed to range over, for example, countable sets. In \textit{all} $\mathrm{NSOP}_{1}$ theories, the Kim-independence relation $x \nind^{K}_{M} y$ can then be seen by Theorems \ref{global1} and \ref{global2} to lack the $(\kappa, \kappa)$-tree property for sufficiently large $\kappa$ as long as $x$ is a variable for an infinite or unbounded finite set. However, when $x$ and $y$ are both finite sets of variables, it is unknown in general whether the relation $x \ind_{M} y$ is $(\kappa, \kappa)$-simple for large enough $\kappa$. Like in Remark \ref{globalweakerthanlocal}, we still do see that $(\kappa, \kappa)$-simplicity of the relation $R(x, y)=: \{a \nind^{K}_{M} b, a \in \mathbb{M}^{m}, b \in \mathbb{M}^{n}\}$, for $m$, $n$ finite and for large enough $\kappa$, does follow from the stable Kim-forking conjecture itself. To see this, when $\kappa$ is large enough and there is a tree witnessing failure of the $(\kappa, \kappa)$-tree property, if the simple Kim-forking conjecture is true we can find a subtree where the consistency of the paths is witnessed by a single simple $\mathcal{L}(M)$-formula $\varphi(x, y)$, but this subtree will witness that $\varphi(x, y)$ has the tree property, a contradiction. (To find this subtree, we can, say, choose an $\omega^{|T|^{+}}$-subtree inductively so that all paths have the same type, then use the fact that the consistency of this path will involve the same formula appearing infinitely many times to get a further subtree.) So as with stability, $(\kappa, \kappa)$-simplicity of the Kim-independence relation gives a weakening of the simple Kim-forking conjecture. Moreover, when $T$ is $\omega$-categorical and $\mathrm{NSOP}_{1}$ with existence (so, again, Kim-forking is defined over finite sets), $(\omega, \omega)$-simplicity of the $R(x, y)=: \{a \nind^{K}_{C} b, a \in \mathbb{C}^{m}, b \in \mathbb{M}^{n}\}$ for $m, n$ finite is equivalent to the simple Kim-forking cojecture over $C$, when $C$ is a finite set.
\end{remark}

So far, we have discussed analogues of the simple Kim-forking conjecture for the non-definable relation of Kim-forking. This leads to ask: what about the simple Kim-forking conjecture itself? We will will show that, in $\mathrm{NSOP}_{1}$ theories with existence, when $\mathrm{tp}(b/C)$ satisfies the property that the notion of an $\ind^{K}$-Morley sequence in this type is definable, forking with $b$ over $C$ satisfies the simple Kim-forking conjecture. By Example \ref{genericprojectiveplanes} below, we will see that this gives us some nontrivial instances of the simple Kim-forking conjecture. The following property of a type $p$ strengthens the property that $F_{Mb}(p)$ is finite:

\begin{definition}\label{definablemorleyproperty}
    Let $T$ be $\mathrm{NSOP}_{1}$ with nonforking existence, and let $C \subseteq \mathbb{M}$ be a (small) set, and let $p(x)$ be a complete type over $C$. Then $p(x)$ \textit{has definable $\ind^{K}$-Morley sequences} if the set of $\ind^{K}$-Morley sequences $\{b_i\}_{i < \omega}$ over $C$ with $b_{i} \models p(x)$ is definable over $C$ relative to the (necessarily type-definable) set of indiscernible sequences $\{b_i\}_{i < \omega}$ over $C$ with $b_{i} \models p(x)$. (That is, for $S_{1}$ the set of $\ind^{K}$-Morley sequences $\{b_i\}_{i < \omega}$ over $C$ with $b_{i} \models p(x)$, and $S_{2}$ the set of $C$-indiscernible sequences $\{b_i\}_{i < \omega}$ with $b_{i} \models p(x)$, there is a formula $\varphi(x_{1}, \ldots x_{n})$ with parameters in $C$, with $|x_{i}| = |b_{i}|=k$, so that $S_{1} = S_{2} \cap \{\{b_{i}\}_{i < \omega} \in (\mathbb{M}^{k})^{\omega} : \mathbb{M} \models \varphi(b_{1}, \ldots, b_{n}) \}$.)
\end{definition}

\begin{remark}\label{definablemorleysequencesimpliesfinitemb}
  To see that if $p \in S(C)$ has definable $\ind^{K}$-Morley sequences, then $F_{Mb}(p) < \infty$, assume the former. For a $C$-indiscernible sequence $I= \{a_{i}\}_{i < \omega}$, $a_{i} \models p(x)$, to be an $\ind^{K}$--Morley sequence is equivalent to it being $n$-independent for all $n < \omega$. But by the proof of Claim \ref{typedefinablemorley}, being $n$-independent is type-definable relative to $I$ being an indiscernible sequence over $C$. So by definable $\ind^{K}$-Morley sequences and compactness, there is some $n < \omega$ so that if $I$ is $n$-independent, it is an $\ind^{K}$-Morley sequence, and therefore $F_{Mb}(p) < \infty$.
\end{remark}

We obtain an instance of the simple Kim-forking conjecture when, for $a\nind_{C} b$,  $\mathrm{tp}(b/C)$ has definable $\ind^{K}$-Morley sequence and $a$ is given as a tuple \textit{enumerated} by a sufficiently long finite sequence (thereby giving us enough variables to build the simple formula.)

\begin{theorem}
   \label{definabilitysimplekimforking}
    
    Let $T$ be $\mathrm{NSOP}_{1}$ with existence\footnote{Note that, as shown by the third author in \cite{Exist}, not all $\mathrm{NSOP}_{1}$ theories have existence; whether this is true is a formerly open problem posed in \cite{DKR22}.}. Let $C, b \subset M$ be such that $p(y) =: \mathrm{tp}(b/C)$ has definable $\ind^{K}$-Morley sequences, and suppose that either $p(y)$ is isolated or $T$ is low. Then there is some $n < \omega$ so that, for any $a = \{a_i\}_{i \in I}$ where $|I| \geq n$ (with the $a_{i}$ allowed to repeat), and such that $a \nind^{K}_{C} b$, there is some formula $\varphi(x, b) \in \mathrm{tp}(a/Cb)$ that Kim-forks over $C$, where $\varphi(x, y)$ is a formula with parameters in $C$ which is simple as a formula of $\mathcal{L}(C)$. 
\end{theorem}

\begin{proof}
    We first show that if there is a formula $\tilde{\varphi}(x, b) \in \mathrm{tp}(a/Cb)$, Kim-forking over $C$, so that

    (*) there is no $k < \omega$ and tree-indexed set $\{b_{\eta}\}_{\eta \in \omega^{< \omega}}$, with $b_{\eta} \models p(y)$ for all $\eta \in \omega^{< \omega}$, so that for all $\sigma \in \omega^{\omega}$, $\{\tilde{\varphi}(x, b_{\sigma|_{i}})\}_{i < \omega}$ is consistent, but for all $\eta \in \omega^{< \omega}$, $\{\tilde{\varphi}(x, b_{\eta \smallfrown \langle i \rangle})\}_{i < \omega}$ is $k$-inconsistent.

then there is a formula $\varphi'(x, b) \in \mathrm{tp}(a/Cb)$, Kim-forking over $C$, so that $\varphi'(x, y)$ is a formula with parameters in $C$ that is simple as a formula of $L(C)$. If $p(y)$ is isolated, then $\varphi'(x, b) =: \tilde{\varphi}(x, b) \wedge p(b)$ will work. If $T$ is low, then there is some $k$ so that if $\{c_i\}_{i < \omega}$ is \textit{any} indiscernible sequence, and $\{\tilde{\varphi}(x, c_{i})\}_{i < \omega}$ is $k$-inconsistent, then it is inconsistent. Then for $\{c_{\eta}\}_{\eta \in \omega^{<\omega}}$ to witness failure of simplicity for $\tilde{\varphi}(x, y)$ as a formula of $L(C)$ means that for all $\sigma \in \omega^{\omega}$, $\{\tilde{\varphi}(x, c_{\sigma|_{i}})\}_{i < \omega}$ is consistent, but for all $\eta \in \omega^{< \omega}$, $\{\tilde{\varphi}(x, c_{\eta \smallfrown \langle i \rangle})\}_{i < \omega}$ is $k$-inconsistent. The property $(*)$ says that $c_{\eta} \models p(y)$ for all $\eta \in \omega^{<\omega}$ implies that $\{c_{\eta}\}_{\eta \in \omega^{<\omega}}$ does not witness the failure of simplicity of $\tilde{\varphi}(x, y)$, but we see now that to witness the failure of simplicity of $\tilde{\varphi}(x, y)$ is type-definable over $C$. So by compactness, there is some $\rho(y) \in p(y)$ so that $c_{\eta} \models \rho(y)$ for all $\eta \in \omega^{<\omega}$ implies that $\{c_{\eta}\}_{\eta \in \omega^{<\omega}}$ does not witness the failure of simplicity of $\tilde{\varphi}(x, y)$. Then $\varphi'(x, b) =: \tilde{\varphi}(x, b) \wedge \rho(b)$ will work.

As in the proof of Subclaim \ref{typedefinablemorley}, let $S$ be the set of formulas $\varphi(x, y)$ with parameters in $C$ so that $\varphi(x, b)$ Kim-divides over $C$; then an indiscernible sequence $\{b_i\}_{i < \omega}$ over $C$, with $b_i \models p(y)$, is an $\ind^{K}$-Morley sequence over $C$ if and only if, for all $i < \omega$, no tuple from $b_{<i}$ satisfies $\varphi(x, b_i)$ for $\varphi(x, y) \in S$. So by definability of $\ind^{K}$-Morley sequences and compactness, there is some finite $x$ and finite subset $\{\varphi_j(x, y)\}_{j < m} \subset S$ so that an indiscernible sequence $\{b_i\}_{i < \omega}$ over $C$, with $b_i \models p(y)$, is an $\ind^{K}$-Morley sequence over $C$ if and only if, for all $i < \omega$, no tuple from $b_{<i}$ satisfies $\varphi_j(x, b_i)$ for all $j < m$.

        Let $\psi(x, y) = \bigvee_{i<m} \varphi_i(x, y)$. Let $n= |x|$. Let $a = \{a_i\}_{i \in I}$ where $|I| \geq n$ (with the $a_{i}$ allowed to repeat), and with $a \nind^{K}_{C} b$. We may assume $\{x_i\}_{i \in I}=x' \supseteq x$. Let $\psi'(x', b) \in \mathrm{tp}(a/bC)$ Kim-fork over $C$. Then $\tilde{\varphi}(x', b)=:\psi(x, b) \vee \psi'(x', b) \in \mathrm{tp}(a/bC)$, and $\tilde{\varphi}(x, b)$ Kim-forks over $C$. It suffices to show that $\tilde{\varphi}(x', b)$ has the property (*). Suppose $\{b_{\eta}\}_{\eta \in \omega^{< \omega}}$, with $b_{\eta} \models p(y)$, witness the failure of (*) for $\tilde{\varphi}(x', b)$ and degree of inconsistency $k$. By compactness, we can find $\{b_{\eta}\}_{\eta \in \kappa^{< \omega}}$, where $\kappa = |b|^{+}$ (or $\kappa = \omega$ if $b$ is finite) , so that for all $\sigma \in \kappa^{\omega}$, $\{\tilde{\varphi}(x', b_{\sigma|_{i}})\}_{i < \omega}$ is consistent, but for all $\eta \in \kappa^{< \omega}$, $\{\tilde{\varphi}(x', b_{\eta \smallfrown \langle i \rangle})\}_{i < \kappa}$ is $k$-inconsistent. In particular, for all $\eta \in \kappa^{< \omega}$ and all $j < m$, $\{\varphi_{j}(x, b_{\eta \smallfrown \langle i \rangle})\}_{i < \kappa}$ is $k$-inconsistent. We claim that for all $\eta \in \kappa^{\omega}$, there is some $\gamma < \kappa$ so that, for all $j < m$, $\varphi_j(x, b_{\eta \smallfrown \langle \gamma \rangle})$ is satisfied by no tuple from $\{b_{\eta|_i}\}_{i < | \eta|}$. Suppose otherwise; because $|\{b_{\eta|_i}\}_{i < | \eta|}| < \kappa$, by the pigeonhole principle there is some fixed $j < m$ and tuple from $\{b_{\eta|_i}\}_{i < | \eta|}$ satisfying $\varphi_j(x, b_{\eta \smallfrown \langle \gamma \rangle})$ for infinitely many $\gamma$, contradicting $k$-inconsistency of $\{\varphi_{j}(x, b_{\eta \smallfrown \langle i \rangle})\}_{i < \kappa}$. This proves our claim; it follows by induction that there is some $\sigma \in \kappa^{< \omega}$ so that for all $i < \omega$ and $j < m$, $\varphi_j(x, b_{\sigma|_{i}})$ is satisfied by no tuple from $\{b_{\sigma|_{i'}}\}_{i' < i }$. Extracting a $C$-indiscernible sequence $\{b_{i}\}_{i < \omega}$ satisfying the same EM-type over $C$ as $\{b_{\sigma|_i}\}_{i < \omega }$, we see that for all $i < \omega$ and $j < m$, $b_{i} \models p(y)$ and $b_{<i}$ satisfies $\varphi_j(x, b_i)$--so $\{b_{\sigma|_i}\}_{i < \omega }$ is an $\ind^{K}$-Morley sequence over $C$ by choice of the $\varphi_{j}(x, y)$--and $\{\tilde{\varphi}(x', b_{i})\}_{i < \omega}$ is consistent, though $\tilde{\varphi}(x', b)$ Kim-forks over $C$ with $b \models p(y)$, contradicting the direction (b) $\Rightarrow$ (a) of Theorem \ref{witnessing}.

   \end{proof}

   Note that if we get rid of the assumption that $\mathrm{tp}(b/C)$ is isolated or $T$ is low, the formula $\varphi(x, b)$ obtained in the above proof, though not known to be simple as a formula in $\mathcal{L}(C)$, still does not have the tree property given by a tree $\{b_{\eta}\}_{\eta \in \omega^{< \omega}}$ where $b_{\eta} \models p$. 

   In the setting of the $\mathrm{NSOP}_{1}$ variant of the Koponen conjecture, Question \ref{nsop1koponenconjecture}, we now have 
   the following partial result as a corollary of Theorem \ref{definabilitysimplekimforking}:

   \begin{cor}
       Let $T$ be $\mathrm{NSOP}_{1}$ with existence, and have quantifier elimination in a finite relational language. Let $C$ be finite. Then for any finite tuple $b$, there is some $n < \omega$ so that, for any $a = \{a_i\}_{i \in I}$ where $|I| \geq n$ (with the $a_{i}$ allowed to repeat), and such that $a \nind^{K}_{C} b$, there is some formula $\varphi(x, b) \in \mathrm{tp}(a/Cb)$ that Kim-forks over $C$, where $\varphi(x, y)$ is a formula with parameters in $C$ which is simple as a formula of $\mathcal{L}(C)$. 
   \end{cor}

   This is because, by Lemma \ref{fewindseq}, any type over a finite set under the above assumptions (which will of course be isolated) must have definable $\ind^{K}$-Morley sequences.
   
   Just as we defined types with algebraic forking in simple theories (Definition \ref{algfork}), we can ask whether there are types $p \in S(C)$ in an $\mathrm{NSOP}_{1}$ theory with definable $\ind^{K}$-Morley sequences, but without algebraic Kim-forking--i.e. there are $a \models p$, $a \nind^{K}_{C} B$ so that $\mathrm{acl}^{eq}(aC) \cap \mathrm{acl}^{eq}(BC) = \mathrm{acl}^{eq}(C)$, so not all Kim-forking extensions of $p$ are given by the dependence relation coming from $\mathrm{acl}^{eq}$. We give an example of such a $p$; this will be isolated (and we suspect $T$ is low, which would mean it is not necessary that $p$ is isolated), so we get a nontrivial instance of the simple Kim-forking conjecture from Theorem \ref{definabilitysimplekimforking}.
   \begin{example}\label{genericprojectiveplanes}
            The following example is due to Conant and Kruckman (\cite{CoK19}). We show that it gives us examples of an $\mathrm{NSOP}_{1}$ theory $T$ (with existence) so that $1 <\mathrm{max}(\{F_{Mb}(p): p \in S(A), A \subset \mathbb{M}\}) < \infty $, and with isolated types $p \in S(A)$ with $a \models p$ and $A \subset B$ so that $B \nind_{A}^{K} a$ and $\mathrm{bdd}(aA) \cap \mathrm{bdd}(B) = \mathrm{bdd}(A)$ (i.e., $p$ does not have algebraic Kim-forking), and $p$ has definable Morley sequences. Since $p$ has forking extensions where the forking is not witnessed by algebraic dependence, but satisfies the hypotheses of theorem \ref{definabilitysimplekimforking}, this gives an instance of the simple Kim-forking conjecture that is not covered by Kim's result in \cite{Kim01} that forking extensions where the forking is witnessed by algebraic dependence satisfy the stable forking conjecture. Moreover, two other existing results on the stable forking conjecture, those of Palacin and Wagner in \cite{PW13} showing that $\omega$-categorical CM-trivial supersimple theories satisfy the stable forking conjecture, and of Peretz in \cite{P96} showing that the stable forking conjecture holds between $\mathrm{SU}$-rank $2$ types in $\omega$-categorical supersimple theories, both require $\omega$-categoricity to obtain a formula $\varphi(x, y, z)$ so that for $a, b, c$ with $|a|=|x|$, $|b| = |y|$ and $|c| = |z|$, $\models \varphi(a, b, c)$ if and only if $a \ind_{c} b$. But for finitely many variables $|x|$ and $a \models p$, it will follow from the characterization of Kim-forking in this example that there is no formula $\varphi(x, a, A)$ so that for $|b| = |x|$, $\models \varphi(b, a, A)$ if and only if $b\ind_{A}^{K} a$.  Finally, Brower in \cite{Brower2012} obtains an instance of the stable forking conjecture in simple theories eliminating hyperimaginaries where $a \nind_{C} b$ with $\mathrm{SU}(a/C) < \omega$ and $\mathrm{SU}(b/C)=2$, though the formula obtained is only stable in the sense of $\mathcal{L}(C)$; to do so, he obtains a quasidesign between $a$ and $b$ (see Remark \ref{brower} below). Though in our example, Kim-forking is always given by a quasidesign, the simple formula $\varphi(x, y)$ obtained in Theorem \ref{definabilitysimplekimforking} does not obviously define a quasidesign.

We recall the example of \cite{CoK19}. Let $\mathcal{L}$ be a language with two sorts $P$ and $L$ and an incidence relation $I$ between singletons of $L$ and singletons of $P$. For $m, n > 1$, Conant and Kruckman show in \cite{CoK19} that the theory of $\mathcal{L}$-structures omitting the complete bipartite graph $K_{m, n}$ between $m$ points of $P$ and $n$ points of $L$ has a model companion $T^{m, n}$. Call a $\mathcal{L}$ structure $A$ \textit{closed} if for distinct $p_{1}, \ldots, p_{m} \in P(A)$ there are distinct $l_{1}, \ldots, l_{n-1} \in L(A)$ so that $\models p_{i} I l_{j}$ for all $1 \leq i \leq m$, $1 \leq j \leq n-1$, and for distinct $l_{1}, \ldots, l_{n} \in L(A)$ there are distinct $p_{1}, \ldots, p_{m-1} \in P(A)$ so that $\models p_{i} I l_{j}$ for all $1 \leq i \leq m-1$, $1 \leq j \leq n$. Then \cite{CoK19} show that the algebraically closed sets in $\mathbb{M}$ are exactly the closed sets, and that $T$ is $\mathrm{NSOP}_{1}$, with $a \ind^{K}_{A} b$ exactly when $\mathrm{acl}(aA) \cap \mathrm{acl}(bA) = A$ and there are no $a' \in \mathrm{acl}(aA) \backslash A$, $b' \in \mathrm{acl}(bA) \backslash A$ so that $\models a' I b'$ or $b' I a'$. Isomorphic algebraically closed sets will satisfy the same type.

Assume without loss of generality that $m \leq n$. Suppose $a \not\subseteq \mathrm{acl}(A)$, and let $p = \mathrm{tp}(a/A)$. We show that if $\mathrm{acl}(aA) \backslash \mathrm{acl}(A)$ contains only points of $P$, $F_{Mb}(p) = m$, and that if $\mathrm{acl}(aA) \backslash \mathrm{acl}(A)$ contains points of $L$, $F_{Mb}(p) = n$. Without loss of generality suppose we are in the second case. We may also assume that $A \subseteq a$ and that $a$, $A$ are algebraically closed. Fix a singleton $a_{0} \in L(a) \backslash A$. We first show that $F_{Mb}(p) \geq n$. We can easily construct a $K_{m, n}$-free $\mathcal{L}$-structure $B=\bigcup_{S \subset \omega, |S| \leq n-1} B_{S}$ satisfying the following:

(a) For all $S, T \subset \omega$ with $|S|, |T| \leq n-1$, $B_{S} \cap B_{T} = B_{S \cap T}$, and for all $i < \omega$ and $b^{i} = B_{\{i\}}$, there are $b^{i}_{0} \in b^{i}$ so that $b^{i}_{0}B^{i}B_{\emptyset}$ is isomorphic to $a_{0}aA$;

(b) For all $S \subset \omega$ with $|S| \leq n-1$, $B_{S}$ is closed, and is a minimal closed set containing $\cup_{i \in S}b^{i}$.

(c) For all $S \subset \omega$ with $|S| \leq n-1$ and $i \in \omega \backslash S$, there are no singletons $b' \in b^{i} \backslash B_{\emptyset}$ and $b'' \in B_{S}$ with $(b', b'') \in I(B)$ or $(b'', b') \in I(B)$.

(Just choose freely so that (a) and (b) are satisfied; for example, if $A = \emptyset$ and $a$ is a singleton of $L$, $B$ just consists of infinitely many $L$-points.)

In $B$, there is no $b^{*} \in P(B)$ so that, for at least $n$ values of $i < \omega$, $(b^{*}, b^{i}_{0}) \in I(B)$: such a $b^{*}$ must lie in some $B_{S}$, but then for some $i \notin S$, $(b^{*}, b^{i}_{0}) \in I(B)$, contradicting (c). So we may extend $B$ to a $K_{m, n}$-free $\mathcal{L}$-structure $B \sqcup \{b^{*}\}$, so that $(b^{*}, b) \in I(B \cup \{b^{*}\})$ for $b \in B$ if and only if $b = b^{i}_{0}$ for some $i < \omega$.

Now embed $B \sqcup \{b^{*}\}$ into $\mathbb{M}$ so that the image of $B_{\emptyset}$ is $A$; let $a^{i}$ be the image of $b^{i}$, $a^{i}_{0}$ the image of $b^{i}_{0}$, $a^{*}$ the image of $b^{*}$. Then for $0 \leq i_{1} < \ldots < i_{n+1} < \omega$, $a^{i_{j}} \ind^{K}_{A} a^{i_{1}} \ldots a_{i_{j-1}}$ for $1 \leq j \leq n$, by (a)-(c). However, $a^{i_{n+1}} \nind^{K}_{A} a^{i_{1}} \ldots a^{i_{n}}$, because $a^{*} \in \mathrm{acl}( a^{i_{1}} \ldots a^{i_{n}})$ and $\models I(a^{*}, a_{0}^{i_{n+1}})$. Extracting an $A$-indiscernible sequence, we see $F_{Mb}(p) \geq m$.

Now let us show that $F_{Mb}(p) \leq n$: let $I = \{a_{i}\}_{i < \omega}$ be an $A$-indiscernible sequence of realizations of $p$, and suppose that $I$ is not a Morley sequence over $A$. It suffices to show that $a_{ n} \nind^{K}_{A} a_{< n}$. We may again assume that $A \subseteq a$ and that $a$, $A$ are algebraically closed, and we may also assume that $a_{i} \cap \mathrm{acl}(a_{< i}) = A$ for all $i < \omega$, as otherwise already $a_{1} \nind_{A}^{K} a_{0}$ (see the proof of Proposition \ref{typeswithalgebraicforkinghavembone}). Now for some $k < \omega$, $a_{ k} \nind^{K}_{A} a_{< k}$, so there is some $a'_{k} \in a_{k} \backslash A$ and $a' \in A'=\mathrm{acl}(a_{<k})$ such that $\models I(a'_{k}, a')$, assuming without loss of generality that $a' \in L(A')$. Then $\{a_{i}\}_{k \leq i < \omega}$ is $A'$-indiscernible, so for $0 \leq i \leq n$ there is some $a'_{k+i} \in a_{k+i} \backslash A$ so that $\models I(a'_{k+i}, a')$. Then $a' \in \mathrm{acl}(a_{k}, \ldots, a_{k+n-1})$, so because $\models I(a'_{k+n}, a')$, $a_{n+k} \nind^{K}_{A} a_{k}, \ldots, a_{k+n-1}$, implying $a_{ n} \nind^{K}_{A} a_{< n}$ by $A$-indiscernibility. This shows $F_{Mb}(p) \leq n$, so $F_{Mb}(p) = n$.

Now assume $\mathrm{acl}(aA)$ is finite, which implies that $p$ is isolated; we observe that $p$ has definable Morley sequences; since $F_{Mb}(p) > 1$, $p$ will also have forking extensions where the forking is not witnessed by algebraic dependence. We may assume that $A, a$ are algebraically closed with $A \subset a$; let $I=\{a_{i}\}$ be an indiscernible sequence of realizations of $p$. Then by the above arguments and the characterization of Kim-forking, $I$ is a Morley sequence if and only if $a_{0} \cap a_{1} = A$ and there are no $a'_{i} \in P(a_{i})$ for all $0 \leq i \leq m$ together with some $a^{*}$ so that $\models I(a'_{i}, a^{*})$ for all $0 \leq i \leq m$, or $a'_{i} \in L(a_{i})$ for all $0 \leq i \leq n$ together with some $a^{*}$ so that $\models I(a^{*}, a'_{i})$ for all $0 \leq i \leq n$. This is definable, by finiteness of $a$.
        \end{example}

  \begin{remark}\label{brower}
In Proposition 2.2.2 of \cite{Brower2012}, Brower identifies an instance of the stable forking conjecture, based on the quasidesigns from \cite{P96}, that generalizes the case where the reason why $a \ind_{C} b$ is that $\mathrm{acl}(aC) \cap \mathrm{acl}(bC) \neq \mathrm{acl}(C)$. Let $B(a, b, C)$ denote that

$a \notin \mathrm{acl}(bC)$, $b \notin \mathrm{acl}(aC)$ and for any nonconstant $aC$-indiscernible sequence $I= \{b_{i}\}_{i < \omega}$ with $b_{0} = b $, $a \in \mathrm{acl}(C)$.

In this case, by Proposition 2.2.2 of \cite{Brower2012} there is a formula $\varphi(x, b) \in \mathrm{tp}(a/Cb)$ which forks over $C$, so that $\varphi(x, y)$ is stable in $\mathcal{L}(C)$.

In Example \ref{genericprojectiveplanes} above, any nontrivial instance of Kim-forking is always given by an instance of $B(a, b, C)$. To express this, we define the following property, which generalizes algebraic forking, Definition \ref{algfork}:

\begin{definition}
    Let $T$ be $\mathrm{NSOP}_{1}$ with existence. Then $p \in S(C)$ has the \emph{Brower property} if whenever $a \models p$ and $a \nind^{K}_{C} b$, either $\mathrm{acl}^{eq}(aC) \cap \mathrm{acl}^{eq}(bC) \neq \mathrm{acl}^{eq}(C)$, or there are some $a' \in \mathrm{acl}^{eq}(aC)$ and $b' \in \mathrm{acl}^{eq}(bC)$ so that $B(a', b', C)$.
\end{definition}

So the types $p$ with definable $\ind^{K}$-Morley sequences discussed in Example \ref{genericprojectiveplanes} have the Brower property.

Note that, similarly to Proposition \ref{algforkingprop}, that there are some $a' \in \mathrm{acl}(aC)$ and $b' \in \mathrm{acl}(bC)$ so that $B(a', b', C)$ implies that $a \nind^{K}_{C} b$, so if $p$ has the Brower property, whenever $a \models p$, $a \nind^{K}_{C} b$ if and only if either $\mathrm{acl}(aC) \cap \mathrm{acl}(bC) \neq \mathrm{acl}(C)$, or there are some $a' \in \mathrm{acl}(aC)$ and $b' \in \mathrm{acl}(bC)$ so that $B(a', b', C)$. Moreover, by Proposition 2.2.2, Corollary 2.3.3 and Theorem 2.1.6 of Brower (\cite{Brower2012}), if $p$ has the Brower property, $b \models p$ and $a \nind^{K}_C b$, then the conclusion of the \textit{stable} Kim-forking conjecture holds (the statements of Brower's results are for forking in simple theories, though the proofs will generalize): there is a formula $\varphi(x, b) \in \mathrm{tp}(a/Cb)$ which Kim-forks over $C$, so that $\varphi(x, y)$ is stable in $\mathcal{L}(C)$. We may then ask whether any instance of the simple Kim-forking conjecture coming from Theorem \ref{definabilitysimplekimforking} actually comes from the Brower property:

\begin{question}
    Let $T$ be $\mathrm{NSOP}_{1}$, and $ p \in S(C)$ have definable $\ind^{K}$-Morley sequences. Does $p$ have the Brower property?
\end{question}
    
Note that in the context of Theorem \ref{definabilitysimplekimforking}, we may also ask this question under the assumption that $T$ is low, or $p$ is isolated.

  \end{remark}

The examples of nontrivial instances of a type with definable $\ind^{K}$-Morley seqeunces in Example \ref{genericprojectiveplanes} are $\mathrm{NSOP}_{1}$, but not simple; we can ask whether simple or stable theories have any nontrivial examples of this property. In a simple theory, we adopt the convention of calling a type with definable $\ind^{K}$-Morley sequences a type with \textit{definable Morley sequences}. 

\begin{question}
    Let $T$ be simple (or stable). Does every type $p \in S(A)$ with definable Morley sequences  have algebraic forking? Does every type $p \in S(A)$ with the Brower property have algebraic forking?
\end{question}

We give a partial result on the first question. To motivate this, we extend the degree of nonminimality to Kim-independence, by replacing every reference to forking with a reference to Kim-forking.

\begin{definition}
    Let $T$ be $\mathrm{NSOP}_{1}$ with existence, and let $p \in S(A)$ be a type admitting a nonalgebraic Kim-forking extension. The \textit{degree of nonminimality} of $p$ is the minimal length $n$ of a sequence of realizations of the type $p$, say $a_1, \ldots , a_n$ such that $p$ has a nonalgebraic Kim-forking extension over $a_1, \ldots , a_n$.
\end{definition}

\begin{example}
    Let $T = K_{m, n}$, and let $p$ be the unique type over $\emptyset$ in, without loss of generality, the point sort. In Example \ref{genericprojectiveplanes} we showed that $p$ (which of course has nonalgebraic Kim-forking extensions) has definable $\ind^{K}$-Morley sequences and that $F_{Mb}(p) = m$; we now show that $\mathrm{nmdeg}(p) = m$, so that $\mathrm{nmdeg}(p) = F_{Mb}(p)$. Let $a_{1}, \ldots a_{m}$ be distinct realizations of $p$; it suffices to show that $a_{1} \ind^{K} a_{2} \ldots a_{m}$. But $a_{1} \nind^{K} a_{2} \ldots a_{m}$ only if $\mathrm{acl}(a_{1}) = a_{1}$ has points incident to lines of $\mathrm{acl}(a_{2} \ldots a_{m})=\{a_{2} \ldots a_{m}\} $, or vice versa, impossible as both sets have only points, and no lines.
\end{example}

The above example motivates the following theorem:

  \begin{theorem}\label{nmdegmbdefinablemorleyproperty}
Let $T$ be $\omega$-stable. Then there is no type $p \in S(A)$ such that $\mathrm{RM}(p)$ is a limit ordinal, $\mathrm{nmdeg}(p)= F_{Mb}(p)$, and $p$ has definable Morley sequences.

\end{theorem}

\begin{proof}
Because a set is independent over $A$ if and only if it is independent over $\mathrm{acl}^{eq}(A)$, and is an indiscernible sequence over $A$ if and only if it is an indiscernible sequence over $\mathrm{acl}^{eq}(A)$, the formula witnessing that $p$ has definable Morley sequences also witnesses, for any extension $p'$ of $p$ to $\mathrm{acl}^{eq}(A)$, that $p'$ has definable Morley sequences, and moreover, $\mathrm{nmdeg}(p') = \mathrm{nmdeg}(p) = F_{Mb}(p) = F_{Mb}(p')$ and $\mathrm{RM}(p) = \mathrm{RM}(p')$ is a limit ordinal. So we may assume $A$ is algebraically closed in $T^{eq}$, and therefore that $p$ is stationary. If $p$ has definable Morley sequences, it follows that the set of Morley sequences over $A$ consisting of realizations of $p$ is relatively definable in the set of \textit{nonconstant} $A$-indiscernible sequences consisting of realizations of $p$. Let $\mathrm{nmdeg}(p)= F_{Mb}(p) = n$, and for $q(y) = (p(x))^{\otimes{n}}$ the type of $n$ independent realizations of $p$ (which is well-defined by stationarity of $p$), let $\sigma$ be the set of formulas $\varphi(x, y)$ where $\varphi(x, a)$ forks over $A$ for any $a \models q(y)$. Since $\mathrm{nmdeg}(p)= F_{Mb}(p) = n$, a nonconstant $A$-indiscernible sequence $\{a_{i}\}_{i < \omega}$ is a Morley sequence if and only if $a_{n} \ind_{A} a_{0} \ldots a_{n-1}$, if and only if $\models \neg\varphi(a_{n}, a_{0} \ldots a_{n-1})$ for every $\varphi(x, y) \in \sigma$. By compactness, and the assumption that $p$ has definable Morley sequences over $a$, there is then some $\varphi(x, y)$ so that, for any $a_{0} \ldots a_{n-1} \models q(y)$, $\varphi(x, a_{0} \ldots a_{n})$ forks over $A$, and so that, for $\{a_{i}\}_{i < \omega}$ a nonconstant $A$-indiscernible sequence, $\{a_{i}\}_{i < \omega}$ is not an $A$-Morley sequence if and only if $\models \varphi(a_{n}, a_{0} \ldots a_{n-1})$. Fix some $b_{0} \ldots b_{n} \models q(y)$, and let $\mathrm{RM}(p) = \lambda$; we have assumed that $\lambda$ is a limit ordinal. Then $\varphi(x, b_{0} \ldots b_{n})$ forks over $A$, so $\lambda' =: \mathrm{RM}(\varphi(x, b_{0} \ldots b_{n}) \cup p(x)) < \lambda$, because otherwise there would be a completion $\tilde{p}(x)\in S(Ab)$ of $\varphi(x, b_{0} \ldots b_{n}) \cup p(x)$ so that $\mathrm{RM}(\tilde{p(x)})=\mathrm{RM}(\varphi(x, b_{0} \ldots b_{n}) \cup p(x)) = \mathrm{RM}(p(x))$, implying that $\tilde{p}(x)$ does not fork over $A$ despite containg $\varphi(x, b_{0} \ldots b_{n})$, a contradiction. Since $\lambda$ is a limit ordinal, there is an ordinal $\lambda''$ so that $\lambda' < \lambda'' < \lambda$. There is then some $q(x) \in S(B)$ extending $p(x)$ so that $\mathrm{RM}(q(x)) = \lambda''$. Let $\{a_{i}\}_{i < \omega}$ be a Morley sequence in $q(x)$ over $B$. Then $\{a_{i}\}_{i < \omega}$ is nonconstant, but is not a Morley sequence over $A$, because if it were, by Kim's lemma $B \ind_{A} a_{0}$, contradicting that $a_{0} \models q(x)$, which forks over $A$. However, $\mathrm{RM}(a_{n}/a_{0} \ldots a_{n-1}) \geq \mathrm{RM}(a_{n}/Ba_{0} \ldots a_{n-1}) = \mathrm{RM}(a_{n}/B) = \mathrm{RM}(q(x))= \lambda''$, the first equality because $\{a_{i}\}_{i < \omega}$ is a Morley sequence over $B$. However, because $\{a_{i}\}_{i < \omega}$ is nonconstant and not a Morley sequence over $A$, $\varphi(x, a_{0} \dots a_{n-1}) \in \mathrm{tp}(a_{n}/a_{0} \ldots a_{n-1})$, so $\mathrm{RM}(\varphi(x, a_{0} \ldots a_{n}) \cup p(x)) \geq \mathrm{RM}(a_{n}/a_{0} \ldots a_{n-1}) \geq \lambda''$. But $a_{0} \ldots a_{n-1} \equiv_{A} b_{0} \ldots b_{n-1}$, so $\lambda' = \mathrm{RM}(\varphi(x, b_{0} \ldots b_{n}) \cup p(x)) = \mathrm{RM}(\varphi(x, a_{0} \ldots a_{n}) \cup p(x)) \geq \lambda'' > \lambda'$, a contradiction.

\end{proof}

\begin{example}
    Since it is open whether there are $\omega$-stable theories with types $p$ so that $1 < F_{Mb}(p) < \infty$, it is unknown whether the hypotheses of this theorem can be satisfied with $\mathrm{nmdeg}(p) > 1$. However, one special case of $\mathrm{nmdeg}(p)= F_{Mb}(p)$ is where $F_{Mb}(p) = 1$. In example \ref{Mbonenonalgebraic}, the unique type $p \in S(\emptyset)$, the $\omega$-stable free pseudoplane from example 7.2.10 of \cite{P96}, is exhibited as a type $p$ without algebraic forking and with $F_{Mb}(p) = 1$; moreover, it has $\mathrm{RM}(p) = \omega$, a limit ordinal. So we have found a type (without algebraic forking) whose failure to have definable Morley sequences is a consequence of Theorem \ref{nmdegmbdefinablemorleyproperty}.
\end{example}
    
We now consider the relationship between definable $\ind^{K}$-Morley sequences and type-counting. In $\omega$-categorical theories, the following converse to Remark \ref{definablemorleysequencesimpliesfinitemb} is immediate:

\begin{prop}\label{omegacategoricaldefinablemorleysequencesisequivalenttofinitemb}
    Let $T$ be an $\omega$-categorical $\mathrm{NSOP}_{1}$ theory with existence, and let $A$ be finite. Then $p \in S(A)$ has definable $\ind^{K}$-Morley sequences if and only if $F_{Mb}(p) < \omega$.
\end{prop}

This leads us to ask if there are any nontrivial examples of this equivalence, or more generally:

\begin{question}
    Let $T$ be an $\omega$-categorical $\mathrm{NSOP}_{1}$ theory, and let $F_{Mb}(p) < \omega$. Does $p$ have algebraic forking (and therefore, $F_{Mb}(p) = 1$)?
\end{question}

However, we do not need all of $\omega$-categoricity for the equivalence in Proposition \ref{omegacategoricaldefinablemorleysequencesisequivalenttofinitemb}, only a weaker claim about type-counting. The following stronger statement follows easily from the definition of definable $\ind^{K}$-Morley sequences:

\begin{prop}
    Let $T$ be an $\mathrm{NSOP}_{1}$ theory with existence, and let $A$ be finite. If $p \in S(A)$, and $F_{Mb}(p)$ and $|S^{F_{Mb}(p) + 1}(A)|$ are finite, then $p$ has definable Morley sequences.
\end{prop}

Unlike the $\omega$-categorical case, we know we have nontrivial examples of this:

\begin{example}
    Let $T= K_{m , 2}$, and let $p$ be the unique type over $\emptyset$ in the point sort. Then $F_{Mb}(p) = m$, and we show that $|S^{F_{Mb}(p) + 1}(\emptyset)|$ is finite, where $S^{F_{Mb}(p) + 1}(\emptyset)$ is taken in the point sort. Let $p_{1}, \ldots, p_{k}$ be points where $k \leq m+1$, and let $l_{1}, \ldots, l_{\ell} \in \mathrm{acl}(p_{1} \ldots p_{m + 1})$ be the at most $m+1$ many lines obtained by finding a line incident to any $m$ points in $p_{1}, \ldots, p_{m + 1}$. Since isomorphic algebraically closed sets satisfy the same type, it suffices to show that $\{p_{1}, \ldots, p_{k}, l_{1} \ldots l_{\ell}\}$ is algebraically closed. This will hold as long as two criteria are satisfied: first, any $m$ of the $p_{1}, \ldots, p_{k}$ are incident to a common line of $l_{1} \ldots l_{\ell}$, which follows by construction, and second, any two of the $l_{1} \ldots l_{\ell}$ are incident to $m-1$ of the $p_{1}, \ldots, p_{k}$. But any individual line in the $l_{1} \ldots l_{\ell}$ is, by construction, incident to $m$ of the $p_{1}, \ldots, p_{k}$, so because $k \leq m+1$, any two lines of the $l_{1} \ldots l_{\ell}$ must be incident to $m-1$ of the $p_{1}, \ldots, p_{k}$.
\end{example}

The nontriviality of this example motivates the following corollary of Theorem \ref{definabilitysimplekimforking} (note that isolatedness of $p$ follows from the hypotheses):

\begin{cor}
    \label{typecounting}

    Let $T$ be $\mathrm{NSOP}_{1}$ with existence, and let $A$ be finite. Let $p \in S(C)$, and let $F_{Mb}(p)$ and $|S^{F_{Mb}(p) + 1}(C)|$ be finite. Then there is some $n < \omega$ so that, for any $a = \{a_i\}_{i \in I}$ where $|I| \geq n$ (with the $a_{i}$ allowed to repeat), and such that $a \nind^{K}_{C} b$, there is some formula $\varphi(x, b) \in \mathrm{tp}(a/Cb)$ that Kim-forks over $C$, where $\varphi(x, y)$ is a formula with parameters in $C$ which is simple as a formula of $L(C)$. 
\end{cor}
\appendix

\section{$\mathrm{EM}$-types in strictly simple theories}

Theorem \ref{omegacategoricalsimplebutnotsupersimple} produced infinitely many $\mathrm{EM}$-types in a theory that is simple but not supersimple, but only under the additional assumption of $\omega$-categoricity. Here, as mentioned in \ref{omegacategoricitynotnecessary} we remove this assumption.

\begin{theorem}
    Let $T$ be simple, and let $\mathrm{SU}(\mathrm{tp}(b/\emptyset))=\infty$. Then there are ($\emptyset$-)indiscernible sequences $I^{i}=(b^{i}_{j})_{j < \omega}$, for $i < \omega$, with $b^{i}_{0} = b$, so that $I^{i} \not\equiv I^{j}$ for $i \neq j < \omega$.
\end{theorem}

\begin{proof}
    As in Definition \ref{resolvable}, we define the following invariants of indiscernible sequences.

    \begin{definition}
        \label{qresolvable}

        (1) Let $r(x) \in S(C)$ be a complete type, and let $I=\{a_{i}\}_{i < \omega}$ be an indiscernible sequence over $\emptyset$ so that $|a_{i}|= x$. Then $I$ is \textit{r}-resolvable if there is some automorphism $\sigma$ of $\mathbb{M}$, so that for $A = \sigma(C)$, $I$ is a Morley sequence over $A$ of realizations of $\sigma(r)$.  

        (2) Let $p(x) \in S(B)$ be a type. Then $p$ is \emph{$r$-resolvable} if there is some automorphism $\sigma$ of $\mathbb{M}$ so that, for $B = \sigma(C)$, there is some  $ q \in S(AB)$, $q \supseteq p, \sigma(r)$ so that $q$ forks over neither $A$ nor $B$.
    \end{definition}

Then similarly to Proposition \ref{resolvabilitygrank}, the following holds:

\begin{prop}\label{resolvabilitygrank2}
    Let $T$ be simple, and let $I$ be an indiscernible sequence. Then the following are equivalent:

    (1) The indiscernible sequence $I$ is $r$-resolvable.

    (2) For some type $p \in S(B)$, so that $I$ is a Morley sequence in $p$ over $B$, $p$ is $r$-resolvable.

    (2') For every type $p \in S(B)$, so that $I$ is a Morley sequence in $p$ over $B$, $p$ is $r$-resolvable.

    (2'') The type $\mathrm{lim}^{+}(I/I) \in S(I)$ is $r$-resolvable.

\end{prop}

Let $p_{0}(x) = \mathrm{tp}(b)$. It suffices to show that there do not exist finitely many $r_{1}, \ldots, r_{n}$ so that every extension of $p_{0}(x)$ is $r_{i}$-resolvable for some $1 \leq i \leq n$: otherwise, since a type $p(x)$ is always $p$-resolvable, we can find types $p_{i}(x) \in S(A_{i})$ extending $p_{0}$, for $i < \omega$, so that for $i \neq j$, $p_{i}(x)$ is not resolvable with respect to all of the same types as $p_{j}(x)$. By the previous proposition, for $I_{i}$ a Morley sequence over $A_{i}$ consisting of realizations of $p_{i}$, $I_{i}$ is resolvable with respect to exactly the same types as $p_{i}$, so for $i \neq j$, $I_{i} \not \equiv I_{j}$ because resolvability with respect to any given type is an invariant property of an indiscernible sequence.

Suppose otherwise, so such $r_{i} \in S(C_{i})$ do exist for $1 \leq i \leq n$. Because $\mathrm{tp}(b) $ has $\mathrm{SU}$-rank $\infty$, we may find $a$, $B$ as Lemma \ref{forkingchains}, which does not assume $\omega$-categoricity, and additionally so that $\mathrm{tp}(a/\emptyset)=p_{0}$. So $\mathrm{tp}(a/B)$ is ranked but for any finite $B_{0} \subset B$, $\mathrm{SU}(\mathrm{tp}(a/B_{0}))= \infty$. Without loss of generality, we assume that for some $1 \leq m \leq n$ and all $1 \leq i \leq m$, there is some finite $B_{0} \subset B$ so that $\mathrm{tp}(a/B_{0})$ is $r_{i}$-resolvable, and moreover that for any finite $B_{0} \subset B$, $\mathrm{tp}(a/B_{0})$ is $r_{i}$-resolvable for some $1 \leq i \leq m$.   So for $1 \leq i \leq m$, $\mathrm{SU}(r_{i})= \infty$.

Let $Y$ be a set of variables large enough to enumerate all of the $C_{i}$. For $c \models r_{i} \in S(C_{i})$, let $s_{i}(Y, x)= \mathrm{tp}( C_{i}, c/\emptyset)$; for $|Z| = |B|$, let $\tilde{\Phi}_{i}(Z, C_{i}, c) $ denote the set of formulas $\varphi(Z, C_{i}, c)$ that do not fork over  $C_{i}$. Let $\Phi_{i}(Z, Y, x)$ denote the set of conjunctions of formulas from $s_{i}(Y, x) \cup \tilde{\Phi}(Z, Y, x)$.  We show that the set

$$\{\vee^{m}_{i=1} \varphi_{i}(B, Y, a): \varphi_{i} \in \Phi_{i}(B, Y, a)\}$$

is consistent. This will be sufficient, because by compactness, one of the $\Phi_{i}(B, Y ,a)$ will then be consistent, which means that there will be some automorphism $\sigma$ of $\mathbb{M}$ so that, for $A = \sigma(C_{i})$, $\mathrm{tp}(a/A)=\sigma(r_{i})$, and $B \ind_{A} a$, equivalently $a \ind_A B$ by symmetry.  But then $\mathrm{SU}(a/A)= \infty$, so $\mathrm{SU}(a/AB) = \infty$, so $\mathrm{SU}(a/B)=\infty$, a contradiction.  By compactness, to show that the above set is consistent, it suffices to show, for any $B_{0} \subset B$ finite, that the part just mentioning $B_{0}$ and $a$ is consistent, or equivalently, that the part of $\Phi_{i}(B, Y ,a)$ mentioning just $B_{0}$ and $a$ is consistent for some $1 \leq i \leq m$.  But $\mathrm{tp}(a/B_{0})$ is $r_{i}$-resolvable for some $1 \leq i \leq m$, so there is some automorphism $\sigma$ of $\mathbb{M}$ so that, for $A = \sigma(C_{i})$, $\mathrm{tp}(a/A)=\sigma(r_{i})$, and $a \ind_{A} B_{0}$, equivalently $B_{0} \ind_A a$ by symmetry; as above, this is what consistency of that part of $\Phi_{i}(B, Y ,a)$ means.

\end{proof}

\bibliographystyle{plain}
\bibliography{refs}

\end{document}